\documentclass[12pt]{report}

\usepackage{tamuconfig}
\usepackage[T1]{fontenc}
\usepackage{graphicx}
\DeclareGraphicsExtensions{.png, .jpg}
\usepackage[hidelinks]{hyperref}
\usepackage{comment}
\usepackage{amssymb,amsmath}
\usepackage{tikz}
\usepackage{mathtools}
\usepackage{tikz-cd}
\usepackage{bbm}
\newtheorem{theorem}{Theorem}     
\numberwithin{theorem}{chapter}
\newtheorem{lemma}[theorem]{Lemma}  
\newtheorem{notation}[theorem]{Notation}        
\newtheorem{corollary}[theorem]{Corollary}     
\newtheorem{proposition}[theorem]{Proposition}     

\theoremstyle{definition}  
\newtheorem{definition}[theorem]{Definition}     
\newtheorem{example}[theorem]{Example}

\newcommand{\CC}{{\mathbb C}}

\newcommand{\NN}{{\mathbb N}}

\newcommand{\RR}{{\mathbb R}}
\newcommand{\TT}{{\mathbb T}}
\newcommand{\ZZ}{{\mathbb Z}}
\newcommand{\PP}{{\mathbb P}}

\newcommand{\calA}{{\mathcal A}}
\newcommand{\calB}{{\mathcal B}}
\newcommand{\calF}{{\mathcal F}}
\newcommand{\calG}{{\mathcal G}}

\newcommand{\calM}{{\mathcal M}}
\newcommand{\calN}{{\mathcal N}}
\newcommand{\calO}{{\mathcal O}}
\newcommand{\calP}{{\mathcal P}}
\newcommand{\calS}{{\mathcal S}}
\newcommand{\calT}{{\mathcal T}}
\newcommand{\bbI}{{\mathbbm 1}}

\DeclareMathOperator{\conv}{conv}
\DeclareMathOperator{\cone}{cone}
\DeclareMathOperator{\Spec}{Spec}
\DeclareMathOperator{\Hom}{Hom}
\DeclareMathOperator{\Log}{Log}
\DeclareMathOperator{\mon}{mon}
\DeclareMathOperator{\im}{Im}
\DeclareMathOperator{\supp}{supp}

\DeclareMathOperator{\Relint}{Relint}

\newcommand\restr[2]{\ensuremath{\left.#1\right|_{#2}}}
\newcommand{\simplex}{\includegraphics[scale=0.8]{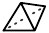}}

\definecolor{navyBlue}{cmyk}{1,1,0,0.2}
\newcommand{\defcolor}[1]{{\color{blue} #1}}
\newcommand{\demph}[1]{{\it\defcolor{#1}}}

\begin{document}
\renewcommand{\tamumanuscripttitle}{Irrational Toric Varieties}
\renewcommand{\tamupapertype}{Dissertation}
\renewcommand{\tamufullname}{Ata Firat Pir}
\renewcommand{\tamudegree}{DOCTOR OF PHILOSOPHY}
\renewcommand{\tamuchairone}{Frank Sottile}
\renewcommand{\tamumemberone}{J. Maurice Rojas}
\newcommand{\tamumembertwo}{Laura Matusevich}
\newcommand{\tamumemberthree}{John Keyser}
\renewcommand{\tamudepthead}{Emil Straube}
\renewcommand{\tamugradmonth}{December}
\renewcommand{\tamugradyear}{2018}
\renewcommand{\tamudepartment}{Mathematics}

\providecommand{\tabularnewline}{\\}

\begin{titlepage}
	\begin{center}
		\MakeUppercase{\tamumanuscripttitle}
		\vspace{4em}
		
		A \tamupapertype
		
		by
		
		\MakeUppercase{\tamufullname}
		
		\vspace{4em}
		
		\begin{singlespace}
			
			Submitted to the Office of Graduate and Professional Studies of \\
			Texas A\&M University \\
			
			in partial fulfillment of the requirements for the degree of \\
		\end{singlespace}
		
		\MakeUppercase{\tamudegree}
		\par\end{center}
	\vspace{2em}
	\begin{singlespace}
		\begin{tabular}{ll}
			& \tabularnewline
			& \cr
			Chair of Committee, & \tamuchairone\tabularnewline
			Committee Members, & \tamumemberone\tabularnewline
			& \tamumembertwo\tabularnewline
			& \tamumemberthree\tabularnewline
			Head of Department, & \tamudepthead\tabularnewline
			
		\end{tabular}
	\end{singlespace}
	\vspace{3em}
	
	\begin{center}
		December \hspace{2pt} \tamugradyear
		
		\vspace{3em}
		
		Major Subject: \tamudepartment \par
		\vspace{3em}
		Copyright \tamugradyear \hspace{.5em}\tamufullname 
		\par\end{center}
\end{titlepage}
\pagebreak{}

\chapter*{ABSTRACT}
\addcontentsline{toc}{chapter}{ABSTRACT} 

\pagestyle{plain} 
\pagenumbering{roman} 
\setcounter{page}{2}

\indent Classical toric varieties are among the simplest objects in algebraic geometry. They arise in an elementary fashion as varieties parametrized by monomials whose exponents are a finite subset $\calA$ of $\ZZ^n$. They may also be constructed from a rational fan $\Sigma$ in $\RR^n$. The combinatorics of the set $\calA$ or fan $\Sigma$ control the geometry of the associated toric variety. These toric varieties have an action of an algebraic torus with a dense orbit. Applications of algebraic geometry in geometric modeling and algebraic statistics have long studied the nonnegative real part of a toric variety as the main object, where the set $\calA$ may be an arbitrary set in $\RR^n$. These are called irrational affine toric varieties. This theory has been limited by the lack of a construction of an irrational toric variety from an arbitrary fan in $\RR^n$. 

We construct a theory of irrational toric varieties associated to arbitrary fans. These are $(\RR_>)^n$-equivariant cell complexes dual to the fan. Such an irrational toric variety is projective (may be embedded in a simplex) if and only if its fan is the normal fan of a polytope, and in that case, the toric variety is homeomorphic to that polytope. We use irrational toric varieties to show that the space of Hausdorff limits of translates an irrational toric variety associated to a finite subset $\calA$ of $\RR^n$ is homeomorphic to the secondary polytope of $\calA$.

\pagebreak{}

\chapter*{ACKNOWLEDGMENTS}
\addcontentsline{toc}{chapter}{ACKNOWLEDGMENTS}  

\indent 
I want to express my sincere gratitude to my advisor Frank Sottile for the continuous support of research, for his patience, motivation, and immense knowledge. His guidance helped me in all the time of research and writing of this thesis. I have been extremely lucky to have an advisor who cared so much about my work, and who responded to my questions and queries so promptly. The example you have set will help me throughout my life.

I would like to thank my family for all their love and support in every stage of my life. Words can not express how grateful I am to my father, mother, sister and brother for all of the sacrifices that they have made on my behalf. Their prayers for me was what sustained me thus far.

\pagebreak{}

\chapter*{CONTRIBUTORS AND FUNDING SOURCES}
\addcontentsline{toc}{chapter}{CONTRIBUTORS AND FUNDING SOURCES}  

\subsection*{Contributors}
This work was supported by a dissertation committee consisting of Professors Frank Sottile, Laura Matusevich and Maurice Rojas of the Department of Mathematics and Professor John Keyser of the Department of Computer Science and Engineering.

All other work conducted for the dissertation was completed by the student independently.
\subsection*{Funding Sources}
Graduate study was supported by a fellowship from Texas A\&M University and grants DMS-1001615 and DMS-1501370  from the National Science Foundation.
\pagebreak{}

\phantomsection
\addcontentsline{toc}{chapter}{TABLE OF CONTENTS}  

\begin{singlespace}
	\renewcommand\contentsname{\normalfont} {\centerline{TABLE OF CONTENTS}}
	
	\setcounter{tocdepth}{4} 

	\setlength{\cftaftertoctitleskip}{1em}
	\renewcommand{\cftaftertoctitle}{%
		\hfill{\normalfont {Page}\par}}

	\tableofcontents
	
	\addtocontents{toc}{\protect\afterpage{~\hfill\normalfont{Page}\par\medskip \medskip \medskip}}
\end{singlespace}

\pagebreak{}


\phantomsection
\addcontentsline{toc}{chapter}{LIST OF FIGURES}  

\renewcommand{\cftloftitlefont}{\center\normalfont\MakeUppercase}

\setlength{\cftbeforeloftitleskip}{-12pt} 
\renewcommand{\cftafterloftitleskip}{12pt}

\renewcommand{\cftafterloftitle}{%
	\\[4em]\mbox{}\hspace{2pt}FIGURE\hfill{\normalfont Page}\vskip\baselineskip}

\begingroup

\begin{center}
	\begin{singlespace}
		\setlength{\cftbeforechapskip}{0.2cm}
		\setlength{\cftbeforesecskip}{0.15cm}
		\setlength{\cftbeforesubsecskip}{0.15cm}
		\setlength{\cftbeforefigskip}{0.2cm}
		\setlength{\cftbeforetabskip}{0.2cm}
		
		
		
		\listoffigures
		
	\end{singlespace}
\end{center}

\pagebreak{}

\pagestyle{plain} 
\pagenumbering{arabic} 
\setcounter{page}{1}

\chapter{\uppercase {Introduction}}

Toric varieties form an important class of algebraic varieties that are among the simplest objects in algebraic geometry, which provide ``a remarkably fertile testing ground for general theories'' \cite{Fulton}.
The rich interaction between algebraic geometry and combinatorics has been important in both areas \cite{Demazure,KKMS,Stanley,Brion}. 

Demazure defined toric varieties in 1970 \cite{Demazure}. At that time a toric variety was referred as ``the scheme defined by the fan $\Sigma$'' \cite{Demazure}, ``torus embeddings'' \cite{KKMS}, and ``almost homogeneous algebraic variety under torus action'' \cite{MiyakeOda}. A toric variety is commonly defined to be a normal variety $Y$ containing a dense torus $\TT$ with the action of the torus $\TT$ on itself extending naturally to an action of $\TT$ on $Y$ \cite{Fulton}. These normal toric varieties enjoy a functorial construction using rational fans $\Sigma$. A (not necessarily normal) affine toric variety $Y_\calA$ is a variety parametrized by monomials having exponents in $\calA \subset \ZZ^n$ \cite{GKZ,BerndConvexPolytopes}. 

Toric varieties have found applications in many other fields. In \cite{CoxKatz}, the connections between the toric geometry and mirror symmetry are explored. The role of toric surfaces as generalizations of B\'ezier patches in geometric modeling is described in \cite{Krasauskas}. The nonnegative real part of a toric variety is relevant for geometric modeling \cite{vilnius} and toric B\'ezier patches are defined naturally when $\calA$ is any finite subset of $\RR^n$. 

Toric varieties in disguise were studied in algebraic statistics long before they appeared in algebraic geometry. Log-linear models~\cite{Goodman} are the nonnegative real part of a toric variety $Y_\calA$. These are more flexible than toric varieties in algebraic geometry, as the set $\calA$ can be any finite subset of $\RR^n$. In 1963, Birch \cite{Birch} showed that the nonnegative part of a toric variety $Y_\calA$ is homeomorphic to the cone generated by $\calA$. 

These applications of toric varieties in geometric modeling and algebraic statistics led to the study of the nonnegative real part of a toric variety $Y_\calA$ for a finite set $\calA  \subset \ZZ^n$ and then to relax the condition on $\calA$ so that it may be any finite subset of $\RR^n$. These resulting objects are called \demph{irrational toric varieties}. In \cite{PSV}, Hausdorff limits of translates of irrational toric varieties were studied. When $\calA \subset \ZZ^n$, the space of Hausdorff limits is homeomorphic to the secondary polytope of $\calA$ \cite{GSZ}. When the exponents $\calA \subset \RR^n$ are not integral, all Hausdorff limits were identified in \cite{PSV} as toric degenerations and thus were related to the secondary fan of $\calA$, but the authors were unable to identify the set of Hausdorff limits with the secondary polytope. This obstruction was due to deficiencies in the theory of irrational toric varieties, in particular a construction using arbitrary fans and a relation to polytopes. 

We develop a theory of irrational toric varieties constructed from arbitrary fans in $\RR^n$. Such an irrational toric variety is an $(\RR_>)^n$-equivariant cell complex with a dense orbit whose cells are dual to the cones of the associated fan. Their construction is functorial; maps of fans correspond to equivariant maps of irrational toric varieties and the fan may be recovered from the toric variety. We also complete the work in \cite{PSV}, identifying the space of Hausdorff limits with the secondary polytope. 

This dissertation is organized as follows. In Section \ref{CH:Background}, we provide basic terminology and results on geometric combinatorics. In Section \ref{CH:ClassicalTV}, we recall various constructions of classical toric varieties. We study some properties of classical toric varieties in Section \ref{CH:PropertiesClassicalTV}. We construct affine irrational toric varieties and abstract irrational toric varieties, and establish some of their main properties in Section \ref{CH:ITV}. We study some global properties of irrational toric varieties in Section \ref{CH:PropertiesITV}. In Section \ref{CH:Hausdorff}, we identify the space of Hausdorff limits with the secondary polytope. We give a summary of the thesis and discuss future directions in Section \ref{CH:Conclusion}.

\pagebreak{}

\chapter{SOME GEOMETRIC COMBINATORICS}\label{CH:Background}

This section will develop a background on geometric combinatorics which we will be using to construct toric varieties. We refer the reader to \cite{Fulton, CLS, Oda, Ziegler, Ewald} for a more complete background.

\section{Polyhedral Cones}\label{S:Cones}

Let $N_\ZZ$ be a free abelian group of rank $n$ and let $M_\ZZ:= \Hom (N_\ZZ, \ZZ)$ be its dual group. We have a canonical $\ZZ$-bilinear pairing 
\[
\langle \ ,\ \rangle \colon M_\ZZ \times N_\ZZ \rightarrow \ZZ,
\]
where $\langle x,y\rangle := x(y)$.

By scalar extension to the field $\RR$ we have $n$-dimensional real vector spaces $N:= N_\ZZ \otimes \RR$ and $M:=M_\ZZ \otimes \RR$ with a canonical $\RR$-bilinear pairing 
\[
\langle \ ,\ \rangle \colon M \times N\rightarrow \RR.
\]
When $M=\RR^n$ and $N = \RR^n$, the pairing $\langle\ , \ \rangle$ is the usual inner product in $\RR^n$. 

Let $\RR_>$ be the positive real numbers, $\RR_\geq$ be the nonnegative real numbers and $\RR_\leq$ be the nonpositive real numbers.

\begin{definition}
	A subset $\sigma$ of $N$ is called a \demph{polyhedral cone} if there exists a finite set $S= \{v_1, \ldots, v_r\}$ in $N$ such that 
	\begin{equation*}
	\sigma = \RR_\geq v_1 + \cdots + \RR_\geq v_r.
	\end{equation*}
	
	The cone $\sigma$ is said to be generated by $S$ and denoted by $\sigma = \cone (S)$. The \demph{dimension} $\dim \sigma$ of $\sigma$ is the dimension of the linear subspace $\langle \sigma \rangle: = \sigma + (-\sigma)$. A polyhedral cone $\sigma = \cone (S)$ is called \demph{rational} if $S \subseteq N_\ZZ$. 
	
\end{definition}

Note that a polyhedral cone $\sigma$ is convex. That is, when $x, y \in \sigma$ then $\lambda x + (1- \lambda) y  \in \sigma$ for all $ \lambda \in \left[ 0,1\right]$. It is also a cone, i.e, if $x \in \sigma$ then $\lambda x \in \sigma$ for all $\lambda \in \RR_\geq$.

\begin{example}\label{ConesInR2}
	Let $N=\RR^2$ with canonical basis $\{e_1, e_2\}$. Consider the three cones $\sigma_1 = \cone\{e_1\}, \ \sigma_2 = \cone\{e_1,e_2\}, \ \text{and} \ \sigma_3 = \cone\{2e_1-e_2, e_2\}$ in $\RR^2$. These are shown in Figure \ref{F:ConesR2}.
	
	\begin{figure}[!ht]
		\centering
		\begin{tikzpicture}[line cap=round,line join=round,>=latex,scale=.90]
		\filldraw[very thick] (-1,0) circle (.04cm) node [below] {$0$};
		\filldraw[very thick] (.5,0) circle (.04cm) node [below] {$e_1$};
		\draw[->] (-1,0) -- (1.25,0); 
		\node at (0,10pt)  {$\sigma_1$};

		\fill[line width=0.pt,color=green,fill=green,fill opacity=0.4] (4,0) -- (6,0) -- (6,2) -- (4,2) -- cycle;
		\filldraw[very thick] (4,0) circle (.04cm)node [below] {$0$};
		\filldraw[very thick] (5.5,0) circle (.04cm)node [below] {$e_1$};
		\filldraw[very thick] (4,1.5) circle (.04cm)node [left] {$e_2$};
		
		\draw[->] (4,0) -- (6.25,0); 
		\draw[->] (4,0) -- (4,2.25); 
		\node at (5,1)  {$\sigma_2$};
		
		draw[->] (0,0) -- (2,0); 
		draw[->] (0,0) -- (2,0); 
		draw[->] (0,0) -- (2,0); 
		
		\fill[line width=0.pt,color=green,fill=yellow,fill opacity=0.4] (9,0) -- (13,-2)  -- (9,2) -- cycle;
		[line cap=round,line join=round,>=latex,scale=1]
		\filldraw[very thick] (9,0) circle (.04cm)node [below] {$0$};
		\filldraw[very thick] (12,-1.5) circle (.04cm) node [below] {$2e_1-e_2 \quad \quad $};
		\filldraw[very thick] (9,1.5) circle (.04cm)node [left] {$e_2$};
		\draw[->] (9,0) -- (13.1,-2.05); 
		\draw[->] (9,0) -- (9,2.25); 
		\node at (10,0.3)  {$\sigma_3$};
		
		\end{tikzpicture}
		
		\caption{Examples of cones in $\RR^2$.}
		\label{F:ConesR2}
	\end{figure}
	\noindent Note that $\sigma_1$ has dimension 1, whereas both $\sigma_2$ and $\sigma_3$ have dimension 2.  \hfill$\square$
\end{example}

\begin{definition}
	Given a cone $\sigma \subset N$, \demph{dual} of $\sigma$ in $M$ is
	\[
	\sigma^\vee := \{ u \in M \mid \langle u,v \rangle \geq 0 \ \text{for all } v \in \sigma \}.
	\]
\end{definition}

\begin{example} \label{dualsInR2}
	Let $\{e_1^*,e_2^*\}$ denote the dual basis of $(\RR^2)^*$. The duals of the cones given in Example \ref{ConesInR2} are given by $\sigma_1^\vee = \cone\{e_1^*, e_2^*, -e_2^*\}$, $\sigma_2^\vee = \cone\{e_1^*, e_2^*\}$ and $\sigma_3^\vee=\cone\{e_2^*, e_1^*+2e_2^*\}$. These are depicted in Figure \ref{F:DualConesR2}. In particular the duals are all polyhedral cones in $M$. We will prove in Corollary \ref{DualIsaCone} that this is true for all polyhedral cones.  \hfill$\square$
	
	\begin{figure}[!ht]
		\centering
		\begin{tikzpicture}[line cap=round,line join=round,>=latex,scale=.75]
		\fill[line width=0.pt,color=blue,fill=blue,fill opacity=0.3] (-1,2) -- (1,2)  -- (1,-2)--(-1,-2) -- cycle;
		\filldraw[very thick] (-1,0) circle (.04cm)node [left] {$0$};
		\draw[<->] (-1,2.25) -- (-1,-2.25); 
		\draw[->] (-1,0) -- (1.25,0); 
		\filldraw[very thick] (-1,1.5) circle (.04cm)node [left] {$e_2^*$};
		\filldraw[very thick] (-1,-1.5) circle (.04cm)node [left] {$-e_2^*$};
		\filldraw[very thick] (.5,0) circle (.04cm)node [below] {$e_1^*$};
		\node at (0,1)  {$\sigma_1^\vee$};

		\fill[line width=0.pt,color=green,fill=green,fill opacity=0.4] (4,0) -- (6,0) -- (6,2) -- (4,2) -- cycle;
		\filldraw[very thick] (4,0) circle (.04cm);
		
		\draw[->] (4,0) -- (6.25,0); 
		\draw[->] (4,0) -- (4,2.25); 
		\filldraw[very thick] (4,1.5) circle (.04cm)node [left] {$e_2^*$};
		\filldraw[very thick] (4,0) circle (.04cm)node [left] {$0$};
		\filldraw[very thick] (5.5,0) circle (.04cm)node [below] {$e_1^*$};
		\node at (5,1)  {$\sigma_2^\vee$};
		
		draw[->] (0,0) -- (2,0); 
		draw[->] (0,0) -- (2,0); 
		draw[->] (0,0) -- (2,0); 
		
		\fill[line width=0.pt,color=green,fill=yellow,fill opacity=0.4] (9,0) -- (11,0)  -- (11,4) -- cycle;
		[line cap=round,line join=round,>=latex,scale=1]
		\filldraw[very thick] (9,0) circle (.04cm) node [left]{$0$};
		\filldraw[very thick] (10.5,3) circle (.04cm)node [left] {$e_1^*+2e_2^*$};
		\filldraw[very thick] (10.5,0) circle (.04cm)node [below] {$e_1^*$};
		
		\draw[->] (9,0) -- (11.25,0); 
		\draw[->] (9,0) -- (11.15,4.3); 
		\node at (10.5,1)  {$\sigma_3^\vee$};
		\end{tikzpicture}
		
		\caption{Examples of dual cones.}
		\label{F:DualConesR2}
	\end{figure}
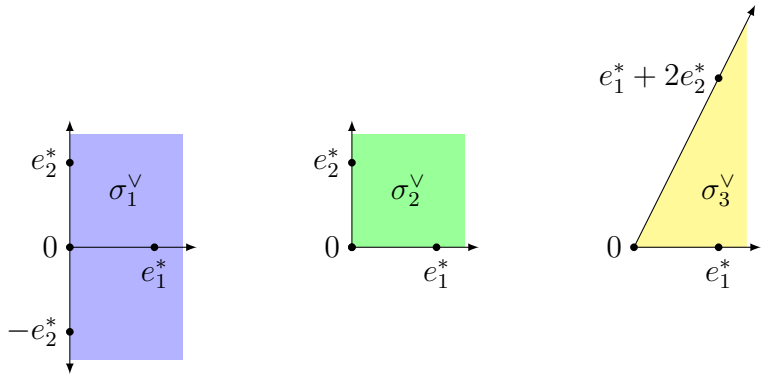
	
\end{example}

The following theorem from the theory of convex polytopes is useful to prove properties of polyhedral cones. Its proof can be found in \cite[Section 2.2, Theorem 1]{Grunbaum}.

\begin{theorem}\label{Seperation}
	Let $C$ be a nonempty closed convex set in $N$, and let $x \in N \setminus C$ be a point. Then there exists a hyperplane $H$ of $N$ which strictly separates $x$ and $C$, i.e., $x$ is in one of the open halfspaces bounded by $H$, and $C$ is in the other open halfspace bounded by $H$. 
\end{theorem}

This theorem has a useful corollary which we will use to prove that the dual of the dual cone is itself.

\begin{corollary}\label{C:Seperation}
	Let $\sigma \subset N$ be a polyhedral cone and $x \in N \setminus \sigma$. Then there exist an element $u \in \sigma^\vee$ such that $\langle u, x \rangle < 0$.
\end{corollary}


\begin{corollary}
	Let $\sigma\subset N$ be a cone. Then $(\sigma^\vee)^\vee = \sigma$.
\end{corollary}

\begin{proof}
	By definition, for every element $v \in \sigma$ and $u \in \sigma^\vee$, we have $\langle u,v \rangle \geq 0$. Hence $\sigma \subset (\sigma^\vee)^\vee$. For the reverse inclusion we will use Corollary \ref{C:Seperation}. Assume that $x \in (\sigma^\vee)^\vee \setminus \sigma$. Then there exists an element $u\in \sigma^\vee $ such that $\langle u, v \rangle < 0$, therefore  $u \notin (\sigma^\vee)^\vee$,  which is a contradiction. Hence we get the reverse inclusion. 
\end{proof}

Given $m \in M$ define
\begin{align*}
H_m &:= \{ v \in N \mid \langle m,v \rangle = 0\}=m^\perp,\\
H_m^+ &:= \{ v \in N \mid \langle m,v \rangle \geq 0\}.
\end{align*}

Note that when $m \neq 0$, $H_m$ is a hyperplane in $N$, and $H_m^+$ is a closed halfspace in $N$. The hyperplane $H_m$ is a supporting hyperplane of a cone $\sigma$ if $\sigma \subset H_m^+$, and when this happens $H_m^+$ is a supporting halfspace.

The following theorem gives an alternate description of polyhedral cones in terms of supporting hyperplanes. For a detailed proof we refer to \cite[Theorem 1.3]{Ziegler}.

\begin{theorem} \label{FiniteIntofHalfspaces}
	A cone $\sigma$ is polyhedral if and only if it can be written as an intersection of finitely many closed halfspaces.
\end{theorem}

Next, we show that the dual of a polyhedral cone $\sigma$ is again a polyhedral cone.  First we need a lemma.

\begin{lemma}
	Let $S = \{v_1, \ldots, v_r\} \subset N$ and $\sigma = \cone (S)$ be a polyhedral cone in $N$. Then, $\sigma^\vee = \{ u \in M \mid \langle u,v \rangle \geq 0 \text{ for all } v \in S \}$.
\end{lemma}

\begin{proof} 
	Let $C = \{ u \in M \mid \langle u,v \rangle \geq 0 \text{ for all } v \in S \}$. Every element in $S$ is also in $\sigma$. Hence by definition, for any element $m \in \sigma^\vee$, we must have $\langle m,v \rangle \geq 0$ for all $v \in S$.  Hence $\sigma^\vee \subset C$. 
	
	Next assume $u \in C$. For any element $v \in \sigma$, we can write  it as $v = \sum_{i=1}^r \lambda_i v_i$ for some $\lambda_1, \ldots, \lambda_r \in \RR_\geq$. Since the pairing is linear, we have
	$\langle u,v \rangle = \sum_{i=1}^r \lambda_i \langle u, v_i \rangle$, which is nonnegative as each term in the summand is nonnegative. Hence $u\in \sigma^\vee$.  This completes the proof. 
\end{proof}

In the previous lemma, since each $v_i \in S$ gives a closed halfspace $H_{v_i}^+$ in $M$, we can deduce that $\sigma^\vee = \bigcap_{v_i\in S} H_{v_i}^+$. Combining this with Theorem \ref{FiniteIntofHalfspaces} gives us the following corollary.

\begin{corollary}\label{DualIsaCone}
	Let $\sigma$ be a polyhedral cone in $N$. Then $\sigma^\vee$ is a polyhedral cone in $M$. 
\end{corollary}

\section{Faces}
We can use supporting hyperplanes and supporting halfspaces to define faces of a cone. These are again polyhedral cones and play an important role when constructing abstract toric varieties associated to fans. 

\begin{definition} 
	Given a polyhedral cone $\sigma \subset N$, a \demph{face} of $\sigma$ is a set of the form $\tau = H_m \cap \sigma$ for some $m \in \sigma^\vee$. A face $\tau \neq \sigma$ of $\sigma$ is called a \demph{proper} face. A face $\tau$ of $\sigma$ is denoted by $\tau \preceq \sigma$, whereas we write $\tau \prec \sigma$ when $\tau$ is a proper face. 
	
	A facet $\tau$ of $\sigma$ is a face of $\sigma$ with codimension $1$, i.e., $\dim (\tau)= \dim(\sigma)-1$. 
\end{definition}

\begin{example}
	Consider the cone $\sigma_3$ in Example \ref{ConesInR2}. We illustrate its proper faces  and hyperplanes that give these faces in Figure \ref{F:FacesofCone}. First consider $k= e_1^*+e_2^* \in \sigma_3^\vee$.  Then the hyperplane corresponding to $k$ is given by 
	\[
	H_k = \{v \in N \mid \langle k, v  \rangle \geq 0 \} = \{(a,b) \in N \mid a+b = 0 \}. \] Intersecting $H_k$ with $\sigma_3$ will give the origin $\{0\}$ as a face of $\sigma_3$. 
	
	Next consider $l = e_1^*$. Then the hyperplane corresponding to $l$ is given by $H_l = \{(a,b) \in N \mid a= 0 \}$. Intersecting $H_l$ with $\sigma_3$ yields the face $\tau_1 = \cone\{e_2\}$. 
	
	Taking $m= 2e_1^* + e_2^*$ yields the hyperplane $H_m = \{(a,b) \in N \mid 2a+b= 0 \}$. Intersecting $H_m$ with $\sigma_3$ gives the proper face $\tau_2 = \cone \{2e_1-e_2\}$. 
	
	In this example, $\tau_1$ and $\tau_2$ are 1-dimensional faces of $\sigma_3$.  Since $\dim \sigma_3 =2$, they are also facets of $\sigma_3$.  \hfill$\square$
	
	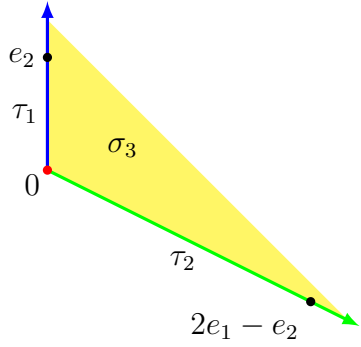
\begin{figure}[!ht]
		\centering
		\begin{tikzpicture}[line cap=round,line join=round,>=latex,scale=1]
		
		draw[->] (0,0) -- (2,0); 
		\fill[line width=0.pt,color=green,fill=yellow,fill opacity=0.6] (0,0) -- (4,-2)  -- (0,2) -- cycle;
		[line cap=round,line join=round,>=latex,scale=1]
		
		\node at (1,0.3)  {$\sigma_3$};
		
		\draw[->,blue,very thick] (0,0) -- (0,2.25); 
		\node at (-.3,.75) {$\tau_1$};
		
		\draw[->,green,very thick] (0,0) -- (4.15,-2.075); 
		\node at (1.8,-1.2) {$\tau_2$};
		
		\filldraw[very thick,red] (0,0) circle (.04cm);
		\node at (-.2,-.2) {$0$};
		\filldraw[very thick] (3.5,-1.75) circle (.04cm)node[anchor=north east] {$2e_1-e_2$};
		\filldraw[very thick] (0,1.5) circle (.04cm)node[left] {$e_2$};
		\end{tikzpicture}
		
		\caption{Faces of a cone.}
		\label{F:FacesofCone}
	\end{figure}
\end{example}

We continue with some properties of faces of a polyhedral cone.

\begin{lemma}
	Let $\sigma  =\cone\{v_1, \ldots,v_r\}\subset N$ be a polyhedral cone. Then:
	\begin{enumerate}
		\item Every face of $\sigma$ is a polyhedral cone.
		\item Intersection of any two faces of $\sigma$ is again a face of $\sigma$. 
		\item A face of a face of $\sigma$ is again a face of $\sigma$.
		\item Let $\tau$ be a face of  $\sigma$. Then $\sigma^\vee \subset \tau^\vee$. 
	\end{enumerate}
\end{lemma}

\begin{proof}$\ $
	Let $\calA =\{v_1,\ldots,v_r \}$ and $\tau = H_m \cap \sigma$ for some $m \in \sigma^\vee$. Then $\tau$ is generated by $\{v_i \in \calA \mid \langle m, v_i \rangle=0 \}$. Hence it is again a polyhedral cone in $N$. 
	
	Suppose $\tau = H_m \cap \sigma$ and $\tau'=H_{m'} \cap \sigma$ for some $m, m' \in \sigma^\vee$. The intersection of these faces is $\tau \cap \tau' = \left(H_m \cap H_{m'} \right)\cap \sigma$. We will prove that  $\tau \cap \tau' = H_{m+m'} \cap \sigma$. 
	
	First note that the sum $m+m' \in \sigma^\vee$. 
	Then, for any $v \in (H_m \cap H_{m'} ) \cap \sigma$, we have $\langle m+m',v\rangle =\langle m,v\rangle + \langle m',v\rangle$. Each summand vanishes. Hence $\langle m+m',v\rangle=0$ which implies $\tau \cap \tau' \subset H_{m+m'} \cap \sigma$. 
	Conversely, for any $v \in H_{m+m'} \cap \sigma$, we have $0=\langle m+m',v \rangle = \langle m,v\rangle + \langle m',v\rangle$. Since $v \in \sigma$ and $m,m' \in \sigma^\vee$, we have $\langle m,v\rangle , \langle m',v\rangle \geq 0$. Hence both summand has to be 0, in which case $v \in (H_m \cap H_{m'}) \cap \sigma.$ This gives 
	$H_{m+m'}  \cap \sigma \subset \tau \cap \tau' $, which completes the proof. 
	
	Let $\gamma$ be a face of $\tau$ which is a face of $\sigma$. We claim that $\gamma$ is a face of $\sigma$. Since $\gamma \preceq \tau$ and $\tau \preceq \sigma$,  we can write $\gamma = \tau \cap H_v$ and $\tau = \sigma\cap H_u$ for some $v\in \tau^\vee$ and $u \in \sigma^\vee$. We want to find a  positive number $p  \in \RR_>$ so that $v+pu \in \sigma^\vee$ and $\tau = \sigma \cap H_{v+pu}$. We claim that 
	\[
	p := 1+ \max_{v_i \notin \tau} \frac{\vert \langle u, v_i \rangle \vert}{\langle u,v_i \rangle }
	\]
	satisfies the desired properties. 
	
	First note that when $v_i \notin \tau $, $\langle u, v_i \rangle >0$. Hence $p \in \RR_>$. To see $v+pu \in \sigma^\vee$, we look at $\langle  v+pu,v_i \rangle$ for any $i \in \{1,\ldots,r\}.$ Observe that when $v_i \in \tau, $ $\langle u , v_i \rangle \geq 0$, hence  $\langle  v+pu,v_i \rangle \geq 0$.  If $v_i \notin \tau$, then 
	\[
	\langle v+pu,v_i \rangle = \langle u, v_i \rangle + p\langle u, v_i \rangle \geq \langle u,v_i\rangle >0,
	\]
	where the final inequality holds because $\langle u,v_i\rangle  \geq 0$ for all $i$, with equality if and only if $v_i \in \tau$. This proves that $v+pu \in \sigma^\vee$. 
	
	Next we want to show $\gamma = \tau \cap H_v= \sigma \cap H_u \cap H_v = \sigma \cap H_{v+pu}$. For an element $w \in \tau$, we have $\langle u,w \rangle =\langle v,w \rangle =0$. This implies that $\langle v+pu,w \rangle =0$. Hence $w \in \sigma \cap H_{v+pu}$. For the reverse inclusion, it suffices to check that for every $v_i \in \sigma \cap H_{v+pu}$, we have $v_i \in \gamma$. For $v_i \notin \gamma$, we showed above that  $\langle  v+pu,v_i \rangle >0$. So if $v_i \notin \tau$, then $v_i \notin H_{v+pu}$. 
	Now, suppose $v_i \in \sigma \cap H_{v+pu}$. Then $v_i$ has to be in $H_u$. So we have $0 = \langle v+pu,v_i \rangle = \langle u,v_i \rangle + p \langle u ,v_i \rangle = \langle u,v_i \rangle$. This means that $u_i \in H_v$. So $v_i \in \sigma \cap H_u \cap H_v$. 
	Hence $\gamma = \sigma \cap H_{v+pu}$. In particular $\gamma$ is a face of $\sigma$. 
	
	Lastly, since $\tau$ is contained in $\sigma$, for any element $m \in \sigma^\vee$ and $t\in \tau \subset \sigma$, we have $\langle m,t \rangle \geq 0.$ Hence $m \in \tau^\vee$. Thus $\sigma^\vee \subset \tau^\vee$. 
\end{proof}

For a cone $\sigma : = \cone (\calA)$, a face $\tau = H_m \cap \sigma$ of $\sigma$ is generated by a subset $\calF := \{u \in \calA \mid \langle m,u \rangle = 0 \} $ of $\calA$. We call $\calF$ a \demph{face} of $\calA$ and denote it by $\calF \preceq \calA$. 

\begin{definition}
	Given a face $\tau$ of a polyhedral cone $\sigma \subset N$, we define
	\begin{align*}
	\tau^\perp &= \{m \in M \mid \langle m , u \rangle = 0 \text{ for all } u \in \tau \},\\
	\tau^* & = \{m \in \sigma^\vee \mid \langle m , u \rangle = 0 \text{ for all } u \in \tau \} = \sigma^\vee \cap \tau^\perp.
	\end{align*}
	
	We call $\tau^*$ the face of $\sigma^\vee$ dual to $\tau$. 
\end{definition}

In Corollary \ref{DualIsaCone}, we have seen that $\sigma^\vee$ is again a polyhedral cone in $M$. Now, we will relate the faces of $\sigma$ to faces of its dual $\sigma^\vee$. 

\begin{definition}
	Given a polyhedral cone $\sigma \in N$, the \demph{relative interior}, denoted by $\Relint(\sigma)$, of $\sigma$ is the topological interior of $\sigma$ in its span $\langle \sigma \rangle $.
\end{definition}

The following theorem gives a relation between faces of $\sigma$ and $\sigma^\vee$. 

\begin{theorem}\label{ConeCorrespondence}
	Let $\tau$ be a face of a polyhedral cone $\sigma \subset N$. Then
	\begin{enumerate}
		\item $\tau^*$ is a face of $\sigma^\vee$. 
		\item The map $\tau \to \tau^*$ is an inclusion reversing bijection between the faces of $\sigma$ and the faces of $\sigma^\vee$. 
		
	\end{enumerate}
\end{theorem}

\begin{proof}
	Let $v \in \Relint(\tau)$. Since $v\in \sigma = (\sigma^\vee)^\vee$, $F = H_v \cap \sigma^\vee$ is a face of $\sigma^\vee$. We will show that $F = \tau^*$. Pick an element $u \in \tau^*$. We have $\langle u , t \rangle = 0$ for all $t \in \tau$. In particular $\langle u,v \rangle =0$. Hence $u \in F$. Conversely, take any $u \in F$. If $\{ v_1, \ldots, v_k \}$ is a generating set of $\tau$, since $v \in \Relint(\tau)$, we can write $v = a_1 v_1 + \cdots + a_k v_k$, where  each $a_i$ is positive. Then we have $\langle u,v \rangle = a_1 \langle u, v_1 \rangle + \cdots + a_k \langle u, v_k \rangle$. Since $u \in \sigma^\vee$, we have $\langle u , v_i \rangle \geq 0$ for all $i$. Since the $a_i$'s are strictly positive, we must have $\langle u,v_i \rangle =0$ for all $i$, which means for any $t \in \tau$, $\langle u,t \rangle =0 $, i.e., $u \in \tau^\perp$. Thus $ u \in \sigma^\vee \cap  \tau^\perp$. This proves $\tau^* = F $. Hence $\tau^*$ is a face of $\sigma^\vee$.
	
	For the second part of the theorem, first note that the map is inclusion reversing. If $\tau \subset \tau'$ are two faces of $\sigma$, then $\tau^\perp \supset (\tau')^\perp$, which implies $\tau^* \supset (\tau')^*$. 
	Moreover, every face of $\sigma$ arises this way. If $F$ is a face of $\sigma^\vee$, then $F = H_v \cap \sigma^\vee$ for some $v \in \sigma$. Pick the smallest face $\tau$ of $\sigma$ containing $v$. We claim $F = \tau^*$. For any element $u \in \tau^*$, we have $u \in \tau^\perp \cap \sigma^\vee$. Since $u \in \tau^\perp$ and $v \in \tau$, we have $u \in H_v$. Hence $u \in H_v \cap \sigma^\vee = F$. Conversely, pick an element $u\in F$. Note that $u \in v^\perp$. By assumption, $\tau  \subset u^\perp$. Hence $u \in \tau^*$. 
	By definition of $\tau^*$, we have $\tau \subset (\tau^*)^*. $ Since the mapping is order reversing, we have $\tau^* \supset ((\tau^*)^*)^*$. On the other hand, any element $m \in \tau^*$ is orthogonal to every element of $\sigma \cap (\tau^*)^\perp \subset (\tau^*)^\perp$, so $\tau^* \subset ((\tau^*)^*)^*$, which implies $\tau^*  = ((\tau^*)^*)^*$. Combining this with the surjectivity we deduce that the mapping is bijective. 
\end{proof}

\begin{corollary}\label{C:LinealitySpace}
	The polyhedral cone $\sigma$ has a unique minimal face $W$ with respect to inclusion. In particular, $W$ satisfies the following properties.
	\begin{enumerate}
		\item $W = (\sigma^\vee)^\perp$.
		\item $W = \sigma \cap (-\sigma)$.
		\item $W$ is the largest subspace contained in $\sigma$.
	\end{enumerate} 
\end{corollary}

\begin{proof}
	Because of the correspondence between faces of $\sigma$ and $\sigma^\vee$, the minimal face W of $\sigma$ must correspond to the maximal face of $\sigma^\vee$, namely $\sigma^\vee$ itself. Thus, the minimal face of $\sigma$ is $(\sigma^\vee)^* = (\sigma^\vee)^\vee \cap (\sigma^\vee)^\perp = (\sigma^\vee)^\perp.$ 
	
	To show $(\sigma^\vee)^\perp= \sigma \cap (-\sigma)$, pick any $v  \in (\sigma^\vee)^\perp$. By definition, $v \in (\sigma^\vee)^\vee = \sigma$ and also $v$ lies in $((-\sigma)^\vee)^\vee= -\sigma$. Hence $v \in \sigma \cap (-\sigma).$ Conversely, given any $v \in \sigma \cap (-\sigma)$, for any element $u \in \sigma^\vee$, we have $\langle u,v \rangle \geq 0$ (since $v \in \sigma$) and  $-\langle u,v \rangle = \langle u, -v \rangle \geq 0$ (since $-v \in \sigma$). Thus $\langle u,v \rangle =0$, i.e., $v \in (\sigma^\vee)^\perp$. Hence $W = \sigma \cap (-\sigma)$. 
	
	First note that since $W$ is a face of $\sigma$, it contains $0$, it is closed under addition, and closed under scalar multiplication. Hence $W$ is subspace in $\sigma$. Then note that for any linear subspace $W'$ in $\sigma$ and any $u \in \sigma^\vee$, we must have $W' \subseteq u^\perp$. Hence $W' \subset W$. So $W$ is the largest subspace contained in $\sigma$.  
\end{proof}

\begin{definition}
	$W$ is called the \demph{lineality space} of $\sigma$. 
\end{definition}

The following example from \cite[Example 1.2.11]{CLS} illustrates Corollary \ref{C:LinealitySpace} when $\dim \sigma < \dim N$. 

\begin{example}
	Let $\sigma = \cone\{e_1,e_2 \} \subseteq \RR^3$. Figure \ref{F:LinealitySpace}
	shows $\sigma$ in $\RR^3$ and $\sigma^\vee$ in $(\RR^3)^*$. 
\end{example}	
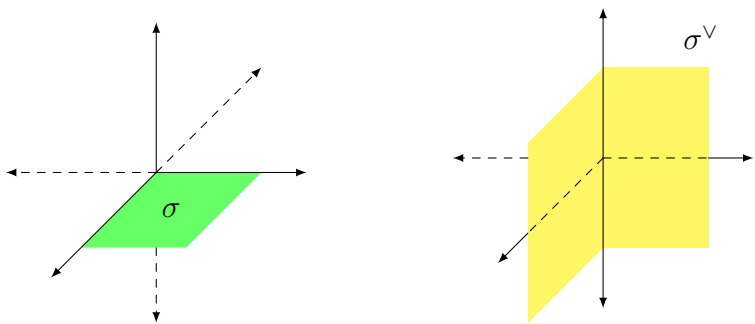
\begin{figure}[!h]
	\centering
	\begin{tikzpicture}[line cap=round,line join=round,>=latex,scale=1]
	
	\draw[->,dashed] (0,-1) -- (0,-2); 
	\fill[line width=0.pt,color=green,fill=green,fill opacity=0.6] (0,0) -- (-1,-1)  -- (.4,-1)--(1.4,0) -- cycle;
	
	\draw[->] (0,0) -- (0,2); 
	\draw[->] (0,0) -- (2,0); 
	\draw[->,dashed] (0,0) -- (-2,0); 
	\draw[->] (0,0) -- (-1.4,-1.4); 
	\draw[->,dashed] (0,0) -- (1.4,1.4); 
	
	\node at (0.2,-.5) {$\sigma$};
	\end{tikzpicture} \hspace{4em}
	\begin{tikzpicture}[line cap=round,line join=round,>=latex,scale=1]
	
	\fill[line width=12pt,color=green,fill=yellow,fill opacity=0.6] (0,-1.2) -- (-1,-2.2)  -- (-1,.2)--(0,1.2) -- cycle;
	\fill[line width=1pt,color=green,fill=yellow,fill opacity=0.6] (0,-1.2) -- (1.4,-1.2)  -- (1.4,1.2)--(0,1.2) -- cycle;
	
	\draw[<->] (0,-2) -- (0,2); 
	\draw[-,dashed] (0,0) -- (1.4,0); 
	\draw[->] (1.4,0) -- (2,0); 
	\draw[->,dashed] (-1,0) -- (-2,0); 
	\draw[-, dashed] (0,0) -- (-1,-1); 
	\draw[->] (-1,-1) -- (-1.4,-1.4); 
	\node at (1.3,1.6) {$\sigma^\vee$};
	\end{tikzpicture}
	
	\caption{A two-dimensional cone in $\RR^3$ and its dual in $(\RR^3)^*$.}
	\label{F:LinealitySpace}
\end{figure}

The maximal face of $\sigma$ is itself. Then
\begin{align*}
\sigma^* &= \{m\in (\RR^3)^* \mid \langle m,v \rangle = 0 \text{ for all } v \in \sigma \}\\
&=  \{m\in (\RR^3)^* \mid \langle m,e_1 \rangle = \langle m,e_2 \rangle = 0\}\\
&=\cone\{\pm e_3^*\}.
\end{align*}
Hence the minimal face of $\sigma^\vee$ is the $z$-axis. Note also that $\dim \sigma + \dim \sigma^* = 3$.  \hfill$\square$

We next prove some useful properties of faces. 
\begin{lemma} \label{SumofElementsinCone}
	Let $\tau$ be a face of a polyhedral cone $\sigma$ in $N$. If $v,w \in \sigma$ such that $v+w \in \tau$, then $v,w \in \tau$. 
\end{lemma}

\begin{proof}
	Since $\tau$ is a face of $\sigma$, there exists an $m \in \sigma^\vee$ such that $\tau = H_m \cap \sigma$. Since $v,w \in \sigma$, we have $\langle m,v \rangle,  \langle m,w \rangle \geq 0$. On the other hand, since $v+w \in \tau$, we have $\langle m, v+w \rangle =0 $. So, we must have $ \langle m,v \rangle = \langle m ,w \rangle =0$, i.e., $v,w \in \tau$. 
\end{proof}

\begin{lemma}\label{L:coneFact}
	Let $\tau$ be a face of a polyhedral cone $\sigma$ in $N$. Then for any $w \in \tau^\vee$, there are $u,v \in \sigma^\vee$ with $v \in \tau^\perp$ such that $w = u-v$. 
\end{lemma}

\begin{proof}
	Let $\sigma = \cone (v_1,\ldots,v_r)$, and $\tau = H_m \cap \sigma$ for some $m \in \sigma^\vee$. Note that if $v_i \notin \tau$, then $\langle v_i , m \rangle >0$. Hence define 
	\[
	k := \max_{v_i \notin \tau} \frac{\vert \langle w,v_i \rangle \vert }{\langle m , v_i\rangle}.
	\]
	Note that since $m\in \tau^\perp$, $km \in \tau^\perp$ as well. We claim that $w+km \in \sigma^\vee$. If $v_i \in \tau $, $\langle w +km,v_i \rangle = \langle w,v_i \rangle + k \langle m,v_i \rangle \geq 0$ as all terms that appear are nonnegative. If $v_i \notin \tau$,
	\[
	\langle w +km,v_i \rangle = \langle w, v_i \rangle + k \langle m, v_i\rangle =  \langle w,v_i \rangle + \max_{v_i \notin \tau} \frac{\vert \langle w,v_i \rangle \vert }{\langle m , v_i\rangle} \langle v_i, m\rangle  \geq 0.
	\]
	Hence $w+km \in \sigma^\vee$. So pick $u = w+km$ and $v=km$. 
\end{proof}

Using the correspondence given in Theorem \ref{ConeCorrespondence}, we can prove a stronger version of Corollary \ref{Seperation}. That is, when two cones intersect in a common face, we can separate these cones by a hyperplane. 

\begin{lemma}[Separation Lemma] \label{SeparationLemma}
	Let $\sigma$ and $\sigma'$ be polyhedral cones in $N$ that meet in a common face $\tau =\sigma \cap \sigma'$. Then, there exists an element $m \in \sigma^\vee \cap (-\sigma')^\vee$ such that 
	\[
	\tau = H_m\cap \sigma = H_m \cap \sigma'.
	\] 
\end{lemma}

\begin{proof}
	Let $\gamma = \sigma + (-\sigma')$. Notice that $\gamma$ is a polyhedral cone. Fix any $m \in\Relint (\gamma^\vee).$ Then by Theorem \ref{ConeCorrespondence}, $(\gamma^\vee)^* = \gamma \cap H_m$ is the smallest face of $\gamma$, and 
	$\gamma \cap H_m = \gamma \cap (-\gamma)$. Hence $\gamma \cap H_m = (\sigma - \sigma') \cap (\sigma' - \sigma).$
	We claim that $m$ satisfies the desired properties. Since $\sigma \subset \gamma$, we have $m \in \gamma^\vee \subset \sigma^\vee$. Since $\tau \subset \sigma$, we have $\tau \subset \gamma$. Also since $\tau \subset \sigma'$, we have $\tau \subset -\gamma$. Hence $\tau \subset \gamma \cap (-\gamma) = \gamma \cap H_m \subset H_m$. 
	
	Conversely, given any element $v \in H_m \cap \sigma$, we have $$v \in H_m \cap \sigma \subset \gamma \cap H_m = \gamma \cap (-\gamma) \subset -\gamma = \sigma' - \sigma.$$ Hence $v = w-w'$ for some $w \in \sigma$ and $w' \in \sigma'$. This means that $v+w \in \sigma \cap \sigma' = \tau$. By Lemma \ref{SumofElementsinCone}, we have $v \in \tau$. This proves $\tau = H_m \cap \sigma$. Similar arguments can be made to show $\tau = H_m \cap \sigma'$.  
\end{proof}

\section{Monoids}\label{S:Monoids}

Monoids play an important role in constructing toric varieties associated to rational polyhedral cones. We will define these objects and give some important properties which will be helpful later.

\begin{definition}
	
	A \demph{monoid} is a nonempty set $S$ with an associative operation $* \colon S \times S \to S$ that has an identity element $1_S$, $1_S * s =s \text{ for all } s \in S$.
	If $S$ and $T$ are monoids, a map $f \colon S \to T$ is called a \demph{monoid homomorphism} if $f$ is compatible with the structure of the monoids, that is $f(s * t) =f(s) * f(t)$ and $f(1_S)= 1_T$.
	
	A monoid $S$ is called an \demph{affine monoid} if $S$ is finitely generated and isomorphic to a submonoid of a free abelian group.
\end{definition}

\begin{example}
	The monoids $\NN^n$ and $\ZZ^n$ with usual addition are affine monoids. Also for a given finite set $\calA \in \ZZ^n$, the monoids $\ZZ \calA$ and $\NN \calA$ are again affine monoids.  \hfill$\square$
\end{example}

\begin{definition}
	Given an affine monoid $S \subset M$, the \demph{monoid algebra} $\CC\left[S \right]$ is the vector space over $\CC$ with $S$ as a basis and multiplication induced by the monoid structure of $S$. To be more precise
	\[
	\CC\left[S\right] = \biggl\{  \sum_{m \in S} c_m \chi^m \mid c_m \in \CC \text{ and } c_m =0 \text{ for all but finitely many } m\biggr\},
	\]
	where the multiplication is distributive and induced by $\chi^m \cdot \chi^{m'} = \chi^{m * m'}$.
\end{definition}

\begin{example}
	The polynomial ring $\CC \left[\NN^n\right] = \CC \left[x_1, \ldots, x_n \right]$ is a monoid algebra.
	The ring of Laurent polynomials $\CC \left[\ZZ^n\right] = \CC \left[x_1^\pm, \ldots, x_n^\pm \right]$ is a monoid algebra. 
	
	For the set $\calA = \{ 2,3\} \subset \ZZ$, consider the affine monoid $S = \NN \calA$. Then the corresponding monoid algebra is given by $\CC \left[S\right] = \CC\left[ t^2,t^3\right] \subset \CC\left[t\right]. $
	
	For the set $\calA = \{(3,0), (2,1), (1,2), (0,3) \} \subset \ZZ^2$, consider the affine monoid $S= \NN \calA$. Then the corresponding monoid algebra is given by $\CC \left[S\right] = \CC\left[s^3,s^2t, st^2,t^3\right] \subset \CC \left[s,t\right]$.
	\hfill$\square$
\end{example}

Given a rational polyhedral cone $\sigma \subset N$,  we are particularly interested in the monoid $S_\sigma : = \sigma^\vee \cap M_\ZZ.$ These monoids play an important role when we construct toric varieties associated to $\sigma$. 

\begin{proposition}[Gordan's Lemma]
	$S_\sigma$ is finitely generated. 
\end{proposition}

\begin{proof}
	Since $\sigma$ is rational, so is $\sigma^\vee$. Take a generating set $\{u_1, \ldots, u_r \} \subset \sigma^\vee \cap M_\ZZ$ for $\sigma^\vee$, and define $K := \{\sum_{i=1}^{r} t_i u_i \mid 0\leq t_i \leq 1 \}$. Since $K$ is closed and  bounded in $M \simeq \RR^n$, it is compact. Since $M_\ZZ$ is discrete, the intersection $K\cap M_\ZZ$ is finite. (If this intersection were infinite, then by compactness there has to be a limit point $x$ of any infinite sequence of elements. But then $x$ is an element of the discrete set $K \cap M_\ZZ$ all of whose neighborhoods contain other points of $K \cap M_\ZZ$, which contradicts the discreteness.)
	
	Next we claim that $K \cap M_\ZZ$ generates $S_\sigma$ as a monoid. To prove this, let $w \in S_\sigma$. Since $w \in \sigma^\vee$, we can write $w = \sum_{i=1}^{r} a_i u_i$, where $a_i \geq 0$ for all $i$. Then write $a_i = z_i + t_i$, where $z_i$ is a nonnegative integer and $0 \leq t_i \leq 1$.  Hence, $u = \sum_{i=1}^{r} z_i u_i + \sum_{i=1}^{r} t_iu_i$. Note that the former sum is in $M_\ZZ$ and since $M_\ZZ$ is closed under subtraction, the latter sum is in $M_\ZZ$. It is in $K$ by definition. Hence the latter sum is in $K \cap M_\ZZ$. Since each $u_i$ is also in $K \cap M_\ZZ$, $w$ is a nonnegative integer combination of elements in $K \cap M_\ZZ$. 
\end{proof}

\begin{example} \label{Ex:Gordan}
	Consider the cones in Figure \ref{F:ConesR2}. Their duals are given in Figure \ref{F:DualConesR2}. The monoids $S_{\sigma_1}$, $S_{\sigma_2}$, and $S_{\sigma_3}$ are generated by the sets $\{e_1^*, e_2^*, -e_2^* \}$, $\{e_1^*, e_2^*\}$, and  $\{e_1^*, e_1^*+e_2^*, e_1^*+2e_2^*\}$, respectively. This also shows the cone generators and monoid generators may differ. \hfill$\square$
\end{example}

Next we would like to understand the relation between $S_\sigma$ and $S_\tau$ for some face $\tau$ of $\sigma$. We begin with a lemma.

\begin{lemma} \label{L:DualFaceRelation}
	Let $\sigma \subset N$ be a polyhedral cone. For an element $m \in \sigma^\vee$, let $\tau = H_m \cap \sigma$. Then, $\tau^\vee = \sigma^\vee + \RR_\geq(-m)$. 
\end{lemma}

\begin{proof}
	Since both $\tau^\vee$ and $\sigma^\vee + \RR_\geq(-m)$ are polyhedral cones, it suffices to show that they have the same duals. We claim that $(\sigma^\vee + \RR_\geq(-m))^\vee = \sigma \cap (-m)^\vee$. To see this, take an element $v \in \sigma \cap (-m)^\vee$. For any element $u+r (-m) \in \sigma^\vee + \RR_\geq(-m) $, since $v$ is both in $\sigma$ and $(-m)^\vee$, we have $\langle u+r(-m),v \rangle = \langle u, v\rangle + r \langle -m,v \rangle \geq 0$. On the other hand, any element $t \in (\sigma^\vee + \RR_\geq(-m))^\vee$ must in particular be in $(\sigma^\vee)^\vee$ and $(-m)^\vee$. Hence we have $(\sigma^\vee + \RR_\geq(-m))^\vee = \sigma \cap (-m)^\vee$. 
	
	For  an element $v \in \sigma$, since $m \in \sigma^\vee$, we have $\langle m,v \rangle \geq 0$. Thus, $v \in (-m)^\vee$ if and only if $\langle -m,v \rangle \geq 0$. Hence $\langle m,v \rangle =0.$ So $\sigma \cap (-m)^\vee = H_m \cap \sigma = \tau$. Since $(\tau^\vee)^\vee =\tau$, $\tau^\vee$ and $\sigma^\vee + \RR_\geq(-m)$ have the same duals, so they are equal.
\end{proof}

\begin{example} \label{Ex:FaceViaCone}
	Consider the cones $\tau =\cone \{e_2\}, \ \sigma_1 = \cone\{e_1,e_2\}$, and $\sigma_2 = \cone \{e_2, -e_1-e_2 \}$ in $\RR^2$.  Their duals are given by $\tau^\vee =\cone\{\pm e_1^*,e_2^*\}, \ \sigma_1^\vee = \cone \{e_1^*, e_2^* \}$, and $\sigma_2^\vee = \cone \{ -e_1^*, -e_1^*+e_2^*\}.$
	
	First consider $\tau$ as a face of $\sigma_1$. For the vector $m_1=e_1^* \in \sigma_1^\vee$, we have $\tau = H_{m_1} \cap \sigma_1$ and $\tau^\vee = \sigma^\vee_1 + \RR_\geq(-m_1)$.
	
	Now consider $\tau$ as a face of $\sigma_2$. For the vector $m_2= -e_1^* \in \sigma_2^\vee$, we have $\tau = H_{m_2} \cap \sigma_2$ and $\tau^\vee = \sigma^\vee_2 + \RR_\geq(-m_2)$.\hfill$\square$
	
\end{example}

Using Lemma \ref{L:DualFaceRelation}, we deduce a result on lattice points in a face of a polyhedral cone. 

\begin{corollary} \label{C:MonoidFace}
	Let $\sigma$ be a rational cone, and $m \in \sigma^\vee \cap M_\ZZ$. Then $S_\tau = S_\sigma + \ZZ_\geq (-m)$, where $\tau =H_m \cap \sigma$. 
\end{corollary}

Lemma \ref{L:DualFaceRelation} also helps us to understand what $S_\tau$ looks like when $\tau$ is the intersection of two rational polyhedral cones.

\begin{theorem}
	Let $\sigma$ and $\sigma'$ be two rational polyhedral cones, and $\tau =\sigma \cap \sigma'$ be a face of both $\sigma$ and $\sigma'$. Then $S_\tau = S_\sigma + S_{\sigma'}.$
\end{theorem}

\begin{proof}
	By the Separation Lemma, we have $\tau = H_m \cap \sigma = H_m \cap \sigma'$ for some $m \in \sigma^\vee \cap (-\sigma')^\vee \cap M_\ZZ$. Hence $S_\tau = S_\sigma+\ZZ_\geq (-m) = S_{\sigma'} + \ZZ_\geq(-m). $ In particular, $S_\sigma + S_{\sigma'} \subset S_\tau$. On the other hand, since $m \in (-\sigma')^\vee \cap M_\ZZ$, we have $-m \in S_{\sigma'}$. Hence $S_\tau = S_\sigma + \ZZ_\geq(-m) \subset S_\sigma + S_{\sigma'}.$ So, $S_\tau = S_\sigma+ S_{\sigma'}$ as desired.
\end{proof}

\begin{example}
	We continue with Example \ref{Ex:FaceViaCone}. The monoids $S_\tau$, $S_{\sigma_1}$, and $S_{\sigma_2}$ are generated by the sets $\{ \pm e_1^*, e_2^*\}$, $\{e_1^*, e_2^*\}$, and $\{ -e_1^*, -e_1^*+e_2^*\}$, respectively. Then one has $S_\tau = S_{\sigma_1} + \ZZ_\geq(-m_1)$ and $S_\tau = S_{\sigma_2} + \ZZ_\geq(-m_2)$.\hfill$\square$
\end{example}

\section{Fans}
We now introduce objects called \demph{fans}, which we will use when we construct abstract toric varieties. 
We start with an important type of polyhedral cones. 

\begin{definition}
	A polyhedral cone $\sigma \subseteq N$ is called \demph{strongly convex} if $\sigma \cap (-\sigma) = \{0\}.$
\end{definition}

This condition can be stated in several ways. 

\begin{proposition}
	Let $\sigma \subset N$ be a polyhedral cone. Then the following are equivalent.
	
	\begin{enumerate}
		\item $\sigma$ is strongly convex. 
		\item $\sigma$ contains no positive dimensional subspace of $N$.
		\item $\{0\}$ is a face of $\sigma$. 
	\end{enumerate}
\end{proposition}

\begin{proof}
	(1) and (2) are equivalent, since $\sigma \cap (-\sigma)$ is the largest linear subspace of $\sigma$. (1) and (3) are equivalent because $\sigma \cap (-\sigma)$ is the smallest face of $\sigma$. 
\end{proof}

\begin{definition}
	A \demph{fan} $\Sigma$ in $N$ is a finite collection of polyhedral cones in $N$ such that
	\begin{enumerate}
		\item Every face of a cone in $\Sigma$ is also in $\Sigma$.
		\item The intersection of two cones in $\Sigma$ is a face of each cone.
	\end{enumerate}
	
	A fan $\Sigma$ is called \demph{rational} if every cone in $\Sigma$ is rational. It is called \demph{strongly convex} if every cone in $\Sigma$ is strongly convex.  The \demph{support} of $\Sigma$ is $\vert \Sigma \vert = \bigcup_{\sigma \in \Sigma} \sigma \subseteq N$. The fan $\Sigma$ is called \demph{complete} if $\vert \Sigma \vert =N$. 
	
\end{definition}

\begin{example}
	Figure \ref{F:Fans} are some examples of rational  fans in $\RR^2$. The fan on the left is not complete whereas the other two are complete. Note that any intersection of the cones in each fan is a face of each cone and a cone in the fan. \hfill$\square$
	
	\begin{figure}[!ht]
		\centering
		\begin{tikzpicture}[line cap=round,line join=round,>=latex,scale=.6]
		\fill[color=yellow,fill opacity=0.3] (0,0) -- (3,0) -- (3,3) -- (0,3) -- cycle;
		\fill[color=blue,fill opacity=0.3] (0,0) -- (6.5,-3.25) -- (3,0) -- cycle;

		\draw [->,thick] (0,0) -- (3.35,0);
		\draw [->,thick] (0,0) -- (6.5,-3.25);
		\draw [->,thick] (0,0) -- (0,3.5);
		
		\filldraw[very thick] (5,-2.5) circle (.04cm)node[anchor=east] {$2e_1-e_2$};
		\filldraw[very thick] (2.5,0) circle (.04cm)node[above] {$e_1$};
		\filldraw[very thick] (0,2.5) circle (.04cm)node[left] {$e_2$};
		
		\draw (1.5,1.5) node {$\sigma_0$};
		\draw (3,-.75) node {$\sigma_1$};
		\draw (3,-4) node[below] {$\Sigma_1$};
		\end{tikzpicture}\hspace{1cm}
		\begin{tikzpicture}[line cap=round,line join=round,>=latex,scale=.6]
		\fill[color=yellow,fill opacity=0.3] (0,0) -- (3,0) -- (3,3) -- (0,3) -- cycle;
		\fill[color=blue,fill opacity=0.3] (0,0) -- (0,3) -- (-3,3) -- (-3,-3) -- cycle;
		\fill[color=green,fill opacity=0.3] (0,0) -- (-3,-3) -- (3,-3) -- (3,0) -- cycle;

		\draw [->,thick] (0,0) -- (3.5,0);
		\draw [->,thick] (0,0) -- (-3.25,-3.25);
		\draw [->,thick] (0,0) -- (0,3.5);
		
		\filldraw[very thick] (-2.5,-2.5) circle (.04cm)node[right] {$-e_1-e_2$};
		\filldraw[very thick] (2.5,0) circle (.04cm)node[above] {$e_1$};
		\filldraw[very thick] (0,2.5) circle (.04cm)node[left] {$e_2$};
		
		\draw (1.5,1.5) node {$\sigma_0$};
		\draw (-1.5,1)  node {$\sigma_1$};
		\draw (1,-1.5) node {$\sigma_2$};
		\draw (0,-4) node[below] {$\Sigma_2$};
		\end{tikzpicture}\hspace{1cm}
		\begin{tikzpicture}[line cap=round,line join=round,>=latex,scale=.6]
		\fill[color=yellow,fill opacity=0.3] (0,0) -- (3,0) -- (3,3) -- (0,3) -- cycle;
		\fill[color=blue,fill opacity=0.3] (0,0) -- (0,3) -- (-3,3) -- (-3,0) -- cycle;
		\fill[color=green,fill opacity=0.3] (0,0) -- (0,-3) -- (3,-3) -- (3,0) -- cycle;
		\fill[color=red,fill opacity=0.3] (0,0) -- (0,-3) -- (-3,-3) -- (-3,0) -- cycle;
		
		\draw [->,thick] (0,0) -- (3.5,0);
		\draw [->,thick] (0,0) -- (-3.5,0);
		\draw [->,thick] (0,0) -- (0,3.5);
		\draw [->,thick] (0,0) -- (0,-3.5);
		
		\filldraw[very thick] (-2.5,0) circle (.04cm)node[above] {$-e_1$};
		\filldraw[very thick] (2.5,0) circle (.04cm)node[above] {$e_1$};
		\filldraw[very thick] (0,2.5) circle (.04cm)node[left] {$e_2$};
		\filldraw[very thick] (0,-2.5) circle (.04cm)node[left] {$-e_2$};
		
		\draw (1.5,1.5) node {$\sigma_0$};
		\draw (-1.5,1.5)  node {$\sigma_1$};
		\draw (1.5,-1.5) node {$\sigma_3$};
		\draw(-1.5,-1.5) node {$\sigma_2$};
		\draw (0,-4) node[below] {$\Sigma_3$};
		\end{tikzpicture}
		
		\caption{Fans in $\RR^2$.}
		\label{F:Fans}
	\end{figure}
	
\end{example}

Let $\Sigma $ and $\Sigma'$ be two fans in real vector spaces $N$ and  $N'$, respectively. Then the \demph{product} of $\Sigma$ and $\Sigma'$ in $N \times N'$ is given by
\[
\Sigma\times \Sigma' := \{\sigma \times \sigma' \textrm{ } \mid \sigma \in \Sigma \text{ and } \sigma' \in \Sigma'\}.
\]

\begin{lemma}
	The product $\Sigma \times \Sigma'$ is a fan in the real vector space
	$N \times N'$.  Moreover, $\Sigma \times \Sigma'$ is strongly convex if both $\Sigma$ and $\Sigma'$ are strongly convex.
\end{lemma}

\begin{proof}
	Note that a face of $\sigma \times \sigma'$ has the form $\tau \times \tau'$, where $\tau$ is a face of $\sigma$ and $\tau'$ is a face of $\sigma'$. Since $\Sigma$ and $\Sigma'$ are fans, we have $\tau \in \Sigma$ and $\tau' \in \Sigma'$. Hence their product is in $\Sigma \times \Sigma'$.  Also let $\sigma_1 \times \sigma_1', \sigma_2 \times \sigma_2' \in \Sigma \times \Sigma$. Again since $\Sigma$ and $\Sigma'$ are both fans and intersection of any two cones in a  fan is a face of each, $(\sigma_1 \cap \sigma_2) \times (\sigma_1' \times \sigma_2')$ is a face of both $\sigma_1 \times \sigma_1'$ and $\sigma_2 \times \sigma_2'$.
	
	Now assume that both $\Sigma$ and $\Sigma'$ are strongly convex, and let $\sigma \in \Sigma $ and $\sigma' \in \Sigma'$. Since both $\sigma$ and $\sigma'$ are strongly convex, $0_N$ is a face of $\sigma$ and   $0_{N'}$ is a face of $\sigma'$. Hence $0_{N\times N'} = 0_N \times 0_{N'}$ is a face of $\sigma \times \sigma'$, and thus $\sigma \times \sigma'$ is strongly convex.  Hence $\Sigma \times \Sigma'$ is strongly convex.
\end{proof}

\begin{example}
	Consider the fan $\Sigma_3$ in Figure \ref{F:Fans}. It can be viewed as a product of two fans, where each fan consists of three cones $\RR_\geq$, $\RR_\leq$, and 0.\hfill$\square$
\end{example}

For a cone $\sigma$ in $N$, let $(N_\sigma)_\ZZ$ be the sublattice of $N_\ZZ$ generated (as a group) by $\sigma \cap N_\ZZ$. The lattice $(N_\sigma)_\ZZ$ has the quotient lattice $N(\sigma)_\ZZ := N_\ZZ/(N_\sigma)_\ZZ$, and the dual lattice of $N(\sigma)_\ZZ$ is given by $M(\sigma)_\ZZ := \sigma^\perp \cap M_\ZZ$ \cite[Page 52]{Fulton}. The \demph{star} of a cone $\sigma$  can be defined abstractly as the set of cones $\tau$ in $\Sigma$ that contain $\sigma$ as a face. Such cones are determined by their images in the real vector space $N(\sigma)$, i.e., by  
$$\overline{\tau} = (\tau + N_\sigma) / N_\sigma \subset N / N_\sigma = N(\sigma).$$ 
These cones $\{ \overline{\tau} \colon \sigma \preceq \tau \}$ form a fan in $N(\sigma)$, and we denote this fan by $\text{star}(\sigma)$. (We think of $\text{star}(\sigma)$ as the cones containing $\tau$, but realized as a fan in the quotient lattice $N(\sigma)_\ZZ$.)

We next show that $\text{star}(\sigma)$ is indeed a fan, and we explore some of its properties. 

\begin{lemma}
	\label{L:StarStronglyConvex}
	Let $\Sigma$ be a strongly convex fan in $N$ and $\sigma \in \Sigma$. Then $\text{star}(\sigma)$ is a strongly convex fan in $N(\sigma)$. Moreover, if $\Sigma$ is complete, then $\text{star}(\sigma)$ is also complete.  
\end{lemma}
\begin{proof}
	We will first show that $\text{star}(\sigma)$ is strongly convex. Let $\overline{\tau}$ be a cone in $\text{star}(\sigma)$. We need to show $\overline{\tau} \cap (-\overline{\tau})=\{\overline{0}\}$. Assume that $\overline{v} \in \overline{\tau} \cap (-\overline{\tau})$. Then we can write
	\begin{equation}
	\label{Eq:StarStronglyConvex}
	v = v_1 + w_1 = -v_2 +w_2
	\end{equation}
	for some $v_1,v_2 \in \tau$ and  $w_1, w_2 \in N_\sigma$.
	Since every element of $N_\sigma$ can be written as a difference of elements of $\sigma$, we have
	$$
	w_1 = \alpha_1 - \beta_1, \ \ w_2 = \alpha_2 - \beta_2,
	$$
	for some $\alpha_1,\alpha_2,\beta_1,\beta_2 \in \sigma$. Hence by \eqref{Eq:StarStronglyConvex} we have
	\[
	v_1 +v_2 = w_2 -w_1 = (\alpha_2 + \beta_1) - (\alpha_1 + \beta_2).
	\]
	Therefore, we have
	\[
	v_1 + v_2 + (\alpha_1 + \beta_2) = \alpha_2 + \beta_1.
	\]
	Since $v_1+v_2 \in \tau$, $\alpha_1 + \beta_2 \in \sigma \subset \tau$, and $\alpha_2+\beta_1 \in \sigma$, by Lemma \ref{SumofElementsinCone} we have $v_1 + v_2 \in \sigma$. Again by Lemma \ref{SumofElementsinCone}   we have $v_1,v_2 \in \sigma$. So $\overline{v} \in N_\sigma$, i.e.,  $\overline{v}=\overline{0}$. Hence $\text{star}(\sigma)$ is strongly convex.
	
	Next, we show $\text{star}(\sigma)$ is a fan in $N(\sigma)$. Let $\tau_1$ and $\tau_2$ be two cones in $\Sigma$ containing $\sigma$. We claim that $\overline{\tau_1} \cap \overline{\tau_2} = \overline{\tau_1 \cap \tau_2}$. Reverse inclusion is immediate as any element $\overline{v} \in \overline{\tau_1 \cap \tau_2}$ is both in $\overline{\tau_1}$ and $\overline{\tau_2}$.
	Now, let $\overline{v} \in \overline{\tau_1} \cap \overline{\tau_2}$. Then
	$
	v = v_1 + w_1 =v_2 + w_2
	$
	for some $v_1 \in \tau_1$, $v_2 \in \tau_2$, and $w_1, w_2 \in N_\sigma$. Write $w_i = \alpha_i - \beta_i$ for some $\alpha_i, \beta_i \in \sigma$. Then $v_1 + \beta_1 + \beta_2 = v_2 +\alpha_1 + \alpha_2 \in \tau_2$. So $ v_1 + \beta_1 + \beta_2 \in \tau_1 \cap \tau_2 $. Then $v = (v_1 + \beta_1 + \beta_2) + (w_1 - \beta_1 - \beta_2) \in \tau_1 \cap \tau_2 + N_\sigma$. Hence $\overline{v} \in  \overline{\tau_1 \cap \tau_2}$. We conclude that $\overline{\tau_1} \cap \overline{\tau_2} = \overline{\tau_1 \cap \tau_2}$.
	Since $\tau_1, \tau_2 \in \Sigma$, their intersection is a common face of both. Since both contain $\sigma$, their intersection will contain $\sigma$ as a face. 
	In particular, $\overline{\tau_1} \cap \overline{\tau_2} \in \text{star}(\sigma)$ is a face of both $\overline{\tau_1}$ and $\overline{\tau_2}$.
	
	For a cone $\overline{\tau} \in \text{star}(\sigma)$, let $\overline{\rho}:=u^\perp \cap \overline{\tau}$ be a face of $\overline{\tau}$ for some $u \in M(\sigma)_\ZZ= \sigma^\perp \cap M_\ZZ$. Note that for the face $\overline{\rho}$, we have $\rho = u^\perp \cap \tau$. Since $\rho$ is a face of $\tau$ which contains $\sigma$, we deduce that $\rho \in \text{star}(\sigma)$. Hence $\text{star}(\sigma)$ is a fan.
	
	Next assume $\Sigma$ is complete. Let $v \in \Relint(\sigma)$. Since $N$ is complete, the set 
	\[
	N':= N \setminus \bigcup_{\gamma \nsucceq \sigma } \text{Cl}(\gamma),
	\]
	is open in $N$ as  the union of the closures $\text{Cl}(\gamma)$ is closed.
	Since $v \in N'$, there exists an open ball $B_r := \{ w \in N \mid |v-w| \leq r \}$ such that $B_r$ lies entirely in $\cup_{\tau \succeq \sigma} \tau$, where $|\cdot |$ is the Euclidean distance. Hence $\overline{B_r} \subset \text{star}(\sigma)$. Let $\overline{p} \in N(\sigma)$. We claim that $\overline{p} \in \text{star}(\sigma)$. Let $p \in N$ such that its image in $N(\sigma)$ is $\overline{p}$. If $p \in B_r$, then $\overline{p} \in \overline{B_r} \subset \text{star}(\sigma)$. If $ p \notin B_r$, let $|v-p| := R$. Then $r \leq R$. Hence for $s:= r/2R>0$ and $t := (1-r/2R)>0$, $q:=tv+ sp$ is in $B_r$. So there exists a cone $\tau \succeq \sigma$, such that $q \in \tau$. Then 
	\[
	p = \frac{1}{s} q - \frac{t}{s} v \in \tau + N_\sigma,
	\]
	as $(-t/s) v \in N_\sigma$.
	So $\overline{p} \in \text{star}(\sigma)$. Hence $\text{star}(\sigma)$ is complete. 
\end{proof}

\section{Polytopes} \label{S:Polytopes}

Polytopes provide a natural source of fans. We will define and give some properties of polytopes, and later we will define toric varieties associated to polytopes.

\begin{definition}
	A \demph{polytope} in $M$ is a set of the form
	\[
	P := \conv(S) =\left\{\sum_{a\in S} \lambda_a a \mid \lambda_a \geq 0, \ \sum_{a \in S} \lambda_a = 1 \right\},
	\]
	where $S \in M$ is finite. 
	We say that $P$ is the \demph{convex hull} of $S$. 
	
	If $S \subseteq M_\ZZ$ , i.e., $P$ is the convex hull of lattices in $M$, then $P$ is called a \demph{lattice polytope}. 
\end{definition}

\begin{example}
	The \demph{$n$-simplex} in $\RR^n$ is given by $\Delta_n = \conv\{0,e_1,\ldots,e_n\}.$ The 2-simplex can be seen in Figure \ref{F:2Simplex}.
	
	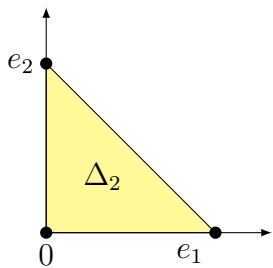
\begin{figure}[!ht]
		\centering
		\begin{tikzpicture}[line cap=round,line join=round,>=latex,scale=1.5]
		\filldraw[color=black,fill=yellow,fill opacity=0.4] (0,0) -- (1.5,0)  -- (0,1.5) -- cycle;
		[line cap=round,line join=round,>=latex,scale=1]
		
		\node at (.5,.5)  {$\Delta_2$};
		
		\draw[->] (0,0) -- (0,2); 
		
		\draw[->] (0,0) -- (2,0);

		\filldraw[very thick] (0,0) circle (.04cm) node[below]{$0$};
		
		\filldraw[very thick] (1.5,0) circle (.04cm)node[anchor=north east] {$e_1$};
		\filldraw[very thick] (0,1.5) circle (.04cm)node[left] {$e_2$};
		\end{tikzpicture}
		
		\caption{The 2-simplex.}
		\label{F:2Simplex}
	\end{figure}
	
	Another polytope in $\RR^2$ is given by the set $ S=\{e_1 + e_2, e_1 - e_2, -e_1 + e_2, -e_1 - e_2\}$, which can be seen in Figure \ref{F:P1P1Polytope}. \hfill$\square$
	
	\begin{figure}[!ht]
		\centering
		\begin{tikzpicture}[line cap=round,line join=round,>=latex,scale=1]
		\filldraw[color=black,fill=yellow,fill opacity=0.4] (-1.5,-1.5) -- (1.5,-1.5)  -- (1.5,1.5) -- (-1.5,1.5)--cycle;
		[line cap=round,line join=round,>=latex,scale=1]
		
		\node at (1,1)  {$P$};
		
		\draw[<->] (0,-2) -- (0,2); 
		
		\draw[<->] (-2,0) -- (2,0);

		\filldraw[very thick] (1.5,0) circle (.04cm)node[anchor=south west] {$e_1$};
		\filldraw[very thick] (0,1.5) circle (.04cm)node[anchor=south east] {$e_2$};
		\filldraw[very thick] (-1.5,0) circle (.04cm)node[anchor=south east] {$-e_1$};
		\filldraw[very thick] (0,-1.5) circle (.04cm)node[anchor=north east] {$-e_2$};
		\end{tikzpicture}
		
		\caption{The polytope $P = \conv \{\pm e_1 \pm e_2 \} \subseteq \RR^2.$}
		\label{F:P1P1Polytope}
	\end{figure}
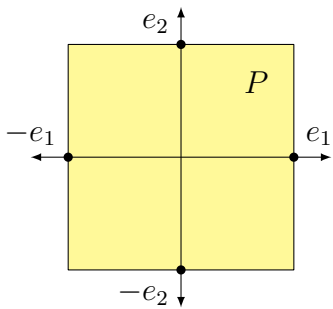
\end{example}

The \demph{dimension} of a polytope $P \in M$ is the dimension of the smallest affine subspace of $M$ containing $P$. 

Given a vector $v \in N$ and a scalar $b$, we get the affine hyperplane $H_{v,b} := \{m \in M\mid \langle m,v \rangle = b \}$ and the closed half-space $H_{v,b}^+:=\{m \in M\mid \langle m,v \rangle \geq b \} $.
A subset $F$ of $P$ is called a \demph{face} of $P$ if there exist $v \in N$ and a scalar $r$ such that $F = H_{v,b} \cap P$ and $P \subseteq H_{v,b}^+$.

Notice that, if we take $u =0$, we get $F = P$. Hence $P$ is always a face of itself. 
A \demph{proper} face of $P$ is a face that is not equal to $P$, and a \demph{facet} of $P$ is a face with codimension $1$. A face of $P$ with dimension 0 is called a \demph{vertex} of $P$. and a face of $P$ with dimension 1 is called an \demph{edge}  of $P$.

\begin{lemma} \label{L:PolytopeFace}
	Let $P =\conv(S)$ be a polytope, and let $F$ be a face of $P$. Then $F$ is the convex hull of $S \cap F$. 
\end{lemma}

\begin{proof}
	Let $F = H_{v,b} \cap P$ for some $v \in N$ and a scalar $b$. Define the subsets $S_0 := \{ a\in S \mid \langle a,v \rangle = b  \}$ and $S_> := \{a\in S \mid \langle a,v \rangle > b  \}$. Since  $P \subseteq H_{v,b}^+$, we have $S = S_0 \cup S_>$. Also note that $S_0 = S \cap F$. We will show $F = \conv(S_0)$.  Let $u \in F$. Since also $u \in P$,  we can write it as a convex combination $u = \sum_{a \in S}\lambda_a a$ where $\lambda_u \geq 0$ and  $\sum_{a \in S} \lambda_a = 1$. 
	Then we have
	\begin{align*}
	b = \langle u,v \rangle & = \biggl \langle \ \sum_{a \in S}\lambda_a a ,v \biggr\rangle\\
	&= \sum_{a \in S} \lambda_a  \langle  a,v \rangle\\
	&= \sum_{a \in S_0} \lambda_a  \langle  a,v \rangle +  \sum_{a \in S_>} \lambda_a  \langle  a,v \rangle\\
	&= b \biggl(\ \sum_{a \in S_0}\lambda_a \biggr)+ \sum_{a \in S_>} \lambda_a  \langle a,v \rangle\\
	& \geq b
	\end{align*}
	and the equality holds if and only if $\lambda_a = 0$ for all $a \in S_>$. In this case 
	$u =\sum_{a \in S_0}\lambda_a a$, where  $\sum_{a \in S_0}\lambda_a =1$. Hence $P = \conv(S_0).$
\end{proof}

We list some immediate results following from Lemma \ref{L:PolytopeFace}. 

\begin{proposition}
	Let $P \subseteq M$ be a polytope. Then
	\begin{enumerate}
		\item A face of $P$ is a polytope.
		\item $P$ has only finitely many faces. 
		\item Any face of a face of $P$ is a face of $P$.
		\item The intersection of any two faces of $P$ is a face of $P$.
		\item Every proper face of $Q$ of $P$ is the intersection of the facets containing $Q$. 
	\end{enumerate}
\end{proposition}

\begin{definition}
	Let $P \subseteq M$ be a polytope. Then the \demph{cone of $P$} is defined by
	\[
	C(P) := \{ \lambda \cdot (v,1) \in M \times \RR \mid v \in P, \ \lambda \geq 0 \}.
	\]
\end{definition}

\begin{example}
	Consider a pentagon $P$ in $\RR^2$, which is a polytope. Figure \ref{F:ConeOfPolytope} illustrates the cone of $P$.  \hfill$\square$
	
	\begin{figure}[!ht]
		\centering
		\begin{tikzpicture}[line cap=round,line join=round,>=latex,scale=1]
		\draw[->] (0,0) -- (0,3); 
		\draw[->] (0,0) -- (3,3); 
		\draw[->] (0,0) -- (1.6,4); 
		\draw[->] (0,0) -- (-3,3); 
		\draw[->] (0,0) -- (-1.6,4); 
		
		\filldraw[color=black,fill=yellow] (0,1) -- (1.5,1.5)  -- (0.8,2) -- (-.8,2)--(-1.5,1.5)--cycle;
		\node at (0.8,1.5)  {$P$};
		\node at (-3,2)  {$C(P)$};
		\draw[->] (0,0) -- (0,3); 
		\draw[->,dashed] (0,0) -- (1.6,4); 
		\draw[->,dashed] (0,0) -- (-1.6,4); 
		\end{tikzpicture}
		
		\caption{The cone of a pentagon P.}
		\label{F:ConeOfPolytope}
		
	\end{figure}
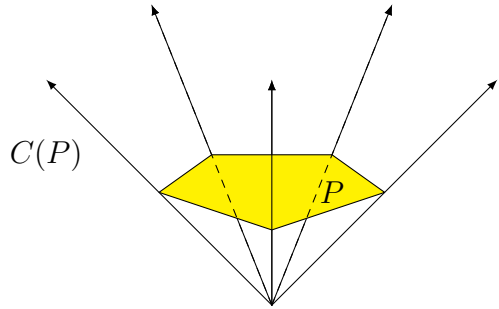
	
\end{example}

Note that $C(P) = \cone(S \times \{1\})$, where $P = \conv(S)$ is a strongly convex cone in $M \times \RR$. The faces of $C(P)$ are the cones over the faces of $P$ with the cone $\{0\}$ corresponding to the empty face of $P$. 

A polytope $P$ can be written as a finite intersection of closed half-spaces \cite[Theorem 1.1]{Ziegler}, i.e.,
\[
P = \bigcap_{i=1}^s H_{v_i,b_i}^+,
\] 
for some $v_1, \ldots,v_s \in N$ and scalars $b_1,\ldots,b_s$. 

When $P$ is full dimensional, that is  $\dim P = \dim M$, its representation as an intersection of closed half-spaces has a nice form because each facet $F$ of $P$ has a unique supporting hyperplane and corresponding closed half-space,
\[
H_F = \{ m \in M \mid \langle m, u_F \rangle = a_F \} \text{ and } H_F^+ =\{ m \in M \mid \langle m, u_F \rangle\geq a_F \},
\]
where $(u_F,a_F)$ is unique up to multiplication. We call $u_F$ an (\emph{inward}) \demph{facet normal} of the facet $F$.  Then one can write 
\[
P = \bigcap_{\mathclap{\substack{F \subseteq P \\
			\text{ is a facet}}}} \ H_F^+.
\]
For simplicity we will assume $P$ is full dimensional and contains $0$ in its interior. 

We have already seen the connection between cones and polytopes. Next we will develop a theory of duality of polytopes similar to cones. 

\begin{definition}
	Let $P \subseteq M$ be a polytope. Then we define the \demph{polar} of $P$ by
	\[
	P^\circ = \{ v \in N \mid \langle m,v \rangle \geq 1 \text{ for all } m\in P \}.
	\]
\end{definition}

\begin{example}
	Consider the polytope $P$ in Figure \ref{F:P1P1Polytope}. The facet normals  of $P$ are $\pm e_1$ and $\pm e_2$. Hence the polar $P^\circ$ is given by the inequalities $\vert x \vert + \vert y \vert \leq 1$ for $x,y \in N$. The polytope $P^\circ$ is pictured in Figure \ref{F:DualP1P1Polytope}. \hfill$\square$
	
	\begin{figure}[!ht]
		\centering
		\begin{tikzpicture}[line cap=round,line join=round,>=latex,scale=1]
		\filldraw[color=black,fill=yellow,fill opacity=0.4] (-1.5,0) -- (0,-1.5)  -- (1.5,0) -- (0,1.5)--cycle;
		[line cap=round,line join=round,>=latex,scale=1]
		
		\node at (1,1)  {$P$};
		
		\draw[<->] (0,-2) -- (0,2); 
		
		\draw[<->] (-2,0) -- (2,0); 
		\filldraw[very thick] (1.5,0) circle (.04cm)node[anchor=south west] {$e_1$};
		\filldraw[very thick] (0,1.5) circle (.04cm)node[anchor=south east] {$e_2$};
		\filldraw[very thick] (-1.5,0) circle (.04cm)node[anchor=south east] {$-e_1$};
		\filldraw[very thick] (0,-1.5) circle (.04cm)node[anchor=north east] {$-e_2$};
		\end{tikzpicture}
		
		\caption{Polar $P^\circ$ of the polytope $P = \conv \{\pm e_1, \pm e_2 \} \subseteq \RR^2$.}
		\label{F:DualP1P1Polytope}
	\end{figure}
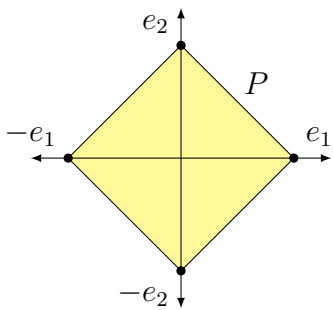
\end{example}

With $C(P)$ the cone of $P$, the dual cone $C(P)^\vee$ consists of vectors $(v,r) \in N \times \RR$ such that $\langle u,v \rangle + r \geq 0$ for all $u \in P$. It follows that $C(P)^\vee$ is the cone of $P^\circ$. This allows us to use the theory we developed in Section \ref{S:Cones}.

\begin{theorem}
	Let $P \subseteq M$ be a polytope. Then
	\begin{enumerate}
		\item $P^\circ$ is a polytope in $N$  and $(P^\circ)^\circ = P$. 
		\item For any face $F$ of $P$, $F^*:= \{v \in P^\circ \mid \langle u,v \rangle =1 \text{ for all } u \in F \}$ is a face of $P^\circ$. The map $F \mapsto F^* $ is one to one and order reversing correspondence between the faces of $P$ and $P^\circ$. 
		\item If $P$ is a lattice polytope, then so is $P^\circ$. 
	\end{enumerate}
	
\end{theorem}

We next construct a fan using a polytope $P$ called the \demph{normal fan} of $P$.

\begin{definition}
	Let $P \subset M$ be a polytope with $\conv(\calA)=P$, and $F$ be a face of $P$. For each face $\calF := F\cap \calA$ of $\calA$, define a set $\sigma_\calF$ by
	\[
	\sigma_\calF := \{v \in N \mid \langle f,v \rangle \leq \langle a ,v \rangle \text{ for all } f \in \calF \text{ and } a \in \calA \}.
	\]
\end{definition}

The following proposition shows that $\sigma_\calF$ is a cone by associating it to a cone over a particular face in the dual polytope. 

\begin{proposition}
	The set $\sigma_\calF$ is a cone in $N$.
\end{proposition}

\begin{proof}
	Define $C:=\RR_\geq \cdot F^*$, which is a cone in $N$. We claim that $\sigma_F = C$. For any element $ v \in C$, we can write $v = \alpha f^*$ for $\alpha \geq 0 $ and $f^*\in F^*$.  In particular $f^* \in P^\circ$, hence for any $a \in \calA$, we have $\langle a, v \rangle = \alpha \langle a,f^*\rangle \geq \alpha$. On the other hand, since $f^* \in F^*$, we have $\langle f,f^* \rangle =1$ for all $f \in \calF$. So we have $\langle f,v \rangle = \alpha \langle f,f^* \rangle = \alpha$. Hence $\langle f,v \rangle \leq \langle a,v \rangle$ for all $f \in \calF$ and $a \in \calA$. Thus $v \in \sigma_\calF$. 
	
	Conversely, suppose $v \in \sigma_\calF$. Since $0$ belongs to the interior of $P$, for all $f \in \calF$ we have $\langle f , v \rangle \leq \langle 0,v \rangle = 0$ with equality if and only if either $f=0$ or $v=0$. Since $F$ is a face of $P$, it must be on the boundary, and $0$ is an interior point of $P$, therefore $u \neq 0$. If $v=0$, then $v  \in C$ and we are done. If $v \neq 0$, then we have $\langle f,v \rangle > 0$. Let $\langle f,v \rangle = \alpha$ for some $\alpha >0$. Then $ v= \alpha (\frac{1}{\alpha} v)$ and $\langle f , \frac{1}{\alpha} v \rangle = 1$, which implies $\frac{1}{\alpha}v \in F^*$. So $v\in C$. 
\end{proof}

The following lemma is an immediate consequence of \cite[Proposition 2.3.7]{CLS}.
\begin{lemma}
	The collection 
	\[
	\Sigma_P \  := \{\sigma_\calF \mid \calF \text{ is a face of } \calA \}
	\] 
	of cones forms a fan in $N$.
\end{lemma}

\begin{definition}\label{D:NormalPolytope}
	The fan $\Sigma_P$ defined above is called (\emph{inner})  \demph{normal fan} of $P =\conv(\calA)$.
\end{definition}

\begin{example} \label{Ex:NormalFanP2}
	Consider the 2-simplex $\Delta_2$ given in Figure \ref{F:2Simplex}. It has three vertices $0^*, e_1^*$, and $e_2^*$. It has three edges $F_1 = \left[0^*,e_1^*\right]$, $F_2 = \left[0^*,e_2^*\right]$, and $F_3 = \left[e_1^*,e_2^*\right]$. 
	Then one has $\sigma_0 = \cone \{e_1,e_2\}, \  \sigma_{e_1} = \cone\{-e_1, -e_1+e_2\}$, and $\sigma_{e_2} = \cone\{e_1-e_2, e_2 \}$. For the edges $\sigma_{F_1} = \cone \{e_2\}, \ \sigma_{F_2} = \cone \{ e_1\}$, and $\sigma_{F_3} = \cone \{ -e_1-e_2\}$. These cones give the fan $\Sigma_2$ in Figure \ref{F:Fans}.\hfill$\square$
\end{example}

\pagebreak{}

\chapter{CLASSICAL TORIC VARIETIES}\label{CH:ClassicalTV}

We will now see different constructions of toric varieties. According to Theorem \ref{Th:EquivalentConstructions} these constructions are equivalent.  We will omit the proofs of some results in this section. The proofs of these results can be found in \cite{Fulton, CLS, Ewald}. 

\section{Affine Toric Varieties}
Let $\CC^\times$ be the group of nonzero complex numbers. A (complex) algebraic torus $\TT$ is an algebraic group isomorphic to $(\CC^\times)^n$ with componentwise multiplication. 

We will need the following useful facts about tori. We refer to \cite[Chapter 16]{humphreys} for proofs.

\begin{proposition}\label{P:Tori}$\ $
	\begin{enumerate}
		\item Let $\TT$ be a torus and $H \subseteq \TT$  be an irreducible subvariety of $\TT$ which is also a subgroup. Then $H$ is a torus. 
		
		\item If a morphism $\psi \colon \TT_1 \to\TT_2$ between two tori is a group homomorphism, then the image of $\psi$ is a closed torus in $\TT_2$.
	\end{enumerate}
	
\end{proposition}
We now introduce two important lattices associated to a torus $\TT$. 

\begin{definition}
	A \demph{character} of  $\TT$ is a morphism $\chi \colon \TT \to \CC^\times$ that is a group homomorphism. 
	A \demph{one parameter subgroup} or \demph{cocharacter} of $\TT$ is a morphism $\lambda \colon \CC^\times \to \TT$ that is a group homomorphism. 
	
\end{definition}

\begin{example}\label{Ex:Characters}
	Let $u = (u_1, \ldots, u_n) \in \ZZ^n$. The homomorphism defined by
	\begin{align*}
	\chi^u\colon \TT &\longrightarrow \CC^\times\\
	(t_1, \ldots ,t_n)  &\longmapsto t^u:=t_1^{u_1} \cdots t_n^{u_n}
	\end{align*} 
	is a group homomorphism. Hence $\chi^u$ is a character of $\TT$. \hfill$\square$
\end{example}

\begin{proposition}
	All characters of $\TT$ arise in the way described in Example \ref{Ex:Characters}.
\end{proposition}

\begin{proof}
	Let $\chi \colon  \TT \to \CC^\times$ be a character of $\TT$.
	Also let $T_i$ be the coordinate subgroup of  $\TT$ with all but the $i$-th coordinate equal to $1$.  Each $T_i$ is isomorphic as a group and variety to $\CC^\times$ under the map $\phi_i(t_1,\ldots,t_n)=t_i$. Let $\chi_i \colon \CC^\times \to \CC^\times$ be the morphism induced by $\chi$ restricted to $T_i$. Since $\chi$ is a group homomorphism, we have
	\[
	\chi = \chi_1 \cdots \chi_n,
	\]
	where each $\chi_i$ is a character of $\CC^\times$. Each $\chi_i$ is a single variable Laurent polynomial of the form $f(t)/g(t)$. Since it must map $\CC^\times$ to $\CC^\times$ and be defined everywhere, the only possible zeros of $f(t)$ and $g(t)$ is 0. Thus, each $\chi_i(t)=t^{u_i}$ for some integer $u_i$. Hence we have $\chi(t_1, \ldots, t_n) = t_1^{u_1} \cdots t_n^{u_n}$ for some $u_1, \ldots, u_n \in \ZZ$. 
\end{proof}

\begin{example}\label{Ex:OneParameterSubgroups}.
	Let $v = (v_1, \ldots,v_n) \in \ZZ^n$. The homomorphism defined by
	\begin{align*}
	\lambda^v \colon \CC^\times &\longrightarrow \TT\\
	t \ &\longmapsto (t^{v_1}, \ldots, t^{v_n})
	\end{align*} 
	is a group homomorphism. Hence $\lambda^v$ is a cocharacter of $\TT$.\hfill$\square$
\end{example}

\begin{proposition}
	All cocharacters of $\TT$ arise in the way described in Example \ref{Ex:OneParameterSubgroups}.
\end{proposition}

\begin{proof}
	Let $\chi_i \colon \TT \to \CC^\times$ be the character given by $\chi_i(t_1,\ldots, t_n) = t_i$ and let $\lambda \colon \CC^\times \to \TT$ be a cocharacter of $\TT$. Then the composition $\chi_i \circ \lambda \colon \CC^\times \to \CC^\times$ is an automorphism of $\CC^\times$. So there exists an integer $v_i$ so that  $\chi_i \circ \lambda (t) = t^{v_i}$. Hence $\lambda^v (t) = (t^{v_1}, \ldots, t^{v_n})$ for some $v_1, \ldots, v_n \in \ZZ$. 
\end{proof}

For an arbitrary torus $\TT \simeq (\CC^\times)^n$, its cocharacters and characters form free abelian groups $N_\ZZ$ and $M_\ZZ$ of rank $n$, respectively.  We say that an element $u \in M_\ZZ$ gives the character $\chi^u$, and an element $v \in N_\ZZ$ gives the cocharacter $\lambda^v$.  There is a natural pairing $\langle \ , \ \rangle \colon M_\ZZ \times N_\ZZ \to \ZZ$. Given a character $\chi^u$ and a cocharacter $\lambda^v$, the composition $\chi^u \circ \lambda^v \colon \CC^\times \to \CC^\times$ is a character of $\CC^\times$, which is given by $t \mapsto t^l$ for some $l \in \ZZ$. Then, we define $\langle u,v\rangle = l$.  We can identify $N_\ZZ$ with $\Hom(M_\ZZ,\ZZ)$ and $M_\ZZ$ with $\Hom(N_\ZZ,\ZZ)$. Hence it is customary to write a  torus as $\TT_N$ to indicate its cocharacters. Note that $\TT_N \simeq N_\ZZ \otimes_\ZZ \CC^\times$.

\begin{definition}\label{D:ATV}
	An \demph{affine toric variety} is an affine variety $V$ that contains $\TT_N$ as a Zariski open subset such that the action of $\TT_N$ on itself extends to an algebraic action of $\TT_N$ on $V$. 
\end{definition}

\begin{example} \label{Ex:AffineTV}
	The basic examples of affine toric varieties are $(\CC^\times)^n$ and $\CC^n$. 
	A less trivial example is the cuspidal cubic curve $C = V(x^3-y^2) \subseteq \CC^2$ with torus 
	$$C \setminus \{0 \}=\{ (t^2,t^3) \mid t \neq 0 \} \simeq \CC^\times.$$
	Note that $C$ is a nonnormal toric variety. 
	Another example of an affine toric variety is the variety $V=V(xy-zw) \subseteq \CC^4$ with the torus 
	$V\cap(\CC^\times)^4 = \{(t_1,t_2,t_3,t_1t_2t_3^{-1} ) \mid t_i \in \CC^\times \} \simeq (\CC^\times)^3.$ \hfill $\square$
\end{example}

We will now construct affine toric varieties from a given finite set $\calA$ in $M_\ZZ$. For $\calA = \{ a_1, \ldots, a_s \} \subseteq M_\ZZ$, every element $a_i \in M_\ZZ$ gives a character $\chi^{a_i} \colon \TT_N \to \CC^\times$. 

\begin{definition}\label{D:AffineTV}
	Let $\calA =\{a_1, \ldots, a_s\}$ be a finite subset in $M_\ZZ$. The \demph{affine toric variety} $Y_\calA$ is the Zariski closure of the image of the monomial map
	\begin{align*}
	\Phi_\calA \colon  \TT_N &\longrightarrow \CC^s \\
	t \ &\longmapsto (\chi^{a_1}(t) , \ldots,\chi^{a_s}(t) ).
	\end{align*}
\end{definition}  

We will justify this definition in Proposition \ref{P:AffineTV}, but first let us look at some examples. 

\begin{example}
	Consider the set $\calA = \{2,3\} \subseteq \ZZ$. Then the monomial map is given by
	\begin{align*}
	\Phi_\calA \colon \CC^\times &\longrightarrow \CC^2\\
	t \ &\longmapsto (t^2, t^3).
	\end{align*}
	The closure of the image of $\Phi_\calA$ is the cuspidal cubic of Example \ref{Ex:AffineTV}. 
	
	For another example, consider the set $\calA = \{(3,0), (2,1), (1,2), (0,3)\} \in \ZZ^2$. Then the monomial map is given by
	\begin{align*}
	\Phi_\calA \colon (\CC^\times)^2 &\longrightarrow \CC^4\\
	(s,t)\ &\longmapsto (s^3, s^2t, st^2, t^3).
	\end{align*}
	The closure of the image of $\Phi_\calA$ is the affine cone over the twisted cubic. \hfill$\square$
\end{example}

\begin{proposition}\label{P:AffineTV}
	$Y_\calA$ is an affine toric variety. 
\end{proposition}

\begin{proof}
	First note that the map $\Phi_\calA$ in Definition \ref{D:AffineTV} can be regarded as a map of tori $\Phi_\calA \colon \TT_N \to (\CC^\times)^s$. Hence, by Proposition \ref{P:Tori} the image $\TT:= \Phi_\calA(\TT_N)$ is a closed torus in $(\CC^\times)^s$. As $Y_\calA = \overline{\TT}$, we have $Y_\calA \cap (\CC^\times)^s = \TT$. Hence $\TT$ is Zariski open in $Y_\calA$. Also, since $\TT$ is irreducible, so is its Zariski closure $Y_\calA$. 
	
	We next examine the action of $\TT$. For every $t \in \TT$, if $W \subset \CC^s$ is a subvariety, then $t \cdot W$ is a variety. Now $ \TT = t \cdot \TT \subseteq t \cdot Y_\calA$. After taking Zariski closure we have $Y_\calA \subseteq t \cdot Y_\calA$. 
	Replacing $t$ with $t^{-1}$ we get an inclusion $Y_\calA \subseteq t^{-1} \cdot Y_\calA$. Hence we get $Y_\calA = t\cdot Y_\calA$. Therefore the action extends and $Y_\calA$ is an affine toric variety.   
\end{proof}

In particular, the torus $\TT$ defined in the previous proof has  character lattice $\ZZ \calA$, the sublattice generated by $\calA$, and so $\dim Y_\calA = \text{rank} (\ZZ\calA)$ \cite[Proposition 1.1.8]{CLS}. 

\section{Toric Ideals}
We will now see a construction of affine toric varieties via toric ideals. We begin with a finite set $\calA =\{a_1, \ldots, a_s\}\subset \ZZ^n$.  Each element $a_i$ in $\calA$ gives a Laurent monomial $t^{a_i} \in \CC \left[t^{\pm1}\right] := \CC\left[t_1^{\pm1}, \ldots,t_n^{\pm1}\right]$. Consider the map
\begin{align} \label{IdealMap}
\Phi_\calA^* \colon \CC\left[y_1,\ldots,y_s\right] &\longrightarrow \CC\left[t^{\pm1}\right] \\
y_i \quad &\longmapsto t^{a_i}. \nonumber
\end{align}

\begin{definition}
	The kernel of the map $\Phi_\calA^*$ in \eqref{IdealMap} is called the \demph{toric ideal} $I_\calA$ associated to $\calA$. 
\end{definition}

\begin{example}\label{E:ToricIdeal}
	Consider the set $\calA = \{ (0,1), (1,1), (1,2)\}$. Then the map $\Phi^*_\calA \colon \CC\left[x,y,z\right] \to \CC\left[s^{\pm1}, t^{\pm1}\right]$ is defined by sending $x \mapsto t$, $y \mapsto st$, and $z \mapsto st^2$. This map has kernel 
	$I_\calA =\langle xz-y^2 \rangle \subset \CC\left[x,y,z\right].$\hfill$\square$
\end{example}

In Example \ref{E:ToricIdeal}, the toric ideal is generated by a set of binomials. This is true for all toric ideals. 

\begin{notation}
	Let $\calA=\{a_1,\ldots a_s\} \subset \ZZ^n$ be a finite set, and $\omega =(\omega_1, \ldots, \omega_s) \in \ZZ^s$. Then $\calA \omega$ denotes the sum
	\[
	\calA \omega \ := \sum_{i=1}^{s} a_i \omega_i.
	\]
\end{notation}
\begin{lemma}\label{L:Binomials}
	The toric ideal $I_\calA$ is spanned as a complex vector space by the binomials
	\[
	\{y^u - y^v \mid u,v \in \NN^s \text{ with } \calA u = \calA v \}.
	\]
\end{lemma}

\begin{proof}
	Let $u,v \in \NN^s$ with $ \calA u = \calA v$. Since $t^{\calA u} = t^{\calA v}$, the binomial $y^u - y^v$ lies in the kernel of $\Phi^*_\calA$.
	We next show that $I_\calA$ is spanned by these binomials. Fix a term order $\prec$ on $\CC\left[t\right]$, and let $f \in I_\calA$. Suppose $f$ cannot be written as a linear combination of binomials. We may assume that the initial term $\text{in}_\prec (f) =c_uy^u$ of $f$ is minimal with respect to $\prec$ among those which cannot be written as a linear combination of binomials.	
	Note that $\Phi^*_\calA (f) = 0$. In particular, $c_u t^{\calA u}$ must cancel during the expansion of $\Phi^*_\calA (f)$. That means $f$ has a summand $g$ with $\text{in}_\prec(g) \prec \text{in}_\prec(f)$ 
	such that $\Phi^*_\calA(c_u y^u) = \Phi^*_\calA(g) $. Set $f' = f - c_uy^u +g$. Note that $\Phi^*(f') = 0,$ i.e., $f' \in I_\calA$ and $\text{in}_\prec (f') \prec \text{in}_\prec (f)$. Since $ \text{in}_\prec (f)$ is minimal, $f'$ has to be $0$, but then this means $f$ is a multiple of a binomial which is a contradiction. 
\end{proof}

Note that the definition of $I_\calA$ gives an injection
\[
\CC[y_1,\ldots, y_s]/ I_\calA \to \CC\left[t^{\pm1}\right].
\]
Since the ring on the right is an integral domain, we deduce that toric ideals are always prime. 

The ideal $I_\calA$ defines an irreducible variety $V(I_\calA) \subset \CC^s$. This variety is indeed a toric variety \cite[Theorem 1.1.17]{CLS}.
\begin{theorem}\label{C:TVIdeal}
	$V(I_\calA)$ is an affine toric variety. 
\end{theorem}

In Theorem \ref{Th:EquivalentConstructions}, we will relate the affine toric variety $V(I_\calA)$ to the affine toric variety $Y_\calA$. 

\section{Toric Varieties Associated to Affine Monoids}

We now connect affine monoids to toric varieties. We first construct an affine toric variety associated to a given affine monoid $S \subset M_\ZZ$. Then we will show the equivalence of the constructions we have seen so far. 

For an affine monoid $S \subset M_\ZZ$, let $\calA = \{a_1, \ldots, a_s\}$ be a finite set in $M_\ZZ$ so that $\NN \calA = S$. First note that $\CC[S] = \CC [\chi^{a_1}, \ldots, \chi^{a_s}]$ is finitely generated, and since $\CC[S] \subset \CC[M_\ZZ]$, it is an integral domain. So it defines an irreducible variety $\Spec (\CC[S])$ over $\CC$. We also have a $\CC$-algebra homomorphism 
\begin{align*}
\pi \ \colon \CC[y_1, \ldots, y_s] & \longrightarrow \CC[M_\ZZ]\\
y_i & \longmapsto \chi^{a_i}.
\end{align*}        
Hence the kernel $\ker(\pi)$ is the toric ideal $I_\calA$. Note that we get an isomorphism 
\[
\CC[y_1,\ldots,y_s]/\ker(\Phi_\calA^*) \simeq \im(\Phi_\calA^*) = \CC[S].
\] Corollary \ref{C:TVIdeal} implies that $\Spec(\CC[S]) \simeq V(I_\calA)$ is an affine toric variety. 

\begin{definition}
	Let $S \subset M_\ZZ$ be an affine monoid. Then the variety $\Spec(\CC[S])$ is called the \demph{affine toric variety} associated to $S$. 
\end{definition}

\begin{example}
	Let $\calA = \{2,3 \} \subseteq \ZZ$ and let $S = \NN \calA$. The kernel of the matrix $A = (2\ 3)$ has basis $(3 \ -2)^T$. Hence $\Spec \CC\left[S\right] \simeq V(x^3-y^2)$ is the cuspidal cubic of Example \ref{Ex:AffineTV}, which we have already seen is an affine toric variety.\hfill$\square$
\end{example}

\begin{example}
	Let $\calA = \{ (1,0,0,1), (0,1,0,1), (0,0,1,-1) \} \subseteq \ZZ^4$ and let $S = \NN \calA$. The kernel of the matrix $A$ whose rows are the elements of $\calA$ is spanned by $(1,1,-1,-1).$ Hence $\Spec \CC\left[S\right] \simeq V(xy-zw)$. We have already seen in Example \ref{Ex:AffineTV} that the variety $V(xy-zw)$ is an affine toric variety. \hfill$\square$
\end{example}

We now state a result \cite[Theorem 1.1.17]{CLS} which tells us that the different constructions of affine toric varieties we have seen so far are equivalent. 

\begin{theorem}\label{Th:EquivalentConstructions} Let $V$ be an affine variety. The following are $\TT_N$ equivariantly isomorphic.
	\begin{enumerate}
		\item V is an affine toric variety.
		\item $V=Y_\calA$ for a finite set $\calA$ in a lattice.
		\item $V$ is an affine variety defined by a toric ideal. 
		\item $V=\Spec(\CC[S])$ for an affine monoid $S$.
	\end{enumerate}
\end{theorem}

We return to a special type of monoid we introduced in Section \ref{S:Monoids}. Recall that, for a rational polyhedral cone $\sigma \in N$, $S_\sigma := \sigma^\vee \cap M_\ZZ$ is an affine monoid.  Hence $U_\sigma:= \Spec(\CC[S_\sigma])$ is an affine toric variety, called the \demph{affine toric variety} associated to $\sigma$. 

\begin{example}
	Consider the cone $\sigma_3$ given in Figure \ref{F:ConesR2}. In Example \ref{Ex:Gordan} we saw that the affine monoid $S_{\sigma_3}$ is generated by the set $\{e_1^*, e_1^*+e_2^*, e_1^*+2e_2^* \}$. The $\CC$-algebra $\CC\left[S_{\sigma_3}\right]$ can be represented as $\CC \left[S_{\sigma_3}\right] = \CC\left[z_1, z_1z_2, z_1z_2^2 \right] $. Hence $U_{\sigma_3} = V(xz-y^2) \subset \CC^3$ is the affine toric variety associated to $\sigma_3$. \hfill$\square$
\end{example}

The torus of the affine toric variety $U_\sigma$ has character lattice $\ZZ S_\sigma \subset M_\ZZ$, where $\ZZ S_\sigma = \{m_1-m_2 \mid m_i \in S_\sigma \}.$ Note that $M_\ZZ / \ZZ S_\sigma $ is torsion free. To see this, let $km \in \ZZ S_\sigma$ for some $k>1$ and $m \in M_\ZZ$. Then $km = m_1 - m_2$ for some $m_1, m_2 \in S_\sigma$. Since both $m_1,m_2 \in \sigma^\vee$, which is a polyhedral cone, we have 
\[
m_1+m_2 = \frac{1}{k} m_1 + \frac{k-1}{k}m_2 \in\sigma^\vee.
\]
Hence $m= (m+m_2) - m_2 \in S_\sigma - S_\sigma = \ZZ S_\sigma$, which proves the claim. If $\sigma$ is strongly convex, then $\sigma^\vee$ is full dimensional. So, $\text{rank} (\ZZ S_\sigma )= \text{rank}(M_\ZZ) = n$. Hence we have $\ZZ S_\sigma = M_\ZZ$, which implies that the torus of $U_\sigma$ is $\TT_N$.

Note that we can represent $\CC[S_\sigma]$ as a coordinate ring in different ways, according to a choice of generators of $S_\sigma$. Different choices provide different representations of $U_\sigma$ in different complex vector spaces. 

\begin{example}
	Let us consider the cone $\sigma =\{0\} \subset \RR^n.$ The dual cone is $\sigma^\vee = (\RR^n)^*$. We can choose different systems of generators of $S_\sigma$. For example both $A_1 =\{\pm e_1^*, \ldots, \pm e_n^*\}$ and $A_2 =\{e_1^*,\ldots,e_n^*,-(e_1^*+\cdots+e_n^*) \}$ generate $S_\sigma$. Let us first consider the set $A_1$. The corresponding monomial algebra is $\CC[t^{\pm1}] \simeq \CC[x_1,\ldots,x_{2n}] / I_\sigma$, where $I_\sigma$ is generated by $\{x_1x_{n+1}-1,x_2x_{n+2}-1, \ldots,   x_nx_{2n}-1\}$. Hence $U_\sigma = V((x_1x_{n+1}-1), \ldots, (x_nx_{2n}-1)).$
	Note that $U_\sigma \cong (\CC^\times)^n$ using the projection $\CC^{2n} \to \CC^n$ on the first $n$ coordinates. 
	
	With the second system of generators $A_2$, $\CC[S_\sigma] = \CC[x_1,\ldots,x_n,x_{n+1}]/I_\sigma$, where $I_\sigma$ is generated by $\{x_1\cdots x_nx_{n+1}-1\}.$ Hence $U_\sigma =V(x_1\cdots x_nx_{n+1}-1)$, which is again a copy of $(\CC^\times)^n$ that lives in $\CC^{n+1}$.\hfill$\square$ 
\end{example}

\section{Toric Varieties Associated to Fans}\label{S:TVFans}

Toric varieties associated to fans are constructed by gluing affine toric varieties associated to cones. We will start with a familiar example of a toric variety to motivate the later constructions. 

\begin{example}\label{Ex:P2}
	Let us denote by $(t_0,t_1,t_2)$ the homogeneous coordinates of the complex projective space $\PP^2:=\CC\PP^2$. It is covered by three coordinate charts:
	\begin{itemize}
		\item $U_0$ corresponding to $t_0 \neq 0$ with affine coordinates $(t_1/t_0, t_2/t_0) = (x,y)$,
		
		\item $U_1$ corresponding to $t_1 \neq 0$ with affine coordinates $(t_0/t_1, t_2/t_1) = (x^{-1},x^{-1}y)$,
		
		\item $U_2$ corresponding to $t _2\neq 0$ with affine coordinates $(t_0/t_2, t_1/t_2) = (y^{-1},xy^{-1})$.
	\end{itemize}
	Next consider the fan $\Sigma:=\Sigma_2$ in Figure \ref{F:Fans}. Its dual fan is given in Figure \ref{F:DualFan}.
	
	\begin{figure}[!ht]
		\centering
		\begin{tikzpicture}[line cap=round,line join=round,>=latex,scale=.8]
		\fill[color=yellow,fill opacity=0.3] (0,0) -- (3,0) -- (3,3) -- (0,3) -- cycle;
		\fill[color=blue,fill opacity=0.3] (0,0) -- (-3,0) -- (-3,3)  -- cycle;
		\fill[color=green,fill opacity=0.3] (0,0) -- (0,-3) -- (3,-3) -- cycle;

		\draw [<->,thick] (-3.28,0) -- (3.28,0);
		\draw [<->,thick] (3.2,-3.2) -- (-3.2,3.2);
		\draw [<->,thick] (0,-3.28) -- (0,3.28);
		
		\filldraw[very thick] (-2.5,2.5) circle (.04cm)node[above] {$  \qquad \quad -e_1^*+e_2^*$};
		\filldraw[very thick] (2.5,-2.5) circle (.04cm)node[right] {$\ e_1^*-e_2^*$};
		\filldraw[very thick] (2.5,0) circle (.04cm)node[below] {$e_1^*$};
		\filldraw[very thick] (-2.5,0) circle (.04cm)node[below] {$-e_1^*$};
		\filldraw[very thick] (0,2.5) circle (.04cm)node[right] {$e_2^*$};
		\filldraw[very thick] (0,-2.5) circle (.04cm)node[right] {$-e_2^*$};
		
		\draw (1.5,1.5) node {$\sigma_0^\vee$};
		\draw (-2,1)  node {$\sigma_1^\vee$};
		\draw (.75,-1.5) node {$\sigma_2^\vee$};
		\draw (0,-3.5) node[below] {$\Sigma_2$};
		\end{tikzpicture}

		\caption{Dual fan of $\Sigma$.}
		\label{F:DualFan}
	\end{figure}
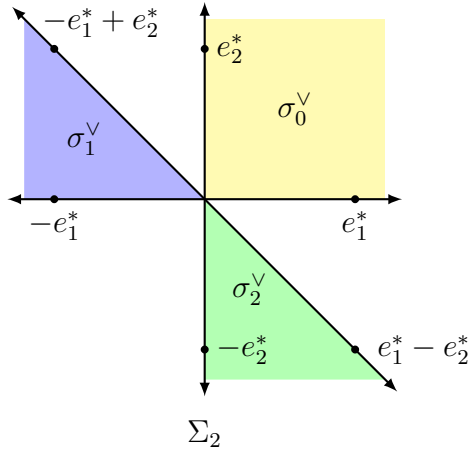

	\begin{itemize}
		\item $S_{\sigma_0}$ is generated by $\{e_1^*,e_2^*\}$, hence $\CC[S_{\sigma_0}]= \CC[x,y]$ and $U_{\sigma_0} = \CC^2_{(x,y)}$.
		
		\item $S_{\sigma_1}$ is generated by $\{-e_1^*,-e_1^*+e_2^*\}$, hence $\CC[S_{\sigma_1}]= \CC[x^{-1},x^{-1}y]$ and $U_{\sigma_1} = \CC^2_{(x^{-1},x^{-1}y)}$.
		
		\item $S_{\sigma_2}$ is generated by $\{-e_2^*,e_1^*-e_2^*\}$, hence $\CC[S_{\sigma_2}]= \CC[y^{-1},xy^{-1}]$ and $U_{\sigma_2} = \CC^2_{(y^{-1},xy^{-1})}$.
	\end{itemize}
	
	We see that three affine toric varieties correspond to the three coordinate charts of $\PP^2$. Indeed, the structure of the fan gives a gluing between these charts allowing to reconstruct the toric variety $\PP^2$ from $U_{\sigma_i}$.  Let us explain what we mean by gluing $U_{\sigma_0}$ and $U_{\sigma_1}$ along $U_{\tau}$, where $ \tau = \sigma_0 \cap \sigma_1$. 
	
	Note that $\tau = H_{-e_1^*} \cap \sigma_0 = H_{e_1^*} \cap \sigma_1$. By Corollary \ref{C:MonoidFace}, we can write $S_\tau = S_{\sigma_0} + \ZZ_\geq(-e_1^*) =  S_{\sigma_1} + \ZZ_\geq(e_1^*)$. The affine toric variety $U_\tau$ is represented by $U_\tau = \CC^\times_{x} \times \CC_y$ in $U_{\sigma_0}$ and $U_\tau = \CC^\times_{x^{-1}} \times \CC_{x^{-1}y}$ in $U_{\sigma_1}$. We can glue $U_{\sigma_0}$ and $U_{\sigma_1}$ along $U_\tau$ using the change of coordinates $(x,y) \mapsto (x^{-1},x^{-1}y)$, and obtain $\PP^2 \setminus \{(0:0:1)\}$. \hfill$\square$
\end{example}

This example is a particular case of the general construction. Let $\tau$ be a face of a cone $\sigma$. Then by Corollary \ref{C:MonoidFace}, $S_\tau = S_\sigma + \ZZ_\geq (-m)$ where $m \in \sigma^\vee \cap M_\ZZ$ and $\tau = H_m \cap \sigma$.  The monoid $S_\tau$ is obtained from $S_\sigma$ by adding one generator $-m$. As $m$ can be chosen to be an element of a generating set $\{a_1, \ldots, a_k\}$ for $S_\sigma$, we may assume $m= a_k$ and denote $a_{k+1} = -m$. Hence, the generating set for $S_\tau$ has one more relation than the generators of $S_\sigma$, namely $a_k + a_{k+1} = 0$. This corresponds the multiplicative relation $t_k t_{k+1}=1$ in $\CC[S_\tau]$ and this is the only supplementary relation we need to obtain $\CC[S_\tau]$ from $\CC[S_\sigma]$. As the generators $t_i$ are the coordinate functions on the affine toric varieties $U_\sigma$ and $U_\tau$, this means that the projection 
\begin{align*}
\CC^{k+1} &\longrightarrow\CC^k\\
(x_1,\ldots,x_k,x_{k+1}) &\longmapsto (x_1,\dots,x_k)
\end{align*}
identifies $U_\tau$ with the open subset of $U_\sigma$ defined by $x_k \neq 0$. Hence we have the following lemma from \cite[Page 225]{Ewald}.

\begin{lemma}\label{L:FaceIdentification}
	There is a natural identification $U_\tau \simeq U_\sigma \setminus (t_k =0)$.
\end{lemma}

For cones $\sigma,\sigma' \in \Sigma$, let $\tau=\sigma \cap \sigma'$ be their common face. Lemma \ref{L:FaceIdentification} allows us to glue together $U_\sigma$ and $U_{\sigma'}$ along their common part $U_\tau$. Let us write $(v_1,\ldots v_l)$ for the coordinates on $U_{\sigma'}$. Then there is an isomorphism $U_\tau \simeq U_{\sigma'} \setminus (v_l = 0)$, and we obtain a gluing map 
\[
\psi_{\sigma,\sigma'} \colon U_\sigma \setminus (u_k = 0) \xrightarrow{\ \ \simeq \ \ } U_\tau \xrightarrow{\ \ \simeq \ \ } U_{\sigma'} \setminus (v_l = 0).
\]

\begin{definition}\label{D:AbstractTV}
	Let $\Sigma$ be a strongly convex rational fan in $N$. Let $Y := \coprod_{\sigma \in \Sigma} U_\sigma$ be the disjoint union of affine toric varieties. Define an equivalence relation $\sim$ on $Y$, where $x \in U_\sigma$ is identified with $x' \in U_{\sigma'}$ if $\psi_{\sigma,\sigma'}(x)=x'$. The resulting space $Y_\Sigma = Y /{\sim}$ is called the (abstract) \demph{toric variety} associated to the fan $\Sigma$. 
\end{definition}

The next theorem \cite[Theorem 3.1.5]{CLS} justifies Definition \ref{D:AbstractTV}.

\begin{theorem}
	Let $\Sigma$ be a strongly convex rational fan in $N$. Then $Y_\Sigma$ is a toric variety. 
\end{theorem}

\begin{example}\label{Ex:Pn}
	The construction of $\PP^2$ in Example \ref{Ex:P2} can be generalized to $\PP^n$ considering the fan $\Sigma \subseteq \RR^n$ generated by all proper subsets of $\{e_1,\ldots,e_n, -(e_1+\cdots+e_n) \}$. Let $\sigma_0$ be the cone generated by $\{e_1,\ldots,e_n\}$, and for $i = 1,\ldots,n$, let $\sigma_i$ be the cone generated by $\{e_1,\ldots,e_{i-1},e_{i+1},\ldots,e_n,-(e_1+\cdots+e_n) \}$. Then the affine toric varieties $U_{\sigma_i}$ are copies of $\CC^n$, corresponding to the classical charts of $\PP^n$ and glued together to obtain $\PP^n$.\hfill$\square$
\end{example}

\begin{example} \label{Ex:P1P1}
	Consider the fan $\Sigma = \Sigma_3$ in $\RR^2$ in Figure \ref{F:Fans}. The duals of the cones 
	$\sigma_0=\cone\{e_1,e_2\},\ \sigma_1=\cone\{-e_1,e_2\},\ \sigma_2=\cone\{-e_1, -e_2\} \text{, and } \sigma_3=\cone\{e_1,-e_2\}$ 
	are 
	$\sigma_0^\vee=\cone\{e_1^*,e_2^*\},\ \sigma_1^\vee=\cone\{-e_1^*,e_2^*\},\ \sigma_2^\vee=\cone\{-e_1^*, -e_2^*\} \text{, and } \sigma_3^\vee=\cone\{e_1^*,-e_2^*\}.$
	These cones give the monoids 
	$S_{\sigma_0}$ generated by $\{e_1^*, e_2^*\}$, 
	$S_{\sigma_1}$ generated by $\{-e_1^*, e_2^*\}$, 
	$S_{\sigma_2}$ generated by $\{-e_1^*,- e_2^*\}$, and
	$S_{\sigma_3}$ generated by $\{e_1^*, -e_2^*\}$.
	The $\CC$-algebras obtained from these monoids are 
	$\CC[S_{\sigma_0}]= \CC[x,y]$,
	$\CC[S_{\sigma_1}]= \CC[x^{-1},y]$, 
	$\CC[S_{\sigma_2}]= \CC[x^{-1},y^{-1}]$, and
	$\CC[S_{\sigma_3}]= \CC[x,y^{-1}]$.
	These $\CC$-algebras give us the affine toric varieties 
	$U_{\sigma_0} = \CC^2_{(x,y)}$,
	$U_{\sigma_1} = \CC^2_{(x^{-1},y)}$,
	$U_{\sigma_2} = \CC^2_{(x^{-1},y^{-1})}$, and
	$U_{\sigma_3} = \CC^2_{(x,y^{-1})}$.
	Consider the common face $\tau = \sigma_0 \cap \sigma_1$ of the cones $\sigma_0$ and $\sigma_1$. Note that, $\tau = H_{-e_1^*} \cap \sigma_0 = H_{e_1^*} \cap \sigma_0$. Considering $\tau$ as a face of $\sigma_0$ gives $U_\tau = \CC^\times_x \times \CC_y$ in $U_{\sigma_0}$, and considering $\tau$ as a face of $\sigma_1$ gives $U_\tau = \CC^\times_{x^{-1}} \times \CC_y$ in $U_{\sigma_1}$. We glue $U_{\sigma_0}$ and $U_{\sigma_1}$ along $U_\tau$ using the change of coordinates $(x,y) \mapsto (x^{-1},y)$, and obtain $\PP^1 \times \CC$ with coordinates $((t_0:t_1),y)$ (where $x = t_0 / t_1$). 
	Similarly, gluing $U_{\sigma_2}$ and $U_{\sigma_3}$ yields $\PP^1 \times \CC$ with coordinates $((t_0:t_1),y^{-1})$. Lastly, gluing these two gives $Y_\Sigma = \PP^1 \times \PP^1$ with coordinates $((t_0:t_1),(s_0:s_1) )$ (where $y = s_0/s_1$).  \hfill$\square$
\end{example}

The nature of the gluing process is compatible under taking products. In Example \ref{Ex:P1P1}, the fan $\Sigma$ can be viewed as the product $\Sigma_1 \times \Sigma_2$ of the fans given in Figure \ref{F:P1}.

\begin{figure}[!th]
	\centering
	\begin{tikzpicture}[line cap=round,line join=round,>=latex,scale=.5]
	\draw[<->,thick] (4,0)--(-4,0);
	\draw (-2,0.3) node {$\sigma_1$};
	\draw [fill] (0,0) circle [radius=0.1]node[anchor= north]{$0$};
	\draw (2,0.3) node {$\sigma_0$};
	\draw (0,-2) node{$\Sigma_1$};
	
	\draw[<->,thick] (9,-4)--(9,4);
	\draw (9.5,-2) node {$\sigma_1$};
	\draw [fill] (9,0) circle [radius=0.1] node[right]{$0$};
	\draw (9.5,2) node {$\sigma_0$};
	\draw (8,4) node{$\Sigma_2$};
	
	\end{tikzpicture}
	\caption{Fans $\Sigma_1 \text{ and } \Sigma_2$ with $Y_{\Sigma_i} \simeq \PP^1$.}
	\label{F:P1}
\end{figure}

We have seen in Example \ref{Ex:Pn} that $Y_{\Sigma_1} \simeq Y_{\Sigma_2} \simeq \PP^1$. Note that we have $Y_\Sigma \simeq Y_{\Sigma_1} \times Y_{\Sigma_2}$. This is true in general. To prove this, we will first study the product $U_{\sigma_1 \times \sigma_2}$ for cones $\sigma_1 \in \Sigma_1$ and $\sigma_2 \in \Sigma_2$. Once we describe these objects, we will be able to glue them to obtain the toric variety associated to the fan $\Sigma_1 \times \Sigma_2$.

\begin{lemma}\label{L:productofcones}
	Let $\sigma_1 \in \Sigma_1$ and $\sigma_2\in \Sigma_2$ be two cones. Then,
	\[
	U_{\sigma _1} \times U_{\sigma _2} \simeq U_{\sigma _1 \times \sigma _2}.
	\]
\end{lemma}

\begin{proof}	
	First note that $\sigma _1^\vee \times \sigma _2^\vee = (\sigma _1 \times \sigma _2)^\vee$. Indeed, for $(u_1,u_2) \in \sigma_1^\vee \times \sigma_2^\vee$, 
	\[
	\langle(u_1,u_2),(v_1,v_2) \rangle = \langle u_1,v_1\rangle +\langle u_2,v_2 \rangle \geq 0
	\]
	for any $(v_1,v_2) \in \sigma_1 \times \sigma_2$. Hence $\sigma _1^\vee \times \sigma _2^\vee \subseteq (\sigma _1 \times \sigma _2)^\vee$. Conversely, if $u=(u_1,u_2)\in (\sigma _1 \times \sigma _2)^\vee$, then for any $(v_1,v_2) \in \sigma_1 \times \sigma_2$, we have $\langle u, (v_1,v_2) \rangle \geq 0.$ In particular, we have $\langle u, (v_1,0) \rangle = \langle u_1,v_1 \rangle \geq 0$. Hence $u_1 \in \sigma_1^\vee$. Similarly one has $u_2 \in \sigma_2^\vee$. Therefore $u \in \sigma _1^\vee \times \sigma _2^\vee$. This proves $\sigma _1^\vee \times \sigma _2^\vee = (\sigma _1 \times \sigma _2)^\vee$. This property implies $S_{\sigma_1 \times \sigma_2} = S_{\sigma_1} \oplus S_{\sigma_2}$ and $\CC[S_{\sigma_1 \times \sigma_2}] = \CC[ S_{\sigma_1}] \otimes_\CC \CC[ S_{\sigma_2}]. $ Thus we get $U_{\sigma _1} \times U_{\sigma _2} \simeq U_{\sigma _1 \times \sigma _2}$ as desired.
\end{proof}

Note that the toric variety $Y_{\Sigma_1 \times \Sigma_2}$ is obtained by gluing affine toric varieties $\{U_{\sigma_1 \times \sigma_2}\}$. On the other hand  $Y_{\Sigma_1} \times Y_{\Sigma_2}$ is obtained by gluing affine toric varieties $\{U_{\sigma_1} \times U_{\sigma_2}\}$. In Lemma \ref{L:productofcones} we showed that these pieces are isomorphic. Hence this implies the following theorem. 

\begin{theorem}\label{T:Product of fans}
	Let $\Sigma_1 \in N_1$ and $\Sigma_2 \in N_2$ be two strongly convex rational fans. Then,
	\[
	Y_{\Sigma _1 \times \Sigma _2} = Y_{\Sigma _1}\times Y_{\Sigma _2}.
	\]
\end{theorem}

We have seen in Section \ref{S:Polytopes} that lattice polytopes have normal fans.	We will now see examples of toric varieties associated to normal fans of polytopes. 

\begin{example}
	Consider the polytope $P$ given in Figure \ref{F:2Simplex}. In Example \ref{Ex:NormalFanP2}, we saw that its normal fan is the fan $\Sigma_2$ in Figure \ref{F:Fans}. By Example \ref{Ex:P2}, the corresponding toric variety $Y_{\Sigma_P}$ is $\PP^2$.  
	\hfill$\square$
\end{example}

\begin{example}
	Consider the polytope $P$ of Figure \ref{F:P1P1Polytope}. Its normal fan is the fan $\Sigma_3$ given in Figure \ref{F:Fans}. Hence by Example \ref{Ex:P1P1}, the corresponding toric variety $Y_{\Sigma_P}$ is $\PP^1 \times \PP^1$.\hfill$\square$
\end{example}

We end this section with a theorem that classifies normal toric varieties. Its proof is given in \cite[Corollary 3.1.8]{CLS}.

\begin{theorem}\label{T:NormalTV}
	Let $Y$ be a normal separated toric variety with torus $\TT_N$. Then there exists a fan $\Sigma \subseteq N$ such that $Y \simeq Y_\Sigma$.
\end{theorem}

\pagebreak{}

\chapter{PROPERTIES OF CLASSICAL TORIC VARIETIES} \label{CH:PropertiesClassicalTV}

In Section \ref{CH:ClassicalTV} we have seen different constructions of toric varieties. We will now study their properties. We will start by defining projective toric varieties. By Theorem \ref{T:ProjectiveTV}, a normal toric variety $Y_\Sigma$ can be embedded into a projective space only if $\Sigma$ is the normal fan of a lattice polytope. Then we study the orbits of the torus action on $Y_\Sigma$. We will associate each such orbit with a cone in the fan. Recall that a toric variety $Y_\Sigma$ is a normal and separated toric variety by Theorem \ref{T:NormalTV}. We will give an equivalence relation between the category of normal toric varieties with toric morphisms and the category of rational fans with maps of fans. Lastly, we show that a toric variety $Y_\Sigma$ is compact in the classical topology if and only if the fan $\Sigma$ is complete.  

\section{Projective Toric Varieties}
We turn our attention to toric varieties as subvarieties of projective space $\PP^n$. We first observe that $\PP^n$ is a toric variety with torus 
\begin{align*}
\TT_{\PP^n} &=  \PP^n \setminus V(x_0 \cdots x_n ) = \left\{\left[ a_0: \ldots :a_n \right] \in \PP^n \mid a_0 \cdots a_n \neq 0 \right\}\\
&= \left\{ \left[1: t_1: \ldots : t_n \right] \in \PP^n \mid t_1, \ldots, t_n \in \CC^\times \right\}\\
&\simeq  (\CC^\times)^n.
\end{align*}
The torus $\TT_{\PP^n}$ acts on $\PP^n$ via coordinatewise multiplication which makes $\PP^n$ a toric variety. 

Let us understand the torus $\TT_{\PP^n}$ as a quotient. As with projective space we have $$\TT_{\PP^n} = (\CC^\times)^{n+1} / \ \CC^\times.$$
Then we have an exact sequence of tori
\begin{equation}\label{Eq:ExactSequenceProjectiveTorus}
1 \longrightarrow \CC^\times \longrightarrow (\CC^\times)^{n+1} \overset{ \pi}  \longrightarrow \TT_{\PP^n} \longrightarrow 1,
\end{equation}
where $\pi(x_0,x_1,\ldots,x_n) = \left[a_0, \ldots, a_n\right] \in \TT_{\PP^n}.$

Applying $\Hom(-, \CC^\times)$ to \eqref{Eq:ExactSequenceProjectiveTorus}, we conclude that the character lattice $\Hom(\TT_{\PP^n}, \CC^\times)$ of $\TT_{\PP^n}$ is  $\calM_\ZZ = \{ (a_0, \ldots, a_n) \in \ZZ^{n+1} \mid \sum_{i=0}^{n} a_i = 0 \}$. Similarly, if we apply $\Hom( \CC^\times, -)$ to \eqref{Eq:ExactSequenceProjectiveTorus}, we deduce that the cocharacter lattice $\Hom(\CC^\times, \TT_{\PP^n})$ of $\TT_{\PP^n}$ is the quotient $\calN_\ZZ = \ZZ^{n+1} / \ \ZZ(1,\ldots,1)$. 

Let $\TT_N$ be a torus with lattices $M_\ZZ$ and $N_\ZZ$. For a finite set $\calA =\{ a_1, \ldots, a_s\} \subseteq M_\ZZ$, we defined the affine toric variety as the Zariski closure of the image of the map 
\begin{align*}
\Phi_\calA \colon \TT_N &\longrightarrow \CC^s\\
t \ &\longmapsto \left(t^{a_1}, \ldots, t^{a_s}\right).
\end{align*}

Now, we can compose the map $\Phi_\calA$ with the homomorphism $\pi: (\CC^\times)^s \to \TT_{\PP^{s-1}}$ by regarding $\Phi_\calA$ as a map to $(\CC^\times)^s$. 

\begin{definition}
	Let $\calA \in M_\ZZ$ be a finite set. The \demph{projective toric variety} $Y_\calA$ is the Zariski closure of the image of the map $\pi \circ \Phi_\calA$ in $\PP^{s-1}$. 
\end{definition}

\begin{proposition}
	$Y_\calA$ is a toric variety. 
\end{proposition}

The proof of this proposition can be done similarly to the proof given in Proposition \ref{P:AffineTV}. We refer to \cite[Proposition 2.1.2]{CLS} for its proof. We will describe the torus of a projective toric variety in Proposition \ref{P:TorusofPTV}, but first we give some examples. 

\begin{example}
	Consider the set $\calA = \{ (d,0), (d-1,1), \ldots, (0,d)\} \subseteq \ZZ^2$. Then the map $\pi \circ \Phi_\calA$ is given by
	\begin{align*}
	\pi \circ \Phi_\calA \colon (\CC^\times)^2&\longrightarrow \PP^{d}\\
	(s,t)&\longmapsto \left[s^d: s^{d-1}t:\ldots:t^d\right].
	\end{align*}
	The toric variety $Y_\calA$ is the rational normal curve in $\PP^d$. \hfill$\square$
\end{example}

\begin{example}
	Consider the triangle given by $n\Delta = \{ (x,y) \in \RR^2 \mid 0 \leq x,y,x+y\leq n \}$.  Let $\calA = n\Delta \cap \ZZ^2$. Then the map $\pi \circ \Phi_\calA$ is given by
	\begin{align*}
	\pi \circ \Phi_\calA \colon (\CC^\times)^2&\longrightarrow \PP^{\binom{n+2}{2}-1}\\
	(s,t)&\longmapsto \left[1:s:t:s^2:st:t^2:\ldots: s^n:s^{n-1}t: \ldots :t^n \right].
	\end{align*}
	The toric variety $Y_\calA$ is the Veronese embedding of $\PP^2$ in $\PP^{\binom{n+2}{2}-1}$.\hfill$\square$
\end{example}

\begin{example}
	Consider the set $\calA = \{ 0,2,3\} \subseteq \ZZ$. Then the map $\pi \circ \Phi_\calA$ is given by
	\begin{align*}
	\pi \circ \Phi_\calA \ \colon \  \CC^\times &\longrightarrow \PP^2\\
	t \ \  &\longmapsto \left[1:t^2:t^3\right].
	\end{align*}
	The toric variety $Y_\calA$ is the cuspidal cubic.\hfill$\square$
\end{example}

For a given set $\calA = \{ a_1,\ldots,a_s\} \subseteq M_\ZZ$, set $\ZZ'\calA \ := \{ \sum_{i=1}^{s} c_i a_i \mid c_i \in \ZZ,  \sum_{i=1}^{s} c_i =0 \}$. Then the following proposition identifies the character lattice of a projective toric variety \cite[Proposition 2.1.6]{CLS}.

\begin{proposition}\label{P:TorusofPTV}
	Let $Y_\calA$ be the projective toric variety corresponding to $\calA \subseteq M_\ZZ$. Then the lattice $\ZZ'\calA$ is the character lattice of the torus of $Y_\calA$. 
\end{proposition}

We end this section with a theorem which classifies projective toric varieties. We refer to \cite[Section VII.3]{Ewald} for its proof. 

\begin{theorem}\label{T:ProjectiveTV}
	Let $\Sigma$ be a fan in $N$. Then the toric variety $Y_\Sigma$ is projective if and only if $\Sigma$ is the normal fan of a full dimensional polytope $P$ in $M$.
\end{theorem}

\section{Torus Orbits}
We will study the action of $\TT_N$ on the toric variety $Y_\Sigma$. The torus $\TT_N$ is a group acting on itself by multiplication. We first describe the action of $\TT_N$ on the affine toric variety $U_\sigma$. Then we will define the torus action on the abstract toric variety $Y_\Sigma$. We will show that there is a bijective correspondence between $\TT_N$-orbits in $Y_\Sigma$ and cones in $\Sigma$. We will end this section with a structure theorem.

Let $\Sigma$ be a fan in $M$, and let $\sigma \in \Sigma$ be a cone. Consider the affine toric variety $U_\sigma = \Spec (\CC\left[S_\sigma\right]) $. By Theorem \ref{Th:EquivalentConstructions}, we may assume $U_\sigma \simeq Y_\calA \subset \CC^s$ for a finite set $\calA = \{a_1,\ldots,a_s\}\in M_\ZZ$. There is a bijective correspondence between the complex points of $Y_\calA$ and the monoid homomorphisms $S_\sigma \to \CC$ as follows: Given a point $p \in U_\sigma$, define a map $\gamma \colon S_\sigma \rightarrow \CC$ by sending $m \in S_\sigma$ to $\chi^m(p) \in \CC$. This gives a monoid homomorphism. 

Conversely, we can construct $p \in Y_\calA$ as follows. Let $p = (\gamma(a_1), \ldots, \gamma(a_s))$. We claim $p \in Y_\calA$. It suffices to show $x^u - x^v$ vanishes at $p$ for all exponents $u,v \in \ZZ^s$ with $\calA u = \calA v$. Since $\gamma$ is a monoid homomorphism, we have
\[
\prod_{i=1}^{s} \gamma(a_i)^{u_i} = \gamma (\calA u) = \gamma (\calA v) =  \prod_{i=1}^{s}\gamma(a_i)^{v_i}.
\]
Hence $p \in Y_\calA$. This gives us a bijection between the points of the affine toric variety $Y_\calA$ and semigroup homomorphism from $S_\sigma$ to $\CC$. We will often refer to $p$ as an element in $U_\sigma$, where we actually mean its image under the isomorphism $Y_\calA \simeq U_\sigma$.

\begin{lemma}
	Let $p \in U_\sigma$ and $\gamma$ be its corresponding monoid homomorphism. For $t \in \TT_N$, the monoid homomorphism corresponding to $t \cdot p$ is $m \mapsto \chi^m(t) \gamma(m)$. 
\end{lemma}

\begin{proof}
	Let $\calA = \{a_1,\ldots,a_s \} \subset M_\ZZ$, and $\sigma= \cone (\calA) $. Hence $U_\sigma = Y_\calA \subseteq \CC^s$. The action of the torus $\TT_N$ on $U_\sigma$ is given by a map $\TT_N \times U_\sigma \to U_\sigma$. Since both sides are affine varieties, it comes from a $\CC$-algebra homomorphism $\CC\left[S_\sigma \right] \to \CC\left[M_\ZZ\right] \otimes \CC\left[S_\sigma\right]$ given by $\chi^m \mapsto \chi^m \otimes \chi^m$, which becomes $\CC \left[ x_1, \ldots,x_s\right] / I_\calA \to \CC \left[ t_1^\pm, \ldots,t_s^\pm\right]/ I_\calA \otimes\CC \left[ y_1, \ldots,y_s\right] / I_\calA$ given by $x_i \mapsto t_iy_i$. So $t \cdot p$ is given by $\left( \chi^{a_1}(t) \gamma(a_1), \ldots, \chi^{a_s}(t) \gamma(a_s) \right)$. Hence the corresponding monoid homomorphism is given by $m \mapsto \chi^m(t) \gamma(m)$. 
\end{proof}

\begin{definition}\label{D:DistinguishedPoint}
	For each cone $\sigma$ in $\Sigma$, the homomorphism $\gamma_\sigma \colon S_\sigma \to \CC$ defined by
	\[
	\gamma_\sigma(m)=
	\begin{cases}
	1 &\text{if } m \in M_\ZZ \cap \sigma^\perp\\
	0 &\text{otherwise}
	\end{cases}
	\]
	for $m \in S_\sigma$, is called the \demph{distinguished homomorphism}. The point $x_\sigma \in U_\sigma$ corresponding to the monoid homomorphism $\gamma_\sigma$ is called the \demph{distinguished point}.
\end{definition}

\begin{example}\label{Ex:DistinguishedPoints}
	Consider the cone $\sigma_3$ in Example \ref{ConesInR2}. The generators of $S_{\sigma_3}$ are $m_1 = e_1^*$, $m_2 = e_1^*+ e_2^*$ and $m_3 = e_1^*+2e_2^*$. Let $\tau_1$ be the face of $\sigma_3$ generated by $2e_1-e_2$. Then $m_1,m_2 \notin \tau_1^\perp$ and $m_3 \in \tau_1^\perp$. Then $\gamma_{\tau_1} (m_1) = \gamma_{\tau_1}(m_2) = 0$, and $\gamma_{\tau_1}(m_3)=1$. The distinguished point is given by $x_{\tau_1} = (\gamma_{\tau_1}(m_1), \gamma_{\tau_1}(m_2), \gamma_{\tau_1}(m_3)) = (0,0,1)$. 
	Similarly, if $\tau_2$ is the face of $\sigma$ generated by $e_2$, then we obtain that $\gamma_{\tau_2}(m_1)=1$ and $\gamma_{\tau_2} (m_2) = \gamma_{\tau_2}(m_3) = 0$. Hence $x_{\tau_2} = (1,0,0)$. 
	Considering $\sigma$ as a face of itself, we get the distinguished point $x_\sigma = (0,0,0)$. 
	Lastly, the origin $0$ is a face of $\sigma$, which will give the distinguished point $x_0 = (1,1,1)$. \hfill$\square$
\end{example}

Note that $\gamma_\sigma$ is a monoid homomorphism, as $\sigma^\vee \cap \sigma^\perp$ is a face of $\sigma^\vee$. Hence, if $m,m' \in S_\sigma$ and $m+m' \in S_\sigma \cap \sigma^\perp$, then by Lemma \ref{SumofElementsinCone} we have $m,m' \in S_\sigma \cap \sigma^\perp$. 

We now begin to explain torus orbits of a toric variety.

\begin{lemma}
	Let $\sigma \in N$ be a cone. Then there is a bijection between the set
	\[
	O_\sigma\ := U_\sigma -  \bigcup_{\tau \prec \sigma} U_\tau 
	\]
	and the set of monoid homomorphisms 
	$$\{\gamma \in \Hom (S_\sigma,\CC) \mid \gamma(m) \neq 0 \text{ if and only if } m \in \sigma^\perp \cap M_\ZZ \}.$$
\end{lemma}

\begin{proof}
	Let $\tau$ be a face of $\sigma$. Since dualizing reverses inclusions, we have $S_\sigma \subset S_\tau$. Then the points of $U_\sigma$ that are not in $U_\tau$ are the homomorphisms $\gamma \in \Hom (S_\sigma,\CC)$ that do not extend to a homomorphism $\tilde{\gamma} \in \Hom (S_\tau,\CC)$. An extension is available unless there exists an element $m \in S_\sigma$ such that $ \gamma(m) =0$ but $m$ is invertible in $S_\tau$. 
	Let  $\gamma \in \Hom (S_\sigma,\CC)$ be a homomorphism such that $\gamma(m) \neq 0 \text{ if and only if } m \in \sigma^\perp \cap M_\ZZ$. For a proper face $\tau = \sigma \cap H_m$, we have $m \in (\sigma^\vee  \setminus \sigma^{\perp}) \cap M_\ZZ$, because $m \in \sigma^\perp \cap M_\ZZ$ implies $\tau = \sigma$. Hence $\gamma$ cannot be extended to a homomorphism in $\Hom (S_\tau,\CC)$.
	
	Now suppose $\gamma \in \Hom (S_\sigma,\CC)$ with $\gamma(m) \neq 0$ for some $m \in (\sigma^\vee  \setminus \sigma^{\perp}) \cap M_\ZZ$. Assume $ \{ a_1, \ldots, a_k \} \subset M_\ZZ$ generates $\sigma^\vee$. Then  $m =c_1a_1 +\cdots + c_k a_k$ for some integers $c_i$ not all zero. Since $\gamma(m) \neq 0$, we must have $\gamma(a_i) \neq 0$ for all $i$ with $c_i \neq 0$. Note there exists $a_j \notin S_{\sigma^\perp}$ with $c_j \neq 0$ (If such $j$ does not exist then $m \in \sigma^\perp \cap M$). But this implies that $\gamma$ extends to $S_\sigma + \ZZ(-a_j) = S_\tau$, where $\tau = \sigma \cap H_{a_j}$ is a proper face of $\sigma$. 
\end{proof}

Let $\gamma \in \Hom(S_\sigma,\CC)$ be a homomorphism that vanishes on $(\sigma^\vee \setminus \sigma^\perp)\cap M_\ZZ$ and is nonzero on $\sigma^\perp \cap M_\ZZ$. For an element $t \in \TT_N$ and $m \in S_\sigma$, $(t \cdot \gamma)(m) = \chi^m(t) \gamma(m)$, which again vanishes on $(\sigma^\vee \setminus \sigma^\perp)\cap M_\ZZ$ and is nonzero on $\sigma^\perp \cap M_\ZZ$. Hence $t\cdot \gamma \in O_\sigma$, which shows that $O_\sigma$ is a $\TT_N$-orbit. Moreover, since $\gamma_\sigma \in O_\sigma$, we have $O_\sigma = \TT_N \cdot \gamma_\sigma$.

\begin{example}
	In Example \ref{Ex:DistinguishedPoints} we have the following orbits. 
	
	$O_\sigma = \{ (0,0,0)\}$ is the orbit through the distinguished point $x_\sigma = (0,0,0)$.
	
	$O_{\tau_1} = \{0\} \times \{ 0\} \times \CC^\times$ is the orbit through the distinguished point $x_{\tau_1} = (0,0,1)$.
	
	$O_{\tau_2} =\CC^\times \times \{0\} \times \{ 0\} $ is the orbit through the distinguished point $x_{\tau_2} = (1,0,0)$.
	
	$O_0 = (\CC^\times)^2$ is the orbit through the distinguished point $x_0 = (1,1,1)$. \hfill$\square$
	
\end{example}

We summarize the results of this section in the following theorem \cite[Theorem 3.2.6]{CLS}.

\begin{theorem}\label{T:OrbitConeCorrespondence}
	Let $\Sigma \subseteq N$ be a fan, and $Y_\Sigma$ be the toric variety corresponding to $\Sigma$. Then 
	\begin{enumerate}
		\item There is a bijective correspondence between cones in $\Sigma$ and $\TT_N$-orbits in $Y_\Sigma$. 
		\item The affine toric variety $U_\sigma$ is the union of $\TT_N$-orbits,
		\[
		U_\sigma = \bigcup_{\tau \preceq \sigma} O_\tau.
		\]
	\end{enumerate}
\end{theorem}

\begin{example}\label{Ex:Orbits}
	Consider the fan $\Sigma = \Sigma_2$ in Figure \ref{F:Fans}. Let us examine the orbits and affine toric varieties contained in $Y_\Sigma \simeq \RR^2$.
	
	It contains three $2$-dimensional cones, three edges, and one vertex. 
	
	\begin{enumerate}
		\item The vertex $0$ corresponds to the orbit $O_0= \{ (x:y:z) \in \PP^2 \mid x,y,z \neq 0 \} \simeq \TT_N \subseteq \PP^2$. Note that $U_0= O_0 \simeq (\CC^\times)^2$, as $0$ has no face other than itself. 
		
		\item The edge $\tau_1$ of $\Sigma$ is generated by $e_1$. The distinguished point in $U_{\tau_1}$ is $(1:0:1)$. Hence, $O_{\tau_1} = \{ (x:0:z) \in \PP^2 \mid x,z \neq 0 \}$ and $U_{\tau_1} = O_0 \cup O_{\tau_1}$. 
		
		\item The edge $\tau_2$ of $\Sigma$ is generated by $e_2$. The distinguished point in $U_{\tau_2}$ is $(1:1:0)$. Hence, $O_{\tau_2} = \{ (x:y:0) \in \PP^2 \mid x,y \neq 0 \}$ and $U_{\tau_2} = O_0 \cup O_{\tau_2}$. 
		
		\item The edge $\tau_3$ of $\Sigma$ is generated by $-e_1-e_2$. The distinguished point in $U_{\tau_3}$ is $(0:1:1)$. Hence, $O_{\tau_3} = \{ (0:y:z) \in \PP^2 \mid y,z \neq 0 \}$ and $U_{\tau_3} = O_0 \cup O_{\tau_3}$. 
		
		\item Consider $\sigma_0$. The distinguished point in $U_{\sigma_0}$ is $(1:0:0)$. Hence, $O_{\sigma_0} =  \{ (0:0:z) \in \PP^2 \mid z \neq 0 \}$ and $U_{\sigma_0} = O_0 \cup O_{\tau_1} \cup O_{\tau_2}$.
		
		\item Consider $\sigma_1$. The distinguished point in $U_{\sigma_0}$ is $(0:1:0)$. Hence, $O_{\sigma_1} =  \{ (0:y:0) \in \PP^2 \mid z \neq 0 \}$ and  $U_{\sigma_1} = O_0 \cup O_{\tau_2} \cup O_{\tau_3}$.
		
		\item Lastly, consider $\sigma_2$. The distinguished point in $U_{\sigma_0}$ is $(0:0:1)$. Hence, $O_{\sigma_2} =  \{ (0:0:z) \in \PP^2 \mid z \neq 0 \}$ and  $U_{\sigma_2} = O_o \cup O_{\tau_1} \cup O_{\tau_3}$. \hfill$\square$
	\end{enumerate}
	
\end{example}

\section{Toric Morphisms}
We will define toric morphisms between toric varieties and maps between fans. Then we will show a toric morphism gives rise to a map of fans, and a map of fans gives rise to a toric morphism. This will give an equivalence between two categories.

\begin{definition}
	Let $\Sigma_1$ and $\Sigma_2$ be two fans in real vector spaces $N_1$ and $N_2$, respectively. 
	
	A $\ZZ$-linear mapping $\bar{\phi} \colon (N_1)_\ZZ \to (N_2)_\ZZ$ with induced map $\bar{\phi}_\RR(z \otimes r) = \bar{\phi}(z) \otimes r$ is called a \demph{map of fans} if for every cone $\sigma_1 \in \Sigma_1$, there exists a cone $\sigma_2 \in \Sigma_2$ such that $\bar{\phi}(\sigma_1) \subseteq \sigma_2$. 
	
	A morphism $ \phi \colon Y_{\Sigma_1}  \to Y_{\Sigma_2}$ is called \demph{toric} if $\phi(\TT_{N_1}) \subseteq \TT_{N_2}$, and $\restr{\phi}{\TT_{N_1}}$ is a group homomorphism. 
\end{definition}

\begin{example}
	Let $N_1 = \RR^2$ with basis $e_1$ and $e_2$. For $r \in \NN$, let $\Sigma_r$ be the union of four cones
	$\sigma_1 = \cone \{e_1,e_2\}, \ \sigma_2 = \cone \{e_1, -e_2\}, \ \sigma_3= \cone\{-e_1+re_2, - e_2 \}, \ \sigma_4 = \cone \{-e_1 +re_2, e_2\}$,
	and their faces, shown in Figure \ref{F:Hirzebruch}. 
	\begin{figure}[!ht]
		\centering
		
		\begin{tikzpicture}[line cap=round,line join=round,>=latex,scale=.9]
		\fill[color=yellow,fill opacity=0.3] (0,0) -- (3,0) -- (3,3) -- (0,3) -- cycle;
		\fill[color=blue,fill opacity=0.3] (0,0) -- (0,3) -- (-2,3)  -- cycle;
		\fill[color=red,fill opacity=0.3] (0,0) -- (-2,3)  --(-3,3)--(-3,-3)--(0,-3)-- cycle;
		\fill[color=green,fill opacity=0.3] (0,0) -- (0,-3) -- (3,-3) -- (3,0) -- cycle;

		\draw [->,thick] (0,0) -- (3.5,0);
		\draw [->,thick] (0,0) -- (-2.33,3.5);
		\draw [->,thick] (0,0) -- (0,3.5);
		\draw [->,thick] (0,0) -- (0,-3.5);
		
		\filldraw[very thick] (-1.4,2.1) circle (.04cm)node[left] {\footnotesize$(-1,r)$};
		
		\draw (1.5,1.5) node {$\sigma_1$};
		\draw (1,-1.5) node {$\sigma_2$};
		\draw (-1.5,0)  node {$\sigma_3$};
		\draw (-.5,2)  node {$\sigma_4$};
		
		\end{tikzpicture}
		
		\caption{The fan $\Sigma_r$.}
		\label{F:Hirzebruch}
	\end{figure}
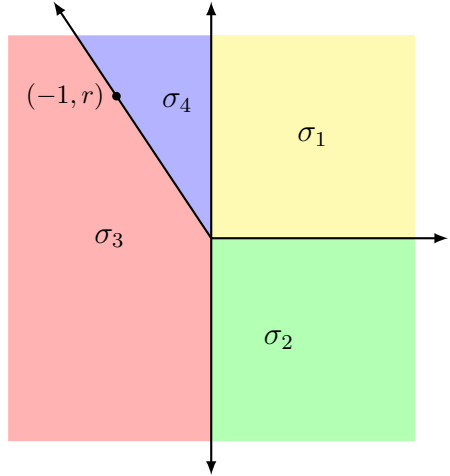

	Also, let $N_2 = \RR$, and let $\Sigma$ be the fan given in Figure \ref{F:P1}.
	Consider the mapping
	\begin{align*}
	\bar{\phi} \colon (N_1)_\ZZ &\longrightarrow \ (N_2)_\ZZ \\
	ae_1+be_2 &\longmapsto a.
	\end{align*}
	Note that $\bar{\phi}(\sigma_1) =\bar{\phi}(\sigma_2) = \bar{\phi}(\cone\{e_1\}) = \bar{\phi}(\cone\{-e_1+re_2\}) =\RR_\geq \in \Sigma$,  $\bar{\phi}(\sigma_3) =\bar{\phi}(\sigma_4) = \RR_\leq \in \Sigma$ and 
	$\bar{\phi}(\cone\{e_2\}) = \bar{\phi}(\cone\{-e_2\}) =\bar{\phi}(0) = 0 \in \Sigma$.
	Hence $\bar{\phi}$ is a map of fans between $\Sigma_r$ and $\Sigma$. \hfill$\square$
\end{example}

\begin{lemma}\label{L:ActionisEquivariant}
	Let $\sigma_1\subset N_1$ and $\sigma_2 \subset N_2$ be cones. A toric morphism $\phi \colon U_{\sigma_1} \to U_{\sigma_2}$ is equivariant, i.e., $\phi(t\cdot p) = \phi(t)  \cdot \phi(p)$ for all $t \in \TT_{N_1}$ and  $p \in V_1$. 
\end{lemma}

\begin{proof}
	Suppose the action of $\TT_{N_i}$ on $U_{\sigma_i}$ is given by a morphism $\Phi_i$. We need to show that the following diagram
	
	\begin{equation*}
	\begin{tikzpicture}[node distance=2.5cm, auto]
	\node (LT) {$\TT_{N_1} \times U_{\sigma_1} $};
	\node (RT) [node distance=5cm, right of=LT]{$U_{\sigma_1}$};
	
	\node (LD) [below of=LT] {$\TT_{N_2} \times U_{\sigma_2}$};
	\node (RD) [node distance=5cm, right of=LD] {$U_{\sigma_2}$};
	\draw[->] (LT) to node {$\Phi_1$} (RT);
	\draw[->] (LD) to node {$\Phi_2$} (RD);
	\draw[->] (LT) to node [swap] {$\restr{\phi}{\TT_{N_1}}\times \phi$} (LD);
	\draw[->] (RT) to node  {$\phi$} (RD);
	\end{tikzpicture}	
	\end{equation*}
	commutes. Replacing $U_{\sigma_i}$ by $\TT_{N_i}$ in the above diagram, we get
	
	\begin{equation*}
	\begin{tikzpicture}[node distance=2.5cm, auto]
	\node (LT) {$\TT_{N_1} \times \TT_{N_1} $};
	\node (RT) [node distance=5cm, right of=LT]{$\TT_{N_1}$};
	
	\node (LD) [below of=LT] {$\TT_{N_2} \times \TT_{N_2}$};
	\node (RD) [node distance=5cm, right of=LD] {$\TT_{N_2}.$};
	\draw[->] (LT) to node {$\Phi_1$} (RT);
	\draw[->] (LD) to node {$\Phi_2$} (RD);
	\draw[->] (LT) to node [swap] {$\restr{\phi}{\TT_{N_1}}\times \restr{\phi}{\TT_{N_1}}$} (LD);
	\draw[->] (RT) to node  {$\phi$} (RD);
	\end{tikzpicture}	
	\end{equation*}
	Since $\restr{\phi}{\TT_{N_1}}$ is a group homomorphism, this diagram commutes. Also note that since $\TT_{N_1}\times \TT_{N_1}$ is dense in $\TT_{N_1} \times U_{\sigma_1}$, the first diagram also commutes.
	Hence $\phi$ is equivariant.
\end{proof}
Lemma  \ref{L:ActionisEquivariant} generalizes to show that any toric morphism $\phi \colon Y_{\Sigma_1} \to Y_{\Sigma_2}$ is an equivariant mapping. 
We want to state a theorem which gives a correspondence between maps of fans and toric morphisms. We start with the following lemma.

\begin{lemma}\label{L:Morphisms}
	Let $\sigma_1 \subset N_1$ and $\sigma_2 \subset N_2$ be two cones. A homomorphism $\bar{\phi} \colon (N_1)_\ZZ \to (N_2)_\ZZ$ induces a toric morphism $\phi \colon U_{\sigma_1} \to U_{\sigma_2}$ extending $\phi \colon \TT_{N_1} \to \TT_{N_2}$ if and only if $\bar{\phi}_\RR(\sigma_1) \subset \sigma_2$. 
\end{lemma}

\begin{proof}
	Let $\bar{\phi}^*  \colon (M_2)_\ZZ \to (M_1)_\ZZ$ be the map dual to $\bar{\phi}$. Note we have $\phi = \Spec \bar{\phi}^*$ when $\bar{\phi}^*$ is considered as a map between coordinate rings $\CC\left[(M_2)_\ZZ\right] \to \CC\left[(M_1)_\ZZ\right]$. Also $\bar{\phi}_\RR(\sigma_1) \subseteq \sigma_2$ if and only if  $\bar{\phi}^\vee_\RR(\sigma_2) \subseteq \sigma_1$, which is exactly when $\restr{\bar{\phi}^\vee}{\sigma_2^\vee} \colon \sigma_2^\vee \to \sigma_1^\vee$ is a monoid homomorphism. Hence $\phi$ is a toric morphism if and only if $\bar{\phi}_\RR(\sigma_1) \subseteq \sigma_2$. 
\end{proof}

We now state the main theorem of this section. A toric morphism gives rise to a map of fans and vice versa. 
\begin{theorem}\label{T:Equivalance}
	Let $\Sigma_i \subseteq N_i$ be fans for $i=1,2$. 
	
	\begin{enumerate}
		\item  If $\bar{\phi} \colon N_1 \to N_2$ is a map of fans between $\Sigma_1$ and $\Sigma_2$, then there exists a toric morphism $\phi \colon Y_{\Sigma_1} \to Y_{\sigma_2}$ such that $\restr{\phi}{\TT_{N_1}}= \bar{\phi} \otimes1 \colon N_1 \otimes\CC^\times \to N_2 \otimes \CC^\times. $
		
		\item If $\phi \colon Y_{\Sigma_1} \to Y_{\Sigma_2}$ is a toric morphism, then $\phi$ induces a map of fans $\bar{\phi} \colon (N_1)_\ZZ \to (N_2)_\ZZ$.
	\end{enumerate}
\end{theorem}

\begin{proof} 
	Take the open cover $\left\{ U_{\sigma_i} \right\}_{\sigma_i \in \Sigma_1}$ of $Y_{\Sigma_1}$. For each $\sigma_i \in \Sigma_1$, there exist a cone $\sigma_i' \in \Sigma_2$ such that $\bar{\phi}(\sigma_i) \subseteq \sigma_i'$.
	By Lemma \ref{L:Morphisms}, we have toric morphisms $\phi_{\sigma_i} \colon U_{\sigma_i} \to U_{\sigma_i'}$. 
	Now let $\sigma_1 ,\sigma_1' \in \Sigma_1$ be two cones, and let $\tau \in \Sigma_1$ be their intersection. Then we have $U_{\sigma_1} \cap U_{\sigma_1'} = U_\tau$. Hence for $u \in \tau \cap (N_1)_\ZZ$ and $c\in \CC^\times$, i.e.,  $u \otimes z \in (N_1)_\ZZ \otimes \CC = \TT_{N_1}$, we have $\restr{\phi_{\sigma_1}}{U_\tau} (u \otimes c)= \bar{\phi}(u) \otimes c = \restr{\phi_{\sigma_1'}}{U_\tau}(u \otimes c)$.
	If we take $\sigma_1 = \{0\}$, then $U_{\{0\}} = \TT_{N_1}$, and we get $\phi_{\{0\}} \colon \TT_{N_1} \to \TT_{N_2}$, which is a group homomorphism by definition. Hence $\phi \colon Y_{\Sigma_1} \to Y_{\Sigma_2}$ is a toric morphism. This proves the first part.

	For the second part, let $v \in (N_1)_\ZZ$, and consider the cocharacter $\lambda^v \colon \CC^\times \to  \TT_{N_1}$. Since $\restr{\phi}{\TT_{N_1}}$ is a group homomorphism, the composition $\restr{\phi}{\TT_{N_1}} \circ \lambda^v \colon \CC^\times \to \TT_{N_2}$ gives an element $\bar{\phi}(v) \in N_2$ and linearity is preserved. Hence $\bar{\phi} \colon N_1 \to N_2$ gives a map $\bar{\phi} \colon (N_1)_\ZZ \to (N_2)_\ZZ$. We claim $\bar{\phi}$ is a map of fans. Since a toric morphism is equivariant, $\TT_{N_1}$-orbits are mapped into $\TT_{N_2}$-orbits. Hence by Theorem \ref{T:OrbitConeCorrespondence} cones in $\Sigma_1$ have to be mapped into cones in $\Sigma_2$. 
\end{proof}

\section{Limits of One Parameter Subgroups and Compactness}\label{S:LimitsOPSGTV}

In this section we first show a way to recover the fan $\Sigma$ from the toric variety $Y_\Sigma$. The key idea is to look at \emph{limits} $\lim\limits_{s \to 0} \lambda^v(s)$ for various $v \in N$.  Next we show that a toric variety $Y_\Sigma$ is compact if and only if the fan $\Sigma$ is complete.

The following lemma \cite[Proposition 3.2.2]{CLS} shows that the limits $\lim\limits_{t \to 0} \lambda^v(t)$ of one parameter subgroups are exactly the distinguished points for the cones in the fan. 
\begin{lemma} \label{L:LimitsofOPSG}
	Let $\sigma \subseteq N$ be a cone and $v \in N_\ZZ$. Then 
	\[
	v \in \sigma \iff \lim\limits_{s \to 0} \lambda^v(s) \text{ exists in } U_\sigma.
	\] 
	Moreover, if $v \in \Relint(\sigma)$, then $\lim\limits_{s \to 0} \lambda^v(s)$ is the distinguished point $x_\sigma$. 
	
\end{lemma}

We next illustrate Lemma \ref{L:LimitsofOPSG} in an example. 

\begin{example}\label{Ex:RecoveringFansTV}
	Consider the fan $\Sigma= \Sigma_3$ of Figure \ref{F:Fans}. We showed in Example \ref{Ex:P2} that $Y_\Sigma = \PP^2$. The torus $\TT_{\PP^2} = \{ (1,s,t) \in \PP^2 \mid s,t \neq 0\} = (\CC^\times)^2 \subseteq \PP^2$. Let $v = (a,b)$ be a point in $N_\ZZ = \ZZ^2$. Then we have $\lambda^v(s) = (1,s^a,s^b) \in \TT_{\PP^2}$. We will analyze $(1,s^a,s^b)$ as $s \to 0$. The reader can compare these limits with Example \ref{Ex:Orbits}.
	\begin{enumerate}
		\item $a,b = 0$: $\lim\limits_{s\to 0} (1,s^a, s^b) = (1,1,1)$. 
		\item $a>0$, $b=0$: $\lim\limits_{s\to 0} (1,s^a, s^b) = (1,0,1)$. 
		\item $a=0$, $b>0$: $\lim\limits_{s\to 0} (1,s^a, s^b) = (1,1,0)$.
		\item $a=b<0$: $\lim\limits_{s\to 0} (1,s^a, s^b) = \lim\limits_{s\to 0} (1,s^a, s^a)= \lim\limits_{s\to 0} (s^{-a},1, 1) = (0,1,1)$.
		\item $a, b>0$: $\lim\limits_{s\to 0} (1,s^a, s^b) = (1,0,0)$.
		\item $a<0$, $a-b <0$ : $\lim\limits_{s\to 0} (1,s^a, s^b) = \lim\limits_{s\to 0} (s^{-a},1, s^{b-a} )= (0,1,0)$.
		\item $a>0$, $a-b >0$ : $\lim\limits_{s\to 0} (1,s^a, s^b) = \lim\limits_{s\to 0} (s^{-b}, s^{a-b},1)= (0,0,1)$.
	\end{enumerate} 
	The regions explained above corresponds to cones of the fan $\Sigma$. This way we can recover the fan from these limit points. \hfill$\square$
\end{example}

We now state the main theorem of this section.

\begin{theorem}\label{T:CompactTV}
	Let $Y_\Sigma$ be a toric variety. Then the following are equivalent. 
	\begin{enumerate}
		\item $Y_\Sigma$ is compact. 
		\item The limit $\lim\limits_{s \to 0} \lambda^v(s)$ exists in $Y_\Sigma$ for all $v \in N_\ZZ$.
		\item $\Sigma$ is complete. 
	\end{enumerate}
\end{theorem}

\begin{proof}
	Assume $Y_\Sigma$ is compact and let $v \in N_\ZZ$. For a sequence $s_i$ in $\CC^\times$ converging to $0$, we get the sequence $\lambda^v(s_i) \in Y_\Sigma$. Since $Y_\Sigma$ is compact, this sequence has a convergent subsequence. We may assume $\lim\limits_{i \to \infty} \lambda^v(s_i) = p \in Y_\Sigma$.  Since $Y_\Sigma$ is covered by affine open subsets, there exists a cone $\sigma \in \Sigma$ such that $p \in U_\sigma$. Pick $m \in S_\sigma$. Since the character $\chi^m$ is continuous, we have 
	\[
	\chi^m(\gamma) = \lim\limits_{i \to \infty} \chi^m (\lambda^v(s_i)) = \lim\limits_{i \to \infty}  s_i^{\langle m,v \rangle}.
	\]
	Since $s_i \to 0$, for the last limit to exist we must have $\langle m,v \rangle \geq0$ for all $m \in S_\sigma$. Hence we have $v \in \sigma$. Then by Lemma \ref{L:LimitsofOPSG}, the limit $\lim\limits_{s \to 0} \lambda^v(s)$ exists in $U_\sigma$.
	
	Now assume that the limit $\lim\limits_{s \to 0} \lambda^v(s)$ exists in $Y_\Sigma$ for all $v \in N_\ZZ$. Pick any $v \in N_\ZZ$. The limit $\lim\limits_{s \to 0} \lambda^v(s)$ exists in $Y_\Sigma$, hence it is an element of $U_\sigma$ for some $\sigma \in \Sigma$ . Again by Lemma \ref{L:LimitsofOPSG}, $v\in \sigma$. Hence $\Sigma$ is complete. 
	
	Now assume $\Sigma$ is complete. We want to show that $Y_\Sigma$ is compact. We will apply induction on $n=\dim N$. When $n =1$, the only complete fan is given in Figure \ref{F:P1}. In this case we have $Y_\Sigma = \PP^1$, which is compact. Now assume that the statement is true for all complete fans of dimension strictly less than $n$, and consider a complete fan $\Sigma$ in $N$. 
	Let $\{\gamma_i \mid i \in \NN\}$ be an infinite sequence in $Y_\Sigma$. Since $Y_\Sigma$ is a union of finitely many orbits $O_\sigma$, we may assume $\{\gamma_i \mid i \in \NN\}$ lies in an orbit $O_\sigma$. If $\sigma \neq \{0\}$ then the closure $\overline{O_\sigma}$ of $O_\sigma$ in $Y_\Sigma$ is isomorphic to $Y_{\text{star}(\sigma)} $ which has dimension $\leq n-1$ \cite[Proposition 3.2.7]{CLS}. Since $\Sigma$ is strongly convex and complete, by Lemma \ref{L:StarStronglyConvex}  $\text{star}(\sigma)$ is also strongly convex and complete. Hence by induction $\{\gamma_i \mid i \in \NN\}$ has a convergent subsequence in $Y_\Sigma$. 
	
	Now assume the sequence $\{\gamma_i \mid i \in \NN\}$ lies entirely in $O_0=\TT_N$. Define a map $L \colon \TT_N \rightarrow N$ as follows. A homomorphism 
	$\gamma \in \TT_N=\Hom (M_\ZZ, \CC^\times) $ is mapped to $L(\gamma) \in \Hom(M_\ZZ, \RR)$, where $L(\gamma)(m) = -\log |\gamma(m)|$.  
	
	Apply $L$ to $\{\gamma_i\mid i \in \NN\}$  to get a sequence $\{L(\gamma_i) \mid i \in \NN\}$  in $N$. Since $\Sigma$ is complete, there exists a cone $\sigma$ in $\Sigma$ such that the intersection $\{L(\gamma_i) \mid i \in \NN\} \cap \sigma$ is infinite. Passing to a subsequence, if necessary, we may assume $\{L(\gamma_i) \mid i \in \NN\}$ lies entirely in $\sigma$. Then for any $m \in S_\sigma$, since $L(\gamma_i) \in \sigma$, we have 
	\[
	\log|\gamma_i(m)| =-\langle m, L(\gamma_i) \rangle \leq 0
	\]
	for all $i \in \NN$.
	Hence $\{\gamma_i\mid i \in \NN\}$ is a sequence of mappings to the closed disk in $\CC$. The monoid $S_\sigma$ is generated by a finite set $\{m_1,\ldots,m_s \}$, and since the closed disk is compact, $\{\gamma_i(m_j)\mid i \in \NN\}$ has a convergent subsequence in $\CC$ for all $j$. Hence $\{\gamma_i\mid i \in \NN\}$ has a subsequence converges to a homomorphism $\gamma \in \Hom(S_\sigma, \CC)$.
	Hence $Y_\Sigma$ is compact. 
\end{proof}

\pagebreak{}

\chapter{IRRATIONAL TORIC VARIETIES}\label{CH:ITV}
In Section \ref{CH:ClassicalTV} we constructed toric varieties corresponding to rational fans. We now develop a theory of irrational toric varieties associated to fans in a finite dimensional real vector space that are not necessarily rational.

\section{Irrational Affine Toric Varieties}
Let $M$ and $N$ be finite dimensional dual real vector spaces. The vector space $N$ is the torus for our irrational toric varieties. When $N$ acts on a space, we will write $T_N$, and use a multiplicative notation for its group operation. 
We identify $T_N$ with the group $\Hom_c(M,\RR_>)$ of continuous homomorphisms from $M$ to $\RR_>$. Here, an element $v \in N$ is sent to the homomorphism $\gamma_v$ whose value at $u \in M$ is $\exp(- \langle u,v\rangle)$.
When $t = \gamma_v$, we write $t^u$ for $\gamma_v(u)$. Elements of $M$ are called \demph{characters} of $T_N$, and elements of $N$ are called \demph{cocharacters} of $T_N$. For a linear subspace $L$ of $N$, we write $T_L$ for the corresponding subgroup of $T_N$.
\begin{notation}
	Let $\calA$ be a finite subset of $M$. We will use $\calA$ as an index set, so that $\RR^\calA := \RR^{|\calA|},\ \RR_>^\calA := \RR_>^{|\calA|}$, and $\RR_\geq^\calA := \RR_\geq^{\vert \calA \vert}$ is the set of $\vert \calA \vert$-tuples of real numbers, positive real numbers, and nonnegative real numbers whose coordinates are indexed by elements of $\calA$, respectively. 
\end{notation}

\begin{definition}
	Let $\calA$ be a finite subset of $M$, and define a map given by
	\begin{align*}
	\varphi_\calA \colon T_N &\longrightarrow \RR_\geq^\calA\\
	t &\longmapsto (t^a \mid a \in \calA ).
	\end{align*}
	The closure $X_\calA$ of the image $X_\calA^\circ$ of $T_N$ under $\varphi_\calA$
	is called an \demph{irrational affine toric variety}.
\end{definition}
The map $\varphi(\calA)$ is a group homomorphism from $T_N$ to $\RR_>^\calA$. Note that $X_\calA$ inherits a $T_N$-action from the homomorphism $\varphi_\calA$.

\begin{definition}
	Let $\calA$ be a subset of $M$. The \demph{annihilator} of $\calA$ is a set
	\[
	\calA^\perp \ :=  \{ v \in N \mid \langle a,v \rangle = 0 \text{ for all } a \in \calA \}.
	\]  
\end{definition}

The kernel of $\varphi_\calA$ is $T_{\calA^\perp}$. Hence $X_\calA^\circ$ is identified with the quotient $T_N / T_{\calA^\perp} (\simeq N / \calA^\perp).$ 
In Lemma \ref{L:Binomials}, we saw that a complex affine toric variety satisfies some binomial equations. This is true for irrational affine toric varieties. 

\begin{proposition}\label{P:IATV}
	The irrational affine toric variety $X_\calA$ is the set of points 
	\begin{equation}\label{Eq:Binomials}
	\mathcal{Z}_A :=  \{ z\in\RR^\calA_\geq \ \mid \ z^u = z^v \text{ for all } u,v\in\RR^\calA_\geq \text{ with }\calA u = \calA v \}.
	\end{equation}
\end{proposition}

\begin{proof}
	Let $z = (t^a \mid a\in \calA)$. For any $u,v \in \RR^\calA_\geq$ with $\calA u = \calA v$
	\[
	z^u = \prod_{a \in \calA} (t^a)^{u_a} = t^{\calA u} = t^{\calA v} = \prod_{a \in \calA} (t^a)^{v_a} = z^v.
	\]
	Hence, $X_\calA \subset \mathcal{Z}_A$.
	
	To show the reverse inclusion, first let $z \in \mathcal{Z}_\calA \cap \ \RR_>^\calA$. Let $\mathcal{B} \subset \calA$ be a basis for the linear span of $\calA$.  Any element $m \in \calA$ can be represented as the sum
	\[
	m = \sum_{b \in \mathcal{B} } \beta_{m,b} \ b \ ,
	\]
	where $\beta_{m,b} \in \RR$ for all $b \in \calB$.
	Let $\calB' := \{b\in \calB \mid \beta_{m,b} \geq 0 \}$, and consider the vectors $u:= (u_a \mid a \in \calA), \ u':= (u'_a \mid a \in \calA)$, and $v:= (v_a \mid a \in \calA)$ in 
	$\RR^\calA_\geq$ given by
	\begin{equation*}
	u_a=
	\begin{cases}
	1 &\text{if } a = m \\
	0  &\text{otherwise}
	\end{cases}, \ \ 
	u'_a=
	\begin{cases}
	-\beta_{m,a} &\text{if } a \in \calB \setminus \calB'\\
	0 &\text{otherwise}
	\end{cases}, \ \text{ and } \ 
	v_a=
	\begin{cases}
	\beta_{m,a} &\text{if } a \in \calB'\\
	0 &\text{otherwise}
	\end{cases}.
	\end{equation*}
	Note that
	\[
	\calA (u+u') = \calA v.
	\]
	This implies $z^{u+u'} = z^v$. Hence
	\begin{equation}\label{Eq:fisonto}
	z_m =  \prod_{ b \in \mathcal{B}} z_b^{\beta_{m,b}} \ .
	\end{equation}
	The image of $\varphi_\calA$ is identified with  the quotient $T_N / T_{\calA^\perp}$. Replacing $T_N$ with $T_N / T_{\calA^\perp}$, we may assume $X_\calA^\circ$ is identified with $T_N$. Note that in this setting, 
	$\calA$ spans $M$, and $\calB$ is a basis for $M$. 
	Let $\calB^*$ be the dual basis for $N$,  write $b^* \in \calB^*$ is the dual basis element to $b\in \calB$. Consider the map
	\begin{align*}
	f \colon N &\longrightarrow \RR^\calA_> \\
	v &\longmapsto \bigg(\prod_{ b \in \mathcal{B}} \exp(-\beta_{a,b} v_{b^*})\mid a \in \calA\bigg),
	\end{align*}
	where $v=\sum_{b^* \in \mathcal{B}}v_{b^*}b^*$. Note that by \eqref{Eq:fisonto}, $\mathcal{Z}_\calA \cap \RR_>^\calA \subset \im(f)$. Let $t=\gamma_v \in T_N$. Then for any $a \in \calA$,
	\begin{align*} 
	t^a &= \exp(-\langle a, v \rangle) \\
	& = \exp\bigg(- \Big\langle \sum_{b \in \mathcal{B} } \beta_{a,b}b,v \Big\rangle \bigg) \\
	& = \prod_{b \in \calB} \exp(- \langle \beta_{a,b}b,v \rangle)\\
	& = \prod_{b \in \calB} \exp (-\beta_{a,b}v_{b^*}).
	\end{align*}
	Therefore, $f(v) = \varphi_\calA(t)$. Let $v=\sum_{b^* \in \mathcal{B}^*}v_{b^*}b^*$, and set $v_{b^*} := -\log(z_b)$ for all $b^* \in \calB^*$. Then, 
	\begin{align*}
	\varphi_\calA(t) &= f(v) = \bigg( \prod_{b \in \calB} \exp (-\beta_{a,b}v_b^*) \mid a \in \calA\bigg)\\
	&= \bigg( \prod_{b \in \calB} \exp (\beta_{a,b} \log(z_b)) \mid a \in \calA\bigg)\\
	&= \bigg( \prod_{b \in \calB} z_b^{\beta_{a,b}} \mid a \in \calA\bigg)\\
	&=z.
	\end{align*}
	
	Now, let $z \in \mathcal{Z}_\calA \setminus \RR_>^\calA$. Consider the cone $F$ generated by the set $\calF := \{ a \in \calA \mid z_a \neq 0 \}$. We claim $F$ is a face of $\cone(\calA)$. Pick any element $\omega \in F^\perp$. We will show $F = \cone(\calA) \cap \omega^\perp$. Let $m \in \langle \calF \rangle \cap \calA$. Then $m= \calF x - \calF y$ for some $x,y \in \RR^\calF_\geq$. Similar to \eqref{Eq:fisonto} we have
	%
	\[
	z_m \prod_{a \in \calF}z_a^{y_a} = \prod_{a \in \calF} z_a^{x_a}.
	\]
	This implies $z_m \neq 0$, i.e., $m \in \calF$.  Therefore $\calF = \langle \calF \rangle \cap \calA$.

	Assume $F$ is not a face of $\cone(\calA)$.  Then there exist an element $x \in \calA$ so that $\langle x, \omega \rangle \neq 0$. Assume $\langle x, \omega \rangle >0$. Then $\langle a, \omega \rangle > 0$ for all $a \in \calA \setminus \langle \calF \rangle.$ Otherwise, if $\langle y, \omega \rangle < 0$ for some $y \in \calA \setminus \langle F \rangle$, then $r_1x+r_2y \in F$ for some scalars $r_1,r_2 \in \RR_>$.  Let $r_1x+r_2y = \calF \lambda$ for some $\lambda \in \RR^\calF_\geq$. 
	As in \eqref{Eq:fisonto}, we must have
	$$z_x^{r_1}z_y^{r_2} = \prod_{a \in \calF} z_a^{\lambda_a}.$$
	%
	Note that
	$
	z_x^{r_1} z_y^{r_2} = 0,
	$
	as $x,y \notin \calF$,
	but 
	$
	\prod_{a \in \calF} z_a^{\lambda_a} \neq 0,
	$
	which is a contradiction. So, for $a \in \calA \setminus \langle \calF \rangle$, we have $\langle a, \omega \rangle >0$. Then $\cone(\calA) \cap \omega^\perp = F$. 
	Hence, $F$ is a face of $\cone(\calA)$.

	Set $z' = (z_f \mid f \in \calF)$. By similar arguments to when $z \in \RR_>^\calA$, we can find an element $v \in F^\vee \subset N$ such that 
	$z'= \varphi_\calF (v)$. 
	Note that $z_a = z'_a$ if $a \in \calF$ and $z_a=0$ otherwise.
	Next, consider the curve
	\[
	\varphi_\calA(v+ r \omega ) = (\exp(\langle -a,v+ r \omega\rangle ) \mid a\in \calA)
	\]
	for a scalar $r \in \RR$. Since $ a \in \calA$ and $\omega \in \cone(\calA)^\vee$, we have $\langle a,\omega \rangle \geq 0$, and it is equal to $0$ only if $a \in \calF$. Hence
	\[
	\lim\limits_{r\to \infty} \exp (\langle -a,v+ r \omega\rangle ) = 
	\begin{cases*}
	0 & \text{if} $\ a \notin \calF$,\\
	\exp (\langle -a,v \rangle ) & \text{if} $\ a \in \calF$.
	\end{cases*}
	\]
	Thus, $\lim_{r \to \infty} \varphi_\calA(v+ r\omega)  = z,$ which implies $z \in X_\calA$ as it is a limit point in $X_\calA$. Hence, we get the reverse inclusion $\mathcal{Z}_\calA \subseteq X_\calA$.   
\end{proof}

For each face $\calF \subset \calA$, $\RR_\geq^\calF$ is naturally included in $\RR_\geq^\calA$, where $\RR_\geq^\calF$ is the set of points $z \in \RR_\geq^\calA$ whose coordinates $z_a$ are zero for $a \notin \calF$. Then $X_\calF^\circ$ is the image of the composition 
\[
T_N  \xrightarrow{\ \varphi_\calA\ }\RR_\geq^\calA \xrightarrow{\ \pi_\calF \ } \RR_\geq^\calF,
\]
where $\pi_\calF$ is the projection of $\RR_\geq^\calA $ onto $\RR_\geq^\calF$. 
Proposition~\ref{P:IATV} implies that $X^\circ_\calF \subset X_\calA$ and $X_\calA$ is the disjoint union
\begin{equation}\label{Eq:decomposition}
X_\calA \ =\ \bigsqcup_{\calF \preceq \calA} X^\circ_\calF\,.
\end{equation}

We end this section with an application of Birch's Theorem \cite{Birch}. 
\begin{proposition} \label{P:Birch}
	The irrational affine toric variety $X_\calA$ is homeomorphic to $\cone\calA$. 
\end{proposition}
\begin{proof}
	Let $\calA = \{a_1, \ldots,a_m\}$. Consider the map 
	\begin{align*}
	\pi \colon X_\calA &\longrightarrow \cone(A)\\
	p  \ &\longmapsto \calA p \ .
	\end{align*}
	
	We will show that this map is a bijection. Let $b  \in \cone(\calA)$. Then $b$ lies in the relative interior of a face  $\cone(\calF)$ of $\cone(\calA)$ for some subset $\calF$ of $\calA$. After rearranging, if necessary, we may assume $\calF = \{a_1, \ldots, a_k\}$ for some $k \leq m$. Then we can write $b = a_1 u_1 + \ldots + a_k u_k$ for some $u_1, \ldots,u_k \in \RR_>$.  Consider the set
	\[
	C_\calF(b) := \{p \in \RR_>^k \mid \calF p = b\}.
	\]
	We will show that the intersection $X_\calF^\circ \cap C_\calF(b)$ contains only one element. We claim that this implies that $\pi$ is a bijection. Let $p \in X_\calF^\circ \cap C_\calF(b)$. Then $\pi(p) = b$, which implies $\pi$ is onto. To show injectivity, suppose for some element $q \in X_\calA$, we have $\calA q = b$. Since  $X_\calA$ is the disjoint union $\bigsqcup_{\calF \preceq \calA} X^\circ_\calF$,  we have $q \in X^\circ_{\calF'}$ for some $\calF' \preceq \calA$. Then $b = \calA q = \calF' q$ lies in the relative interior of $\cone(F')$. Hence we must have $\calF' = \calF$, as $b$ also lies in the relative interior of $\cone(\calF$). So $q \in X_\calF^\circ \cap C_\calF(b)$. This implies $p=q$, which implies $\pi$ is injective.  
	
	Define a function 
	\begin{align*}
	H \colon \RR_>^k &\longrightarrow \quad \RR \\
	p \ &\longmapsto -\sum_{i = 1}^k p_i (\log(p_i)-1).
	\end{align*}
	The Hessian matrix of $H$ is a diagonal matrix with entries $(-1/p_1, \ldots, -1/p_k)$, which is a negative definite matrix. Hence $H$ is strictly concave on $\RR_>^k$. The restriction of $H$ on $C_\calF(b)$ is strictly concave as well. So it attains its maximum at a unique point $p^* \in C_\calF(b)$. 
	
	Consider $\calF$ as a matrix with columns $a_1, \ldots, a_k$.  For any vector $u$ in the kernel of the transpose matrix $\calF^T$, the directional derivative of $H$ vanishes at $p^*$ \cite[Section 1.2]{Pachter}. So we get
	\begin{align} \label{Eq:Hessian}
	0= \sum_{i=1}^k u_i\  \frac{\partial H}{\partial p_i^*}(p^*) = -\sum_{i=1}^k u_i \log(p_i^*).
	\end{align}
	That means $(\log(p_1^*), \ldots, \log(p_k^*) )$ lies in the column span of $\calF$. Pick a vector $\mu \in \RR^n$ such that $\log(p_i^*) = \sum_{j=1}^{n}\mu_j  (a_i)_j$.
	Then $p_i^* = t^{a_i}$ where $t = (\exp(\mu_1), \ldots, \exp(\mu_n))$. Hence $p^* \in X_\calF^\circ \cap C_\calF(b)$. 
	
	To see that $X_\calF^\circ \cap C_\calF(b)$ contains no other point, suppose $q \in X_\calF^\circ \cap C_\calF(b)$. Then $(\log(q_1), \ldots, \log(q_k) )$ lies in the column span of $\calF$. Hence for any vector $u$ in the kernel of the matrix $\calF^T$, the directional derivative of $H$ vanishes at $q$.  
	So $q$ is a critical point of the function $H$ \cite[Section 1.2]{Pachter}. Since the Hessian matrix is negative definite at $q$, this point is a maximum of a strictly concave function $H$, therefore $q=p^*$.  
\end{proof}

\section{Irrational Toric Varieties Associated to Cones}
We will construct irrational affine toric varieties associated to arbitrary cones in $N$,  and show some properties of them.  

Let $C \subset M$ be a cone. Define $\Hom_c(C, \RR_\geq)$ to be the set of all monoid homomorphisms $\varphi \colon C \to \RR_\geq$ that are continuous on the relative interior of each face of $C$. We equip $\Hom_c(C, \RR_\geq)$ with the weakest topology such that every map to $\RR_\geq$ given by point evaluation is continuous. That is, a sequence of homomorphisms $\{\varphi_n \mid n\in \NN \} \subset \Hom_c(C, \RR_\geq)$ converges to a homomorphism $\varphi \in \Hom_c(C, \RR_\geq)$ if and only if for every $u \in C$, the sequence of real numbers $\{\varphi_n(u)  \mid n \in \NN \}$ converges to $\varphi(u)$.

\begin{example}\label{Ex:PosLine}
	Let $M = \RR$ and $C = [0, \infty)$. Note that $C$ is a cone in $M$. Let $\varphi \in \Hom_c(C, \RR_\geq)$. Since $\varphi$ is a monoid homomorphism, we have $\varphi(0)=1$. Set $\alpha := \varphi(1) \geq 0.$ Then for any $n \in \NN$, $\varphi(n) = \alpha^n$. Also since $\alpha = \varphi(1) = \varphi(n\cdot (1/n)) = (\varphi(1/n))^n$, we have $\varphi(1/n) = \alpha^{1/n}$. For any positive rational number $r=m/n$ with $m,n \in \NN$, we have $\varphi(r) = \varphi(m \cdot (1/n)) = (\varphi(1/n))^m = (\alpha^{1/n})^m = \alpha^r$. By the continuity of $\varphi$ on the interior $(0,\infty)$ of $C$, $\varphi(s) = \alpha^s$ for all $s>0$.

	For any $\alpha \in \RR_\geq $, define a map $\varphi_\alpha \colon \RR_\geq \rightarrow  \RR_\geq$ by 
	
	\[
	\varphi_\alpha(s) :=
	\begin{cases*}
	\alpha^s & $s>0$,\\
	1 & $s=0$.
	\end{cases*}
	\]
	If we set $0^0 :=1$, then $\varphi(s) = \alpha^s$ for any $\alpha,s \geq 0$. For any $\alpha \in \RR_\geq$ we write $\varphi_\alpha$ for the monoid homomorphism such that $\varphi_\alpha(s)=\alpha^s$. Figure  \ref{F:HomPositiveLine} displays the graphs of $\varphi_\alpha$ for several $\alpha \in  \RR_\geq$.
	
	\begin{figure}[h]
		\begin{center}
			\begin{picture}(232,89)(-18,-10)
			\put(-0.5,-0.5){\includegraphics[scale=1]{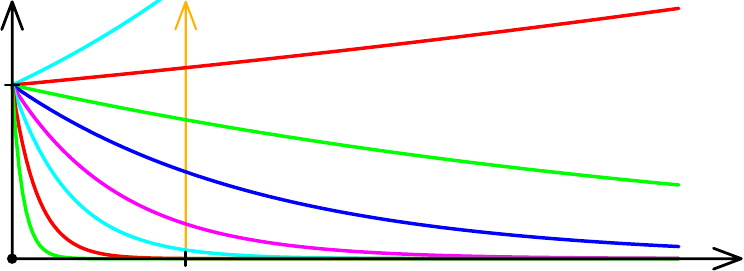}}
			\put(-8,49){$1$} \put(50,-10){$1$}  \put(190,-10){$C$} 
			\put(-18,67){$\RR_\geq$}
			\end{picture}
			\caption{Monoid homomorphisms in $\Hom_c(C,\RR_\geq)$.}
			\label{F:HomPositiveLine}
		\end{center}
	\end{figure}
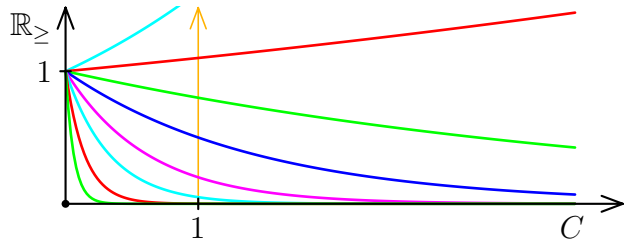
	Since $\lim_{\alpha\to 0}\varphi_\alpha=\varphi_0$, the evaluation map $\varphi\mapsto\varphi(1)$ induces a homeomorphism between $\Hom_c(C,\RR_\geq)$ and $\RR_\geq$.\hfill$\square$
\end{example}

\begin{lemma}
	Let $C$ be a cone in $M$. For any monoid homomorphism $\varphi \in \Hom_c(C, \RR_\geq)$, the set 
	$\{m\in C \mid\varphi(m) > 0\}$
	is a face of $C$. 
\end{lemma}

\begin{proof}
	As in Example \ref{Ex:PosLine}, for any $s \in \RR_\geq$ and $m \in C$, we have $\varphi(sm)  = \left( \varphi(m)\right)^s$. Suppose $m$ lies in the relative interior $\Relint(F)$ of a face $F$ of $C$, and $\varphi(m) = 0$. Then for any $s > 0$, we have $\varphi(sm) = 0$. For any $\omega \in \Relint(F)$, there exists an $s>0$ such that $\omega - sm \in \Relint(F)$. Hence $\varphi(\omega) = \varphi(\omega-sm)\varphi(sm)= 0$. 
	
	Now let $S:= \{ m \in C \mid \varphi(m) > 0\}$. Note that for $m_1, m_2 \in S$, $\varphi(m_1+m_2) = \varphi(m_1) \varphi(m_2) >0$ as both $\varphi(m_1)$ and $\varphi(m_2)$ are positive. Also note that for any nonnegative real $r$ and $m \in S$, $\varphi(rm)= \varphi(m)^r >0$. Hence $S$ is closed under addition and scaling by nonnegative real numbers. So $S$ is a cone. By previous arguments, if $S$ meets the relative interior $\Relint(F)$ of any face $F$ of $C$, then it has to contain $\Relint(F)$. Note that for any $m \in F$ and $\omega \in \Relint(F)$, we have $m+\omega \in \Relint(F)$. So $0 \neq \varphi (m+\omega) = \varphi(m) \varphi(\omega)$. Hence $\varphi(m) \neq 0$, which implies $m \in S$. Thus $F$ is a subset of $S$. As a convex union of faces of a cone is a face of that cone, we showed that $S$ is a face of $C$. 
\end{proof}

\begin{definition}
	For $\varphi \in \Hom_c(C,\RR_\geq)$, the face $\text{supp}(\varphi) := \{ m \in C \mid \varphi(m) > 0 \}$ of $C$ is called the \demph{support} of $\varphi$. 
\end{definition}

\begin{example}
	Let $L$ be a linear subspace of $M$. Then for any $\varphi \in \Hom_c(L,\RR_\geq)$ and $x \in L$, since $x + (-x) = 0$ in L, we have $\varphi(x) \cdot \varphi(-x) = \varphi(0) = 1$. So $\varphi(x) >0$. Hence $\Hom_c(L,\RR_\geq) = \Hom_c(L, \RR_>)$, which is isomorphic to $T_{L^\vee} = T_N / T_{L^\perp} \simeq N / L^\perp$. \hfill$\square$
\end{example}

Consider the map 
\begin{align*}
C &\longrightarrow M \oplus C\\
m &\longmapsto (m,m).
\end{align*}
Applying $\Hom_c(-,\RR_\geq)$ to this map induces another map
\begin{align*}
\mu \colon \Hom_c(M, \RR_\geq) \times \Hom_c(C, \RR_\geq) &\longrightarrow \Hom_c(C, \RR_\geq) \\
(t,\varphi) \quad &\longmapsto \quad t\cdot \varphi,
\end{align*}
where $(t \cdot \varphi) (u) = t^u \varphi(u)$ for $u \in C$. This gives an action of $T_N$ on $\Hom_c(C, \RR_\geq)$. 

\begin{lemma} \label{L:HomIsAffine}
	Let $\calA$ be a finite subset of $M$ and set $C = \cone(\calA)$. Then the map 
	\begin{align*}
	f_\calA \ \colon \Hom_c(C, \RR_\geq) &\longrightarrow \ \RR_\geq^\calA \\
	\varphi \qquad &\longmapsto \ (\varphi(a) \mid a \in \calA),
	\end{align*}
	is a $T_N$-equivariant homeomorphism between $\Hom_c(C, \RR_\geq)$ and the irrational affine toric variety $X_\calA$. In particular, $\Hom_c(C, \RR_\geq)$ is homeomorphic to $C$ under the map $\varphi \mapsto \sum_{a \in \calA} \varphi(a)a.$ 
\end{lemma}

\begin{proof}
	Let $\varphi \in \Hom_c(C, \RR_\geq)$, and suppose that an element $m \in C$ has two representations $m = \calA u  = \calA v$ for some $u,v \in \RR_\geq^\calA$. Note that 
	\begin{equation}\label{Eq:HomIsAffine}
	\varphi(m) = \varphi (\calA u)= \prod_{a \in \calA} (\varphi (a))^{u_a}.
	\end{equation}
	Hence $f_\calA(\varphi)$ satisfies the system of equations \eqref{Eq:Binomials}, and therefore is a point of $X_\calA$ by Proposition \ref{P:IATV}.
	
	Now assume $z \in X_\calA$. Then by Equation \eqref{Eq:HomIsAffine}, the function 
	\begin{align*}
	\varphi \ \colon \calA &\longrightarrow \RR_\geq \\
	a &\longmapsto z_a
	\end{align*}
	extends to a monoid homomorphism $C \to \RR_\geq$. By the decomposition \eqref{Eq:decomposition}, there exists a face $\calF$ of $\calA$  such that $z \in X_\calF^\circ$. So $\varphi$ is continuous and does not vanish on $F = \cone(\calF)$. We claim that $\varphi_z$ is zero on $C \setminus F$, hence it lies in $ \Hom_c(C, \RR_\geq)$. Note that, if $ b \in \calA \setminus \calF$, then $z_b = 0$, and so $\varphi(b) =0$. If $m \in C \setminus F$, then in any expression $m = \calA u$ of $m$ for some $u \in \RR_\geq^\calA$, there exists an element $b \in \calA \setminus \calF$ such that $u_b \neq 0$. Then
	\[
	\varphi_z(m) = \left(\varphi_z (b)\right)^{u_b}  \varphi_z \biggl(\ \sum_{a \in \calA \setminus \{b\}} u_a a\biggr) = 0,
	\]
	which proves the claim.
	
	The maps $\varphi\mapsto f_\calA(\varphi)$ and $z\mapsto \varphi_z$ are inverse bijections between $\Hom_c(C, \RR_\geq)$ and $X_\calA$ which are continuous and therefore homeomorphisms.

	For $t \in T_N$,  
	\begin{align*}
	f_\calA (t\cdot \varphi) &= ((t\cdot \varphi) (a) \mid a \in \calA )= (t(a) \varphi(a) \mid a \in \calA )= t \cdot (\varphi(a) \mid a \in \calA) \\
	&= t\cdot f_\calA (\varphi).
	\end{align*}   
	Hence $f_\calA$ is a $T_N$-equivariant homeomorphism between  $\Hom_c(C, \RR_\geq)$ and the irrational affine toric variety $X_\calA$.
	Moreover, by Proposition \ref{P:Birch} it follows that $\Hom_c(C, \RR_\geq)$ is homeomorphic to $C$ under the map $ \varphi \mapsto \sum_{a \in \calA} \varphi(a)a$. 
\end{proof}

\begin{example}\label{Ex:ConeEx}
	Let $a= (-\sqrt{2},1)$ and $b=(1,0)$ two points in $\RR^2$. Consider the cone $C$ generated by $a$ and $b$. A map $\varphi \in \Hom_c(C, \RR_\geq)$ is is determined by its values $\varphi(a)$ and $\varphi(b)$, which may be any two nonnegative real numbers. Hence the map
	\begin{align*}
	\psi \ \colon \Hom_c(C, \RR_\geq) &\longrightarrow C\\
	\varphi \qquad &\longmapsto \ \varphi(a) a + \varphi(b) b
	\end{align*}
	is a homeomorphism. 
	
	Adding a generator $c=(1,1)$ of $C$, for a map $\varphi \in \Hom_c(C, \RR_\geq)$, we have $$\varphi(c) = \varphi(a+(1+\sqrt{2})b)= \varphi(a)\varphi(b)^{1+\sqrt{2}}.$$
	Hence the map $\varphi \mapsto (\varphi(a),\varphi(b), \varphi(c))$ is a homeomorphism between $\Hom_c(C, \RR_\geq)$ and $X_{\{a,b,c\}}$. Moreover, the map 
	\[
	\varphi \longmapsto \varphi(a)a+\varphi(b)b + \varphi(c)c = \varphi(a)a+\varphi(b)b+\varphi(a)\varphi(b)^{1+\sqrt{2}}(a+(1+\sqrt{2})b)
	\]
	is a homeomorphism between $\Hom_c(C, \RR_\geq)$ and $C$, which is different from $\psi$. Figure \ref{F:IATV} shows the cone $C$ and the irrational affine toric variety $X_{\{a,b,c\}}$.
	\begin{figure}[htb]
		\[
		\begin{picture}(187,84)
		\put(0,0){\includegraphics[height=84pt]{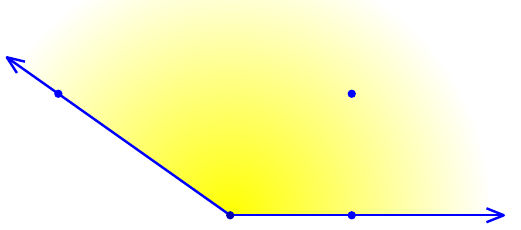}}
		\put(15,40){$a$}    \put(127,10){$b$}    \put(127,40){$c$}   
		\put(38,66){$C=\cone\{a,b,c\}$}
		\end{picture}   
		\qquad
		\begin{picture}(135,90)
		\put(0,0){\includegraphics[height=90pt]{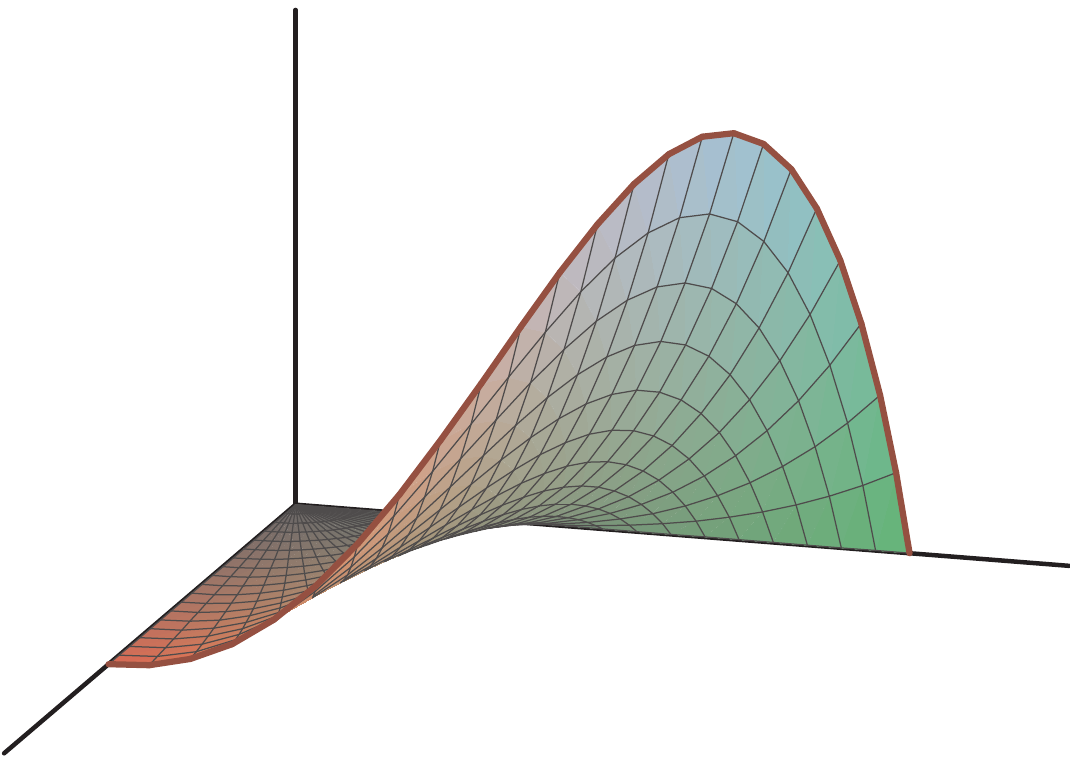}}
		\put(0,9 ){$a$}    \put(127,13){$b$}    \put(40,81){$c$}   
		\put(107,75){$X_{\{a,b,c\}}$}
		\end{picture} 
		\]
		\caption{Cone and irrational affine toric variety.}
		\label{F:IATV}
	\end{figure}
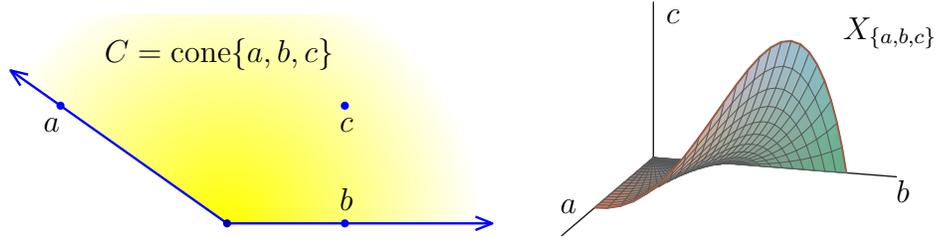\hfill$\square$
\end{example}

We have seen that faces $F$ of the cone $C$ correspond to homomorphisms 
$\varphi \in \Hom_c(C, \RR_\geq)$ that vanish on $C \setminus F$ and are nonzero on $F$. Let $\langle F \rangle$ denote the linear span of $F$. Note that $\Hom_c(\langle F \rangle, \RR_\geq) = \Hom_c(\langle F \rangle , \RR_>)$ as a monoid homomorphism in $\Hom_c(\langle F \rangle, \RR_\geq)$ does not vanish on $\langle F \rangle$. Hence $\Hom_c(\langle F \rangle, \RR_\geq)$ is a single $T_N$-orbit which  may be identified with $T_N / T_{F^\perp}$. 
\begin{notation}
	We write $\varepsilon_F \in \Hom_c(\langle F \rangle, \RR_\geq)$ for the constant homomorphism, that is $\varepsilon_F(m) = 1$ for all $m \in \langle F \rangle$. 
\end{notation}

Note that $\Hom_c(\langle F \rangle , \RR_>)= T_N \cdot \varepsilon_F$. Restricting $\varepsilon_F$ to $F$ and extending it to $C$ by assigning $\varepsilon_F(m) =0$ for all $m \in C \setminus F$ gives a $T_N$-equivariant map $$\Phi_F \ \colon \Hom_c(\langle F \rangle, \RR_>) \longrightarrow \Hom_c(C, \RR_\geq),$$ which sends the constant map $\varepsilon_F$ to the element of $\Hom_c(C, \RR_\geq)$ (still written $\varepsilon_F$) defined by
\[
\varepsilon_F(m)= 
\begin{cases}
1 & \text{if } m \in F\\
0 & \text{if } m \in C\setminus F
\end{cases}.
\]
Let $\calO_F \  := T_N \cdot \varepsilon_F \subset \Hom_c(C, \RR_\geq)$ be the $T_N$-orbit through $\varepsilon_F$. Note that $\calO_F$ is the image of $\Hom_c(\langle F \rangle, \RR_>)$ under the map $\Phi_F$. 

\begin{corollary}\label{C:HomDecomposition}
	Let $F$ be a face of a cone $C$. Then $\calO_F$ consists of monoid homomorphisms in $\Hom_c(C, \RR_\geq)$ that vanish on $C \setminus F$ and are nonzero on $F$. The map 
	$\Phi_F $ 
	is an inclusion and we have the decomposition
	\[
	\Hom_c(C, \RR_\geq) = \coprod_{F \preceq C} \calO_F.
	\]
\end{corollary}
\begin{proof}
	By the definition of the map $\Phi_F$, its image consists of monoid homomorphisms that vanish on $C\setminus F$ and are nonzero on $F$. 
	
	Conversely, let $\varphi \in \Hom_c(C, \RR_\geq)$ that vanishes on $C \setminus F$ and is nonzero on $F$. Restricting $\varphi$ to $F$ gives a monoid homomorphism in $\Hom_c(F, \RR_>)$, and hence has a unique extension to $\langle F \rangle $. Hence $\varphi$ lies on the image of $\Phi_F$, i.e. $\varphi \in \calO_F$.  
	
	The map $\Phi_F$ is an injection since an element of $\Hom_c(\langle F \rangle, \RR_\geq) $ is determined uniquely by its restriction to $F$. 
	
	The decomposition of $\Hom_c(C,\RR_\geq)$ into the orbits $\calO_F$ for faces $F$ of $C$ follows from the decomposition \eqref{Eq:decomposition} and Lemma \ref{L:HomIsAffine}.
\end{proof}

Let $F$ be a face of $C$. Then there exists an element $v \in C^\vee$ such that $F = C \cap v^\perp$ and for $u \in C \smallsetminus F$, $\langle u,v \rangle >0$. For $s \in \RR$, we have the element $\gamma_{sv}$ in $T_N$ whose value at $u \in M$ is $\gamma_{sv}(u)= \exp(\langle -su, v\rangle )$. The map from $\RR \to T_N$ defined by $s \mapsto \gamma_{sv}$ is a one-parameter subgroup of $T_N$. 

\begin{lemma}\label{L:limitOfDistinguishedPoints}With these definitions, we have
	$\varepsilon_F=\lim_{s\to \infty} \gamma_{sv} \cdot \varepsilon_C$.
\end{lemma}

\begin{proof}
	If $\gamma_{sv} \cdot \varepsilon_C$ has a limit as $s \to \infty$ in $\Hom_c(C, \RR_\geq)$, then its value at $u \in C$ is
	\[
	\lim_{s\to \infty} (\gamma_{sv}  \cdot \varepsilon_C )(u) =\lim_{s\to \infty}  \gamma_{sv}(u) \varepsilon_C(u) =\lim_{s\to \infty} \exp(\langle -su, v\rangle )=
	\begin{cases}
	0 &\text{if } u \in F\\
	1 &\text{if } u \notin F
	\end{cases},
	\] 
	Hence $\lim_{s\to -\infty} \gamma_{sv}  \cdot \varepsilon_C $ is $\varepsilon_F$.
\end{proof}

Lemma~\ref{L:limitOfDistinguishedPoints} implies that $\varepsilon_F\in\overline{\calO_C}$.
As $\calO_F=T_N.\varepsilon_F$, we deduce the following.

\begin{corollary}\label{Cor:orbits}
	If $F\subset E$ are faces of the cone $C$, then $\calO_F\subset\overline{\calO_E}$.
	In particular, $\calO_C$ is dense in $\Hom_c(C,\RR_\geq)$, and
	\[
	\overline{\calO_E}\ =\ \coprod_{F \preceq E} \calO_F\ .
	\]
\end{corollary}

\section{Irrational Toric Varieties Associated to Fans}
We next define the irrational toric variety corresponding to an arbitrary fan $\Sigma$ in $N$. For a cone $\sigma \in \Sigma$, define $V_\sigma := \Hom_c(\sigma^\vee, \RR_\geq)$, the irrational affine toric variety associated to the dual cone of $\sigma$. 
For a face $\tau \preceq \sigma$, the inclusion $\tau \subset \sigma$ induces a map $V_\tau \to V_\sigma$ by restricting a monoid homomorphism $\tau^\vee \to \RR_\geq$ to $\sigma^\vee$, as $\sigma^\vee \subset \tau^\vee$.

\begin{lemma}
	The map $V_\tau \to V_\sigma$ is a $T_N$-equivariant inclusion.
\end{lemma}
\begin{proof}
	Let $f: V_\tau \to V_\sigma$ be the induced map. Then for $t \in T_N$ and $\varphi \in V_\tau$, we have $f(t \cdot \varphi) = \restr{(t\cdot \varphi)}{\sigma^\vee}  = t \cdot \restr{\varphi}{\sigma^\vee} = t \cdot f(\varphi)$. Hence the map is $T_N$-equivariant.
	
	Let $\omega \in \tau^\vee$. By Lemma \ref{L:coneFact}, there are $u,v \in \sigma^\vee$ with $v \in \tau^\perp$ such that $\omega = u-v$. Since $\tau^\perp \subset \tau^\vee$ is a linear space, we have $\varphi(v) \neq 0$ and $\varphi(\omega) = \varphi(u)\varphi(v)^{-1}$. Hence $\varphi$ is determined by its restriction to $\sigma^\vee$. Hence $f$ is injective. 
\end{proof}

\begin{example}\label{Ex:Inclusion}
	Consider the cone $\sigma$ generated $(1,\sqrt{2})$ and $(0,1)$ in $\RR^2$. Let $\tau$ be its face generated by 
	$(0,1)$. 
	Then $\sigma^\vee$ is the cone $C$ of Example~\ref{Ex:ConeEx} and 
	$\tau^\vee=\{(x,y)\in\RR^2\mid y\geq 0\}$.
	The points $a=(-\sqrt{2},1)$, $b=(1,0)$, and $c=(1,1)$  lie in both dual cones $\sigma^\vee$ and $\tau^\vee$, and
	$\tau^\vee$ has an additional generator $d:=(-1,0)$.
	Figure~\ref{F:CDATV} displays the cones $\sigma$ and $\tau$, their duals, and the associated 
	\begin{figure}[htb]
		\[
		\begin{array}{lll}
		\begin{picture}(50,60)(-6,0)
		\put(0,0){\includegraphics{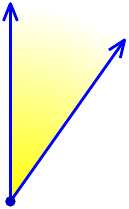}}
		\put(-6,29){$\tau$}  \put(13,43){$\sigma$}
		\end{picture}
		&
		\begin{picture}(155,60)
		\put(0,0){\includegraphics{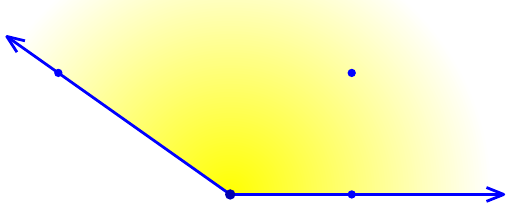}}
		\put(11,31){$a$} \put(50,43){$\sigma^\vee$} \put(93,37){$c$} \put(100,8){$b$}
		\put(38,-8){$\sigma^\perp=0$}
		\end{picture}
		&
		\begin{picture}(95,65)(-2,0)
		\put(0,0){\includegraphics[height=65pt]{AffineTV.pdf}}
		\put(-2,5){$a$}  \put(18,58){$c$} \put(86.5,19){$b$}
		\put(43,6){$X_{\{a,b,c\}}$}
		\end{picture}
		\\
		\begin{picture}(50,85)(-6,0)
		\put(0,0){\includegraphics{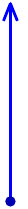}}
		\put(-6,29){$\tau$}
		\end{picture}
		&
		\begin{picture}(155,85)
		\put(0,0){\includegraphics{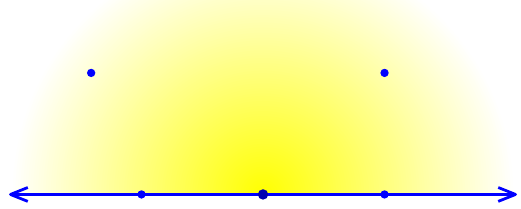}}
		\put(17,32){$a$} \put(75,39){$\tau^\vee$} \put(37,8){$d$} \put(110,8){$b$}
		\put(115,37){$c$}
		\put(-7,4){$\tau^\perp$}
		\end{picture}
		&
		\begin{picture}(93,85)
		\put(0,0){\includegraphics[height=67pt]{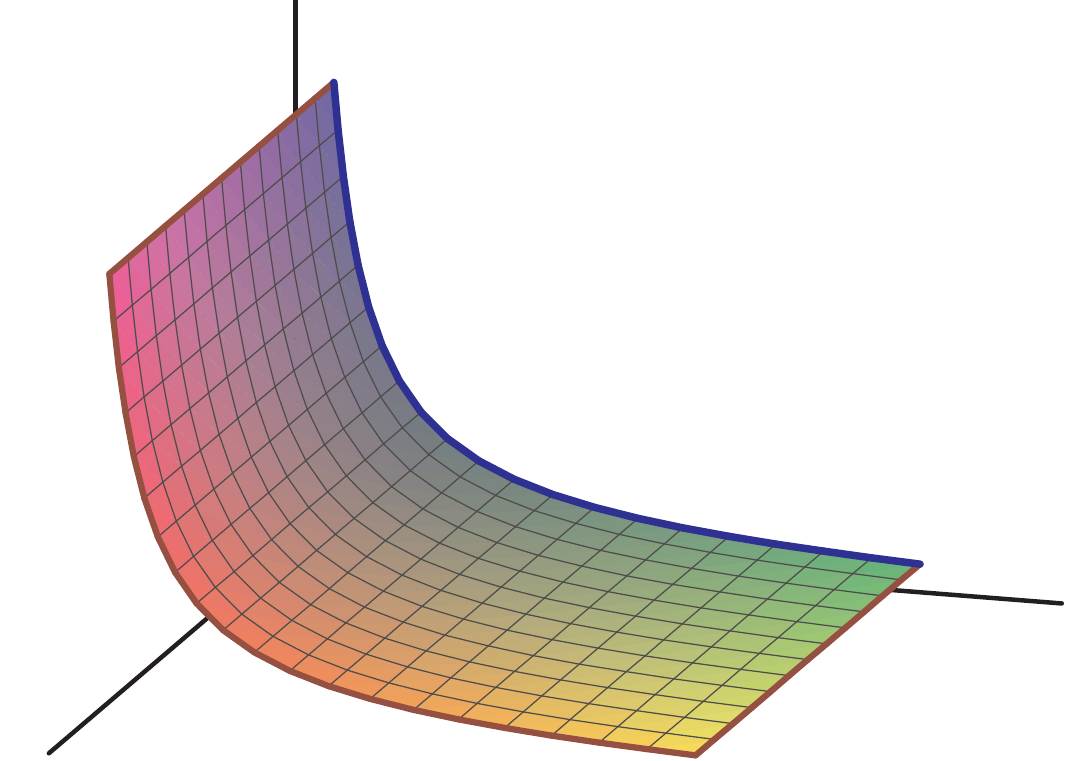}}
		\put(0,4){$a$}  \put(18,63){$d$} \put(87,16.5){$b$}
		\put(55,40){$X_{\{a,b,d\}}$}
		\end{picture}
		
		\end{array}
		\]
		\caption{Cones, their duals, and associated irrational affine toric varieties.}
		\label{F:CDATV}
	\end{figure}
	irrational affine toric varieties $X_{\{a,b,c\}}\simeq V_\sigma$ and  $X_{\{a,b,d\}}\simeq V_\tau$.
	The inclusion $V_\tau\hookrightarrow V_\sigma$ is induced by projecting $X_{\{a,b,d\}}$ to 
	the quadrant $\RR^{\{a,b\}}_\geq$ and then applying the inverse of the projection from $X_{\{a,b,c\}}$.
	The image of $X_{\{a,b,d\}}$  in $X_{\{a,b,c\}}$ only omits the $a$-axis. \hfill$\square$
\end{example}

\begin{definition}
	Let $\Sigma$ be a fan in $N$. The \demph{irrational toric variety} $X_\Sigma$ associated to $\Sigma$ is 
	\[
	Y_\Sigma \ := \bigcup_{\sigma \in \Sigma} V_\sigma,
	\]
	the union of the irrational affine toric varieties $V_\sigma$ for $\sigma \in \Sigma$ glued together along the inclusions $V_\tau \hookrightarrow V_\sigma$ for $\tau \prec \sigma$. 
\end{definition}

\begin{example}\label{P1} 
	Let	$\Sigma \in \RR$ be the fan given in Figure \ref{F:P1}. In Example \ref{Ex:PosLine} we showed that $V_{\sigma_1} \cong \RR_\geq.$ One can similarly show that $V_{\sigma_2}\cong \RR_\geq$.
	
	Note that $0^\vee = \RR$, and for $f \in V_0$,  we have 
	\begin{equation} \label{Eq:Gluing}
	1=f(0) = f(1)f(-1)= \restr{f}{\sigma_1^\vee}(1) \restr{f}{\sigma_2^\vee}(-1).
	\end{equation}
	Two elements $g \in V_{\sigma_1}$ and $h\in V_{\sigma_2}$ are glued together if there exists an element $f \in V_0$ such that $\restr{f}{\sigma_1^\vee}=g$ and $\restr{f}{\sigma_2^\vee}=h$. Since $g$ and $h$ are determined by their values at $1$ and $-1$, respectively, \eqref{Eq:Gluing} implies that they are glued if $g\cdot h =1$. Using Figure \ref{F:P1Fan}, we will explain this gluing process. We place two copies $\{(1,\alpha_t) \mid \alpha_t \in V_{\sigma_1} \} \simeq V_{\sigma_1}$ and $\{(\alpha_s,1) \mid \alpha_s \in V_{\sigma_2} \} \simeq V_{\sigma_2}$  of $\RR_\geq$ on the nonnegative orthant $\RR^2_\geq$. 
	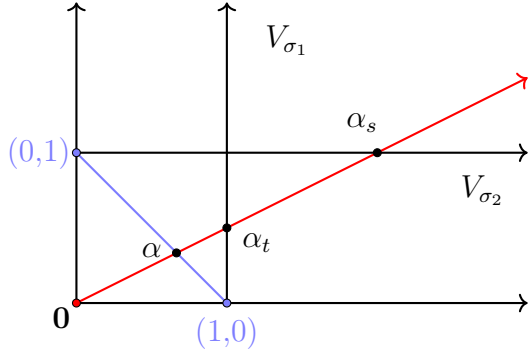
\begin{figure}[h]
		\centering
		\begin{tikzpicture}
		\definecolor{xdxdff}{rgb}{0.49019607843137253,0.49019607843137253,1.}
		
		\draw [->,thick] (0,0) -- (0,4);
		\draw [->,thick] (0,0) -- (6,0);
		\draw [->,thick] (2,0) -- (2,4);
		\draw [->,thick] (0,2) -- (6,2);
		\draw [color=xdxdff,thick] (2,0) -- (0,2);
		\draw [->, color=red,thick] (0,0)--(6,3);

		\draw [fill=xdxdff] (0,2) circle (1.5pt);
		\draw [fill=xdxdff] (2,0) circle (1.5pt);
		\draw [fill=red] (0,0) circle (1.5pt);
		\draw [fill=black] (1.33,0.667) circle (1.5pt);
		\draw [fill=black] (2,1) circle (1.5pt);
		\draw [fill=black] (4,2) circle (1.5pt);

		\draw[color=xdxdff] (2,-0.4) node {(1,0)};
		\draw[color=xdxdff] (-0.5,2) node {(0,1)};
		\draw[color=black] (-0.2,-0.2) node {\textbf{0}};
		\draw[color=black] (1,0.7) node {$\alpha$};
		\draw[color=black] (2.4,0.8) node {$\alpha_t$};
		\draw[color=black] (3.8,2.4) node {$\alpha_s$};

		\draw[color=black] (2.8,3.5) node {$V_{\sigma_1}$};
		\draw[color=black] (5.4,1.5) node {$V_{\sigma_2}$};
		
		\end{tikzpicture}
		\caption{Irrational toric variety corresponding to fan $\Sigma$.}
		\label{F:P1Fan}
	\end{figure}
	
	Let $\alpha$ be the point of intersection of a ray passing from the origin with the line segment between $(1,0)$ and $(0,1)$. This ray cuts the line $x=1$ and $y=1$ at points $(1,\alpha_t)$ and $(\alpha_s,1)$, respectively. Note that we have $\alpha_t \alpha_s=1$. This means we glue these two points together. We identify it by $\alpha$. This procedure identifies $X_\Sigma$ with the blue line segment. \hfill$\square$
\end{example}

For each cone $\sigma \in \Sigma$, let $x_\sigma \in V_\sigma$ be the distinguished point $\varepsilon_{\sigma^\perp}$, where $\sigma^\perp \subset \sigma^\vee$ is its lineality space. We also let $W_\sigma$ be the $T_N$-orbit through $x_\sigma$, so that $W_\sigma = \calO_{\sigma^\perp}$. 

\begin{theorem}\label{Th:ITV_structure}
	Let $\Sigma$ be a fan in $N$. Then the irrational toric variety $X_\Sigma$ is a $T_N$-equivariant cell complex. Each cell is an orbit and corresponds to a unique cone $\sigma \in \Sigma$. The cell corresponding to $\sigma$ is $W_\sigma \simeq N/\langle \sigma \rangle$, and $\tau \subset \sigma$ if and only if $W_\sigma \subseteq \overline{W_\tau}$.
\end{theorem}
\begin{proof}
	Let $\sigma\subset\Sigma$ be a cone. By Corollary \ref{C:HomDecomposition}, the set $V_\sigma$ is a $T_N$-equivariant cell complex whose cells are $T_N$-orbits that corresponds to faces $\tau$ of $\sigma$, where the orbit $W_\tau$ is identified with $N/\langle \tau \rangle$.  By Corollary \ref{Cor:orbits}, the cell $W_\sigma$ is contained in any closure $\overline{W_\tau}$ for a face $\tau$ of $\sigma$. 
	Since $X_\Sigma$ is obtained by gluing the sets along common open subsets, these facts also hold for $X_\Sigma$. 
\end{proof}

\begin{corollary}
	The collection $\{ V_\sigma\mid \sigma\in\Sigma\}$ of irrational affine toric varieties  forms a $T_N$-equivariant open
	cover of $X_\Sigma$ by irrational affine toric varieties.
\end{corollary}
\begin{proof}
	As $X_\Sigma$ is the union of the $V_\sigma$, which are $T_N$-equivariant as is the gluing, we need only show that each
	$V_\sigma$ is open in $X_\Sigma$. Hence it is sufficient to show $X_\Sigma\setminus V_\sigma$ is closed in $X_\Sigma$. Note that if we have $\tau \preceq\rho$ and $\tau \not \preceq \sigma$, then $\rho \not \preceq \sigma$. Hence by Theorem~\ref{Th:ITV_structure}, we have
	\[
	X_\Sigma\setminus V_\sigma\ =\ 
	\bigcup_{\tau\not\preceq\sigma}  W_\tau\ =\ 
	\bigcup_{\tau\not\preceq \sigma}  \overline{W_\tau}\ .
	\]
	As $X_\Sigma\setminus V_\sigma$ is a union of closed sets, it is closed as well. 
\end{proof}

By Theorem \ref{Th:ITV_structure}, an orbit $W_\tau$ lies in the closure of an orbit $W_\sigma$ if and only if $\sigma$ is a face of $\tau$. The following corollary is a consequence of these facts and the definition of star. 

\begin{corollary}\label{C:StarITV}
	For any cone $\sigma \in \Sigma$, the closure of the orbit $V_\sigma$ is the toric variety $X_{\text{star}(\sigma)}$. 
\end{corollary}
\begin{theorem}\label{Th:TVvsITV}
	If $\Sigma\subset N$ is a rational fan, then $X_\Sigma=Y_\Sigma(\RR_\geq)$.
\end{theorem}
\begin{proof}
	Both $X_\Sigma$ and $Y_\Sigma(\RR_\geq)$ are constructed by the same gluing procedure from the sets $V_\sigma=\Hom_c(\sigma^\vee,\RR_\geq)$ and $\Hom_c(S_\sigma, \RR_\geq)$ from the cones $\sigma\in\Sigma$. Hence it is sufficient to show that these two sets are equal for each cone $\sigma$. Let $\sigma\subset M$ be a rational cone.  
	Then $\sigma^\vee$ is generated as a cone by the monoid $S_\sigma$, and 
	let $\calA\subset S_\sigma$ be a generating set for $\sigma^\vee$.
	The map $f_\calA$ of Lemma~\ref{L:HomIsAffine} maps both $V_\sigma$ and $\Hom_{\mon}(S_\sigma,\RR_\geq)$ to $X_\calA$,
	with both maps isomorphisms.
	Thus the restriction map identifies $V_\sigma$ with $\Hom_{\mon}(S_\sigma,\RR_\geq)$, which completes the proof.
\end{proof}

\section{Maps of Fans}
Let $\Sigma_1 \subset N_1$ and $\Sigma_2 \subset N_2$ be two fans and let $X_{\Sigma_1}$ and $X_{\Sigma_2}$ be the irrational toric varieties associated to $\Sigma_1$ and $\Sigma_2$, respectively.
Let $\phi_i \colon T_{N_i} \times X_{\Sigma_i} \to X_{\Sigma_i}$ be the actions of $T_{N_i}$ on $X_{\Sigma_i}$. A map $\psi \colon X_{\Sigma_1} \to X_{\Sigma_2}$ of irrational toric varieties is a continuous map together with a homomorphism $\Psi \colon T_{N_1} \to T_{N_2}$ of topological groups such that the following diagram commutes 
\begin{equation}\label{Eq:IrrationalToricMap}
\begin{tikzcd}
T_{N_1} \times X_{\Sigma_1} \arrow[swap]{d}{\Psi\times \psi}  \arrow{r}{\phi_1}  & T_{N_1} \arrow{d}{\psi} \\
T_{N_2} \times X_{\Sigma_2} \arrow{r}{\phi_2} & T_{N_2}
\end{tikzcd}.
\end{equation}

We next show that the association $\Sigma \mapsto X_\Sigma$ is functorial.
\begin{theorem}\label{Th:mapsOfFans}
	Let $\Sigma_1 \subset N_1$ and $\Sigma_2 \subset N_2$ be fans. If $\Psi \colon \Sigma_1 \to \Sigma_2$ is a map of fans, then there is a continuous map $\psi \colon X_{\Sigma_1} \to X_{\Sigma_2}$ such that the diagram \eqref{Eq:IrrationalToricMap} commutes, where the homomorphism $\Psi \colon T_{N_1} \to T_{N_2}$ is induced by the linear map $\Psi \colon N_1 \to N_2$. 
\end{theorem}

\begin{proof}
	Let $\Sigma_1 \subset N_1$ and $\Sigma_2 \subset N_2$ be two fans, and $\Psi \colon \Sigma_1 \to \Sigma_2$ be a map of fans. The linear map $\Psi \colon N_1 \to N_2$ induces a homomorphism $\Psi \colon T_{N_1} \to T_{N_2}$ of topological groups. We will construct a map $\varphi \colon X_{\Sigma_1} \to X_{\Sigma_2}$ that makes the diagram \eqref{Eq:IrrationalToricMap} commute, by defining $\varphi$ on each irrational toric variety $V_\sigma$ for a cone $\sigma \in \Sigma$.
	
	Let $\sigma_1$ be a cone in $\Sigma_1$. Since $\Psi \colon \Sigma_1 \to \Sigma_2$ is a map of fans, there exists a cone $\sigma_2$ in $\Sigma_2$ such that $\Psi(\sigma_1) \subset \sigma_2$. Let $\Psi^* \colon M_1 \to M_2$ be the adjoint to $\Psi$, where $M_1$ and $M_2$ are dual vector spaces to $N_1$ and $N_2$, respectively. Then we have $\Psi^*(\sigma_2^\vee) \subset \sigma_1^\vee$. Since these are polyhedral cones and $\Psi^*$ is linear, for any face $F_1$ of $\sigma_1^\vee$, $F_2 := \Psi^*(F_1)$ is a face of $\sigma_2^\vee$. 
	
	For $\varphi \in V_{\sigma_1} = \Hom_c(\sigma_1^\vee, \RR_\geq)$, the composition $\psi(\varphi) := \varphi \circ \Psi^*$ is a monoid homomorphism from $\sigma_2^\vee$ to $\RR_\geq$. Hence the inverse image of the support $\text{supp}(\varphi)$ of $\varphi$ is the support of $\psi(\varphi)$, and $\psi(\varphi)$ is continuous on it support.  Hence $\psi$ maps $V_{\sigma_1}$ to $V_{\sigma_2}$. This map is continuous as the topology defined by point evaluation. It is also equivariant, that is the following diagram commutes
	\[
	\begin{tikzcd}
	T_{N_1} \times V_{\sigma_1} \arrow[swap]{d}{\Psi\times \psi}  \arrow{r}{\varphi_1}  & V_{\sigma_1} \arrow{d}{\psi} \\
	T_{N_2} \times V_{\sigma_2} \arrow{r}{\varphi_2} & V_{\sigma_2}
	\end{tikzcd}.
	\] 
	Noting that it is compatible with the gluing completes the proof. 
\end{proof}
\pagebreak{}

\chapter{PROPERTIES OF IRRATIONAL TORIC VARIETIES} \label{CH:PropertiesITV}

In Section \ref{CH:ITV} we constructed irrational toric varieties from arbitrary fans. These have very pleasing similarities with the classical toric varieties. We now study some of their properties. 
We first study irrational toric varieties as monoids. If we adjoin an absorbing element to $X_\Sigma$, we obtain a commutative topological monoid.
We then show how to recover the fan $\Sigma$ from an irrational toric variety $X_\Sigma$. Then we show that $X_\Sigma$ is a compact topological space if and only if the fan $\Sigma$ is complete. 
We end this section by defining projective irrational toric varieties. Theorem \ref{T:ProjectiveITV} gives a homeomorphism between a polytope and the projective irrational toric variety $X_\Sigma$, where $\Sigma$ is the normal fan to that polytope.

\section{Irrational Toric Varieties as Monoids}
In this section we study irrational toric varieties as topological monoids. 
\begin{definition}
	A \demph{topological monoid} is a monoid S with an operation $\bullet$ such that the monoid operation $\bullet \colon S \times S \to S$ is a continuous map.
\end{definition}
The affine irrational toric varieties $X_\calA$ and $\Hom_c(C, \RR_>)$ are topological monoids whose structures are compatible with the isomorphism of Lemma~\ref{L:HomIsAffine}. These monoids contain a dense torus acting on them with finitely many orbits, and are thus irrational analogs of
linear algebraic monoids~\cite{Putcha,Renner}.

\begin{definition}
	Let $C \subset M$ be a cone. For $x,y \in \Hom_c(C,\RR_\geq)$, we define an operation $\bullet$ 
	\begin{align*}
	x \bullet y \colon C &\longrightarrow \ \RR_\geq \\
	u &\longmapsto x(u)y(u).
	\end{align*}
	Let $\Phi(C)$ be the set of faces of $C$. For faces $F,G \in \Phi(C)$, define $F \bullet G := F \cap G$. 
\end{definition}

\begin{proposition}\label{P:MonoidProp}
	Under the compositions $\bullet$, both $\Hom_c(C,\RR_\geq)$ and $\Phi(C)$ are commutative monoids with the map $x\mapsto
	\text{supp}(x)$ a map of monoids, and $\Hom_c(C,\RR_\geq)$ is a topological monoid.
	The identity of $\Hom_c(C,\RR_\geq)$ is the constant map $\varepsilon_C$, and if the lineality space $L$ of $C$ is the origin, then it has an absorbing element $\varepsilon_0$.
	The identity of $\Phi(C)$ is $C$ itself, and $L$ is its absorbing element.
	
	For any $u \in C$, evaluation $x \mapsto x(u)$ is a map $\Hom_c(C,\RR_\geq) \to \RR_\geq$, which is a map of monoids, and for any linear map  $f \colon M' \to M$ and cone $C' \subset M'$ with $f(C') \subset C$, the pullback map $f^* \colon \Hom_c(C, \RR_\geq) \to \Hom_c(C', \RR_\geq)$ is a map of topological monoids. 
\end{proposition}

\begin{proof}
	
	Let $x,y \in \Hom_c(C, \RR_\geq)$. Since $x$ and $y$ are monoid homomorphism, so is $x \bullet y$.
	Also since $\supp(x \bullet y) 	= \supp(x) \cap \supp(y)$ and as an intersection of faces is again a face, $\supp(x \bullet y)$ is a face of $C$. As elements of 
	$\Hom_c(C,\RR_\geq)$ are monoid homomorphisms that are continuous on their support, we conclude that $x \bullet y \in \Hom_c(C,\RR_\geq)$. 
	Also note that $(x \bullet y) (u) = x(u) y(u) = y(u)x(u) = (y \bullet x)(u)$ for all $u \in C$, so this product is commutative. Hence $\Hom_c(C,\RR_\geq)$ is a commutative monoid. As $\bullet$ is defined pointwise evaluation, it is continuous, so $\Hom_c(C,\RR_\geq)$ is a topological monoid. 
	
	Note that for any $x \in \Hom_c(C,\RR_\geq)$ and $u \in C$, $(x \bullet \epsilon_C) (u) = x(u) \epsilon_C(u) = x(u) $. Hence the constant map $\epsilon_C$ is the identity of $\Hom_c(C,\RR_\geq)$. Also note that $(x \bullet \epsilon_0)(u) = \epsilon_0$. 
	Since for any face $F$ of $C$, $F \cap C=F$ and $F \cap L =L$, $C$ is the identity and $L$ is the absorbing element of $\Phi(C)$.
	Also since for any $u \in C$, $(x \bullet y) (u) =  x(u) y(u)$ and $\epsilon_C(u)=1$, the evaluation map  $x \mapsto x(u)$ is a map of monoids. Now let $f \colon M' \to M$ be a linear map and $f(C') \subset C$ for a cone $C' \subset M'$. Then for any $x,y \in \Hom_c(C,\RR_\geq)$ and $u' \in C'$, $f^* (x \bullet y)(u')= (x \bullet y) (f (u'))= x(f(u')) \ y(f(u'))= (f^*(x))(u') \bullet (f^* (y))(u')$. Also $(f^*(\epsilon_C))(u') = \epsilon_C(f(u')) = \epsilon_{C'}$. Hence the pullback map $f^*$ is a map of topological monoids. 
\end{proof}
In general, if we adjoin an absorbing element ${\bf 0}$ to an irrational toric variety $X_\Sigma$, it becomes a
commutative topological monoid such that the inclusion of the irrational affine toric variety $V_\sigma$ is a monoid map,
for each cone $\sigma$ in the fan $\Sigma$.

Let $\calA\subset M$ be a finite subset.
Then $\RR^\calA_\geq$ is a monoid under componentwise multiplication; for $x,y\in\RR^\calA_\geq$ and $a\in\calA$,
$(x\bullet y)_a:= x_a\cdot y_a$.
With this definition, the injective map $f_\calA\colon\Hom_c(\cone(\calA),\RR_\geq)\to\RR^\calA_\geq$ of
Lemma~\ref{L:HomIsAffine} is a monoid homomorphism whose image is $X_\calA$.

Let $\Sigma\subset N$ be a fan.
For a cone $\sigma\in\Sigma$, the irrational affine toric variety $V_\sigma$ is a monoid under pointwise multiplication and
when $\tau\subset\sigma$ is a face, the inclusion $V_\tau\subset V_\sigma$ is a monoid homomorphism. These are both consequences of Proposition \ref{P:MonoidProp}.
We define a product $\bullet$ on $X_\Sigma^+:= X_\Sigma\cup\{{\bf 0}\}$, where $\textbf{0}$ is an isolated point that acts as an absorbing element.
Let $x,y\in X_\Sigma^+$,
\begin{enumerate}
	\item
	If either $x$ or $y$ is ${\bf 0}$, then $x\bullet y={\bf 0}$.
	\item
	If there is a cone $\sigma\in\Sigma$ with $x,y\in V_\sigma$, then we let $x\bullet y$ be their product in $V_\sigma$.
	\item
	If there is no cone $\sigma\in\Sigma$  with $x,y\in V_\sigma$, then $x\bullet y={\bf 0}$.
	(This includes case (1).)
\end{enumerate}

Note that if $\Sigma$ is a fan, the intersection of cones defines a monoid structure on $\Sigma$.
More interesting is the product on $\Sigma^+:=\Sigma\cup\{{\bf 0}\}$ where ${\bf 0}$ is a new point that acts as an
absorbing element, and in which $\sigma\bullet\tau$ is defined to be the smallest cone containing both $\sigma$ and $\tau$
if such a cone exists, and ${\bf 0}$ otherwise.

\begin{theorem}\label{T:Monoids}
	For a fan $\Sigma\subset N$, $X_{\Sigma^+}$ is a commutative topological monoid with the inclusion
	$V_\sigma\hookrightarrow X_{\Sigma^+}$ a map of topological monoids, for every $\sigma\in\Sigma$.
	The map $f \colon X_\Sigma^+ \to \Sigma^+$ that sends an element $x$ to the cone
	$\sigma\in\Sigma$ where $x\in W_\sigma$ or to ${\bf 0}$ when $x={\bf 0}$ is a map of monoids.

\end{theorem}

\begin{proof}
	Let $x, y \in X_\Sigma^+$. Then $x \bullet y$ is either $\bf0$, or there is a cone $\sigma\in\Sigma$ with $x,y\in V_\sigma$, and hence $x\bullet y \in V_\sigma$. In either case $ x \bullet y \in X_\Sigma^+$ and the product is commutative. As $\bullet$ is continuous on each $V_\sigma$, it is continuous on $X_\Sigma^+$. Hence $X_\Sigma^+$ is a topological monoid. 
	
	Let $x,y \in X_{\Sigma}^+$. If either $x$ or $y$ is $\textbf{0}$, then $f(x \bullet y) = 0 = f(x) \bullet f(y)$, as either $f(x)=\textbf{0}$ or $f(y)=\textbf{0}$. If there exists a cone $\sigma \in \Sigma$ with $x, y \in V_\sigma$, then there exists a cone $\tau \in \Sigma$ such that $x,y \in W_\tau$. So $x \bullet y \in W_\tau$. Therefore $f(x \bullet y) = \tau $, which is the smallest cone containing both $f(x) = \tau$ and $f(y) = \tau$. Lastly, if there is no cone $\sigma \in \Sigma$ with $x,y \in V_\sigma$, then $x \in W_\tau$ and $y \in W_{\tau'}$ for some cones $\tau,\tau' \in \Sigma$, where there is no cone containing both $\tau$ and $\tau'$. So $f(x\bullet y) = \textbf{0} = \tau \bullet \tau' =f(x) \bullet f(y)$. Hence $f$ is a map of monoids.  
\end{proof}

\section{Recovering the Fan}\label{S:RecoveringFans}
Let $\Sigma \subset N$ be a fan.  In Example \ref{S:LimitsOPSGTV} we studied the limits of one parameter subgroups, and used these limits to recover the fan $\Sigma$ from a toric variety $Y_\Sigma$. Let $\epsilon$ be the distinguished point in the dense orbit of $T_N$ on $X_\Sigma$. In every affine irrational toric variety $V_\sigma$ for a cone $\sigma \in \Sigma$, $\epsilon$ restricts to the constant homomorphism on $\sigma^\vee$. If $L$ is the lineality space of $\Sigma$, then $\epsilon = x_L$. We study limits of $\epsilon$ in $X_\Sigma$ under one parameter subgroups $\gamma_{sv}$ of $T_N$, and show how to recover the fan $\Sigma$  from an irrational toric variety $X_\Sigma$.

\begin{lemma}\label{L:LimitsofOPSGITV}
	Let $\Sigma$ be a fan in $N$ and $v \in N$. Then the limit $\lim_{s\to \infty} \gamma_{sv}\cdot \epsilon$ exists in $X_\Sigma$ if and only if there is a cone $\sigma \in \Sigma$ with $v \in \sigma$. Moreover, if $v \in \Relint(\sigma)$, then $\lim_{s\to \infty} \gamma_{sv}\cdot \epsilon$ is the distinguished point $x_\sigma$.
\end{lemma}

\begin{proof}
	For $u \in M$ and $s\in \RR$, we have $(\gamma_{sv} \cdot \epsilon)(u) =\gamma_{sv}(u) \varepsilon(u) = \exp(\langle -su, v \rangle)$ since $\epsilon(u)=1$. Hence
	\begin{equation} \label{Eq:LimitEquation}
	\lim_{s\to \infty} (\gamma_{sv}\cdot \epsilon)(u) = \begin{cases}
	0 &\text{if } \langle u,v\rangle  >0 \\
	1 &\text{if } \langle u,v\rangle = 0 \\
	\infty &\text{if } \langle u,v\rangle < 0
	\end{cases}.
	\end{equation}
	If there exist a cone $\sigma \in \Sigma$ with $v\in \sigma$, then for all $u \in \sigma^\vee$, the limit $\lim_{s\to \infty}(\gamma_{sv}\cdot \epsilon)(u)$ is finite by \eqref{Eq:LimitEquation}. Hence by Lemma \ref{L:limitOfDistinguishedPoints}, the family $\gamma_{sv}\cdot \epsilon$ has a limit in $V_\sigma$ as $s \to \infty$.
	
	Conversely, if $\gamma_{sv}\cdot \epsilon$ has a limit in $X_\sigma$ as $s \to \infty$, then there exists a cone $\sigma \in \Sigma$ so that the limit is in $V_\sigma$. Then for all $u \in \sigma^\vee$, the family of real numbers $(\gamma_{sv} \cdot \epsilon)(u)$ has a limit. Hence by \eqref{Eq:LimitEquation} we conclude $v \in \sigma$. 
	
	If $v \in \Relint(\sigma)$,  then the limit  $\lim_{s\to \infty} \gamma_{sv}\cdot \epsilon$ exists and by \eqref{Eq:LimitEquation}
	\begin{equation*} 
	\lim_{s\to \infty} (\gamma_{sv}\cdot \epsilon)(u) = \begin{cases}
	1 &\text{if } u \in M\cap \sigma^\perp \\
	0 &\text{otherwise}
	\end{cases}.
	\end{equation*} Hence the limit is $x_\sigma$.  
\end{proof}

Let $\Sigma \in N$ be a fan. Let $N^\circ \subset N$ be the set of $v \in N$ such that  $\gamma_{sv}\cdot \epsilon$ has a limit in $X_\sigma$ as $s \to \infty$. We define an equivalence relation on $N^\circ$, where $v \sim w$ for $v, w \in N^\circ$ if and only if $\lim_{s \to \infty} \gamma_{sv}\cdot \epsilon = \lim_{s \to \infty} \gamma_{sw}\cdot \epsilon$. By Lemma \ref{L:LimitsofOPSGITV}, the set $N^\circ$ is the support of the fan $\Sigma$ and the equivalence classes are the relative interiors of cones in $\Sigma$. A cone $\sigma \in \Sigma$ is the closure of the set $u \in N^\circ$ with $\lim_{s \to \infty}\gamma_{sv}\cdot \epsilon = x_\sigma$. Since these limits commute with the action of $T_N$, we may replace $\epsilon$ in the definition of $\sim$ by any point $y$ in the dense orbit of $X_\Sigma$. Hence, a cone $\sigma \in \Sigma$  is the closure of the set $u \in N^\circ$ such that for any $y$ in the dense orbit of $X_\Sigma$, $\lim_{s \to \infty} \gamma_{sv}\cdot y \in W_\sigma$. We summarize this discussion.

\begin{corollary}
	The fan $\Sigma \subset N$ may be recovered from the irrational toric variety $X_\Sigma$ using limits under translation by one parameter subgroups $\gamma_{sv}$ of elements $y$ in the dense orbit. 
\end{corollary}

\section{Compact Irrational Toric Varieties}

In this section we show a similar version of Theorem \ref{T:CompactTV}. We start by reviewing some notations and results from \cite{PSV}.

\begin{definition}\label{D:minimumfaceofboundedness}
	Let $\sigma$ be a cone in $N$, $\{v_i\}$ be a sequence in $\sigma$, and $\tau$ be a face of $\sigma$. The sequence $\{v_i\}$ is called \demph{$\tau$-bounded} if for every linear function $\psi$ vanishing on $\tau$, the sequence $\{\psi(v_i)\}$ is bounded in $\RR$.  The \demph{minimum face of boundness} of $\{v_i\}$ is the smallest face $\tau$ of $\sigma$ such that $\{v_i\}$ is $\tau$-bounded.
\end{definition}

\begin{example}[{\cite[Example 1.2]{PSV}}]
	Let $N = \RR^2$ and $\sigma := \cone\{(-1,-1), (0,-1) \}$. The cone $\sigma$ has two one dimensional faces, $\tau := \cone\{(-1,-1) \}$ and $\rho := \cone \{(0,-1) \}$. Consider the two sequences $\{v_i\}$ and $\{u_i\}$ in $\sigma$, defined for $i\geq 1$ by
	\[
	v_i := \left( -i-\frac{1}{-i}, -i-1 \right) \text{\qquad and \qquad}  u_i := \left(-i +\sqrt{i}, -i\right).
	\]  
	We display the first few terms of the two sequences in Figure \ref{F:MinimumFaceofB}.
	
	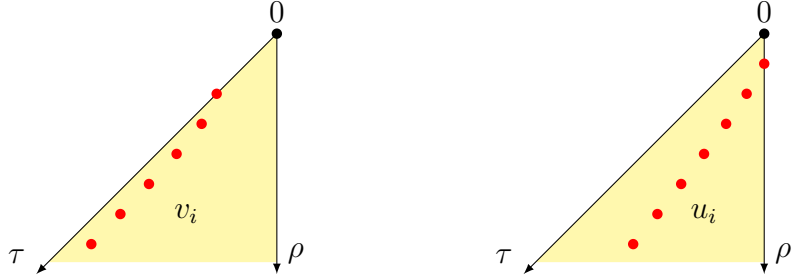
\begin{figure}[!ht]
		\centering
		\begin{tikzpicture}[line cap=round,line join=round,>=latex,scale=.80]
		\fill[line width=0.pt,fill=yellow!40!white] (0,0) -- (0,-3.8)  -- (-3.8,-3.8) -- cycle;
		\filldraw[very thick] (0,0) circle (.06cm) node [above] {$0$};
		\draw[->] (0,0) -- (-4,-4) node[anchor=south east]{$\tau$}; 
		\draw[->] (0,0) -- (0,-4) node[anchor=south west]{$\rho$};

		\filldraw[very thick,color=red] (-1,-1) circle (.06cm) ;
		\filldraw[very thick,color=red] (-1.25,-1.5) circle (.06cm);
		\filldraw[very thick,color=red] (-1.666,-2) circle (.06cm);
		\filldraw[very thick,color=red] (-2.125,-2.5) circle (.06cm);
		\filldraw[very thick,color=red] (-2.6,-3) circle (.06cm);
		\filldraw[very thick,color=red] (-3.08333,-3.5) circle (.06cm);
		\node at (-1.5,-3)  {$v_i$};
		
		\end{tikzpicture} \hspace{2cm}
		\begin{tikzpicture}[line cap=round,line join=round,>=latex,scale=.80]
		\fill[line width=0.pt,fill=yellow!40!white] (0,0) -- (0,-3.8)  -- (-3.8,-3.8) -- cycle;
		\filldraw[very thick] (0,0) circle (.06cm) node [above] {$0$};
		\draw[->] (0,0) -- (-4,-4) node[anchor=south east]{$\tau$}; 
		\draw[->] (0,0) -- (0,-4) node[anchor=south west]{$\rho$}; 
		
		\filldraw[very thick,color=red] (0,-.5) circle (.06cm) ;
		\filldraw[very thick,color=red] (-.292,-1) circle (.06cm);
		\filldraw[very thick,color=red] (-.6334,-1.5) circle (.06cm);
		\filldraw[very thick,color=red] (-1,-2) circle (.06cm);
		\filldraw[very thick,color=red] (-1.38,-2.5) circle (.06cm);
		\filldraw[very thick,color=red] (-1.775,-3) circle (.06cm);
		\filldraw[very thick,color=red] (-2.1775,-3.5) circle (.06cm);
		\node at (-1,-3)  {$u_i$};
		\end{tikzpicture}
		
		\caption{Two sequences in $\sigma$.}
		\label{F:MinimumFaceofB}
	\end{figure}
	Neither sequence is $0$-bounded or $\rho$-bounded and both are $\sigma$-bounded.  Only $\{v_i\}$ is $\tau$-bounded. Note that both sequences have the same asymptotic direction (along $\sigma$),
	\[
	\lim_{i\to \infty} \frac{v_i}{|v_i|} = \lim_{i\to \infty} \frac{u_i}{|u_i|} =(-1,-1).
	\]
\end{example}

The following lemma justifies the Definition \ref{D:minimumfaceofboundedness} showing that there exist a unique minimal face of boundedness. 

\begin{lemma}[{\cite[Lemma 1.3]{PSV}}]\label{L:MinimumFaceofB}
	Let $\{v_i\} \subset \sigma$ be a sequence in a cone $\sigma$. Then there is a face $\tau$ of $\sigma$ and a subsequence $\{u_i\}$ of $\{v_i\}$ such that $\tau$ is the minimum face of boundedness of any subsequence of $\{u_i\}$. 
\end{lemma}

\begin{theorem}\label{T:CompactnessITV}
	Let $\Sigma$ be a fan in N. The irrational toric variety $X_\Sigma$ is compact if and only if the fan $\Sigma$ is complete.
\end{theorem}

\begin{proof}
	Suppose $X_\Sigma$ is compact. Then for every $v \in N$, the family $\gamma_{sv} \cdot \epsilon$ has a limit in $X_\sigma$ as $s \to \infty.$ By Lemma \ref{L:LimitsofOPSGITV}, there is a cone $\sigma \in \Sigma$ such that $v \in \sigma$. Hence $\Sigma$ is complete. 
	
	Now assume $\Sigma$ is complete and let $\{y_i\}$ be a sequence in $X_\Sigma$. We will show that $\{y_i\}$ has a convergent subsequence. By Theorem \ref{Th:ITV_structure}, since $X_\Sigma$ is a union of finitely many orbits, there exist an orbit $W_\sigma$ whose intersection with $\{y_i\}$ is infinite. Replacing $\{y_i\}$ by a subsequence, we may assume $\{y_i\} \subset W_\sigma$. By Corollary \ref{C:StarITV}$, \overline{W_\sigma} = X_{\text{star}(\sigma)}$. Since $\Sigma$ is complete, then by Lemma \ref{L:StarStronglyConvex}, $\text{star}(\sigma)$ is also complete. Replacing $X_\Sigma$ by $X_{\text{star}(\sigma)}$, we may assume $\{y_i\}$ lies in the dense orbit of $X_\Sigma$. 
	
	The dense orbit of $X_\Sigma$ is parametrized by $N$ under the map $ v \mapsto \gamma_v \cdot \epsilon$. For each $i \in N$, choose a point $v_i \in N$ such that $y_i = \gamma_{v_i}\cdot \epsilon$. This gives a sequence $\{v_i\} \subset N$. Since $\Sigma$ is complete, $N$ is the disjoint union of the relative interiors $\Relint(\sigma)$ of cones $\sigma \in \Sigma$. Hence, there exists a cone $\sigma$ so that $\Relint(\sigma)  \cap \{v_i\}$ is infinite. Replacing $\{v_i\}$ by its intersection with $\Relint(\sigma)$, we may assume $\{v_i\} \subset \Relint(\sigma)$. Let $\tau$ be the minimum face of boundness of $\{v_i\}$. 
	Replace $\{v_i\}$ by a subsequence whose image in $N / \langle \tau \rangle$ is bounded. So there exists a closed bounded set $B \subset \sigma \subset N$ so that $\{v_i\} \subset \tau + B$. Thus there are sequences $\{u_i\} \subset \tau$ and $\{\overline{v_i}\} \subset B$ with $v_i = u_i + \overline{v_i}$ for all $i \in \NN$. Since $B$ is bounded, $\{\overline{v_i}\}$ has an accumulation point in $v \in B$. Passing to a subsequence, we may assume $\lim_{i \to \infty} \overline{v_i} = v$. 
	
	Replacing all sequences by their corresponding subsequences, we claim that 
	\begin{equation*}\label{Eq:CompletenessEq}
	\lim_{i \to \infty} y_i = \gamma_v \cdot x_\tau \in W_\tau,
	\end{equation*} 
	which will complete the proof. 
	
	Consider the sequence $\{y_i\}$ as a subset of $V_\sigma = \Hom_c(\tau^\vee, \RR_\geq)$. For $u \in \tau^\vee$, the proof of Lemma \ref{L:LimitsofOPSGITV} shows that $y_i(u) = \exp(\langle -u, v_i\rangle )$, as $y_i = \gamma_{v_i}\cdot \epsilon.$ Since $v_i = u_i + \overline{v_i}$, 
	\[
	y_i(u) = \exp (\langle -u , v_i\rangle ) = \exp(\langle -u , u_i \rangle) \exp( \langle-u , \overline{v_i} \rangle ). 
	\]
	If $u \in \tau^\perp$, then $\langle u , u_i \rangle = 0$, so that 
	\[
	\lim_{i \to \infty} y_i(u) = \lim_{i \to \infty}\exp (\langle -u , \overline{v_i}\rangle ) = \exp(\langle -u, v\rangle ). 
	\]
	If $u \in \tau^\vee \setminus \tau^\perp$, then $u$ exposes a proper face of $\tau$, and the minimality of $\tau$ implies that $\langle u , v_i \rangle$ has no bounded subsequence. But then $\langle u , u_i \rangle $ has no bounded subsequence. Since $\langle u,u_i \rangle \geq 0$, we conclude that $\lim_{i \to \infty} \langle u,u_i \rangle = \infty.$ Thus
	\[
	\lim_{i \to \infty} y_i(u) = \lim_{i \to \infty} \exp(\langle -u, v_i\rangle )=0.
	\]
	Hence $\lim_{i \to \infty} y_i = \gamma_v \cdot x_\tau \in W_\tau$ as desired. 
\end{proof}

\section{Projective Irrational Toric Varieties}\label{S:ProjectiveITV}
For irrational toric varieties, the analog of projective space is the standard simplex
\[
\simplex^n :=\ \bigl\{ (x_0,x_1,\dotsc,x_n)\in\RR^{n+1}_\geq \mid \sum x_i = 1\bigr\}\,.
\]
The simplex $\simplex^n$ can be identified  with the space of rays in $\RR^{n+1}_\geq$. A point
$x\in\simplex^n$ corresponds to the ray $\RR_\geq\cdot x\subset\RR^{n+1}_\geq$ and a ray $r$ corresponds to its intersection with $\simplex^n$. 
It is convenient to write $\simplex^\calA\subset\RR^\calA_\geq$, where $\calA\subset M$ is a finite set of exponents.

\begin{example}
	The simplex $\simplex^n$ has the structure of an irrational toric variety associated to a fan. Let $[n]= \{0,1,\ldots,n\}$ and $e_0 ,\ldots, e_n$ be the standard basis for $\RR^{[n]} \simeq \RR^{n+1}$. Then the standard simplex $\simplex^{[n]}$ is the convex hull of $e_0, \ldots, e_n$. Define a fan $\Sigma_{[n]}\subset \RR^{[n]}$ with one cone $\sigma_I$ for each proper subset $I \subsetneq [n]$ defined by
	\[
	\sigma_I := \cone\{e_i \mid i \in I\} + \RR \bbI\ ,
	\]
	where $\bbI=e_0+ \cdots + e_n$. Note that $\sigma_\emptyset = \RR \bbI$. We consider the irrational toric variety  $X_{\Sigma_{[n]}}$.
	
	The hyperplane $\{u \in \RR^{[n]} \mid \langle u, \bbI \rangle =0 \}$ is spanned by the differences $e_i -e_j$ for $i,j \in [n]$ and contains all dual cones $\sigma_I^\vee$. Set $V_I = \Hom_c(\sigma_I^\vee, \RR_\geq)$. For $I \subsetneq [n]$, $j \notin I$, and $\varphi \in V_I$, set 
	\[
	\psi_i(\varphi) = \RR_\geq \Bigl(\ \sum_{i \in [n]} \varphi(e_i - e_j) e_i\Bigr) \cap \simplex^{[n]},
	\]
	the intersection of the ray through $\sum_i \varphi(e_i -e_j)e_i$ with the simplex $\simplex^{[n]}$. This injective map does not depend on the choice of $j \notin I$ and defines a map $\psi_I \colon V_I \to \simplex^{[n]}$. These maps $\psi$ are compatible with the gluing in that if $y \in X_{\Sigma_{[n]}}$ lies in two affine patches $V_I$ and $V_J$, then $\psi_I(y) = \psi_J(y)$. Thus these maps induce a homeomorphism $\Psi \colon X_{\Sigma_{[n]}}\xrightarrow{\sim}\simplex^{[n]}$.
	
	The quotient map $\RR^{[n]}_\geq{\setminus}\{0\}\twoheadrightarrow\simplex^{[n]}$ may be understood in terms of
	Theorem~\ref{Th:mapsOfFans}. 
	Let $\Sigma'_{[n]}\subset\RR^{[n]}$ be the fan consisting of the boundary of the nonnegative orthant. Its cones are 
	\[
	\sigma'_I\ :=\ \cone\{ e_i\mid i\in I\}\,,
	\]
	for all proper subsets $I\subsetneq[n]$.
	The irrational toric variety $X_{\Sigma'_{[n]}}$ is $\RR^{n+1}_\geq{\setminus}\{0\}$, as the orbit consisting of the origin in
	$\RR^{n+1}_\geq$ corresponds to the omitted full-dimensional cone.
	As $\sigma'_I\subset \sigma_I$, these inclusions induce a map of fans $\Psi\colon\Sigma'_{[n]}\to\Sigma_{[n]}$, and thus
	a functorial map of toric varieties $X_{\Sigma'_{[n]}}\to X_{\Sigma_{[n]}}$, which is the 
	quotient map  $\RR^{[n]}_\geq{\setminus}\{0\}\twoheadrightarrow\simplex^{[n]}$.\hfill$\square$
\end{example}

Suppose that $\calA \subset M$ lies on an affine hyperplane. Hence, there exists some $v \in N$ and $0 \neq  r \in \RR$ with $\langle a ,v \rangle =r$ for all $a \in \calA$. Thus for $t \in T_N$ and $s\in \RR$, we have
\[
\varphi_\calA(\gamma_sv \cdot t) = (\gamma_{sv}(a)t^a \mid a \in \calA) = \exp(-sr) \varphi_\calA(t) \ ,
\]
as $\gamma_{sv}(a) = \exp(\langle -sa,v \rangle) =\exp(-sr)$ for $a \in \calA$ and $s\in\RR$.  Consequently, the affine irrational toric variety $X_\calA \subset \RR_\geq^\calA$ is a union of rays.
\begin{definition}
	Let $\calA \subset M$ be a finite set that lies on an affine hyperplane. Then the \demph{projective irrational toric variety} $Z_\calA$ is the intersection $X_\calA \cap \simplex^\calA$, equivalently, the quotient $(X_\calA \setminus \{0\}) / \RR_>$ under the scalar multiplication.
\end{definition}

The projective irrational toric variety $Z_\calA$ has an action of $T_N$ with a dense orbit, as the action of $T_N$ on $\RR_\geq^\calA$ gives an action on rays and hence on $\simplex^\calA$, which restricts to an action on $Z_\calA$. 

The restriction of the tautological map $\pi_\calA\colon\RR^\calA_\geq\to\cone(\calA)$ to the simplex $\simplex^\calA$ is
the canonical parametrization of the convex hull of $\calA$,
\[
\simplex^\calA\ni u\ \longmapsto\ \sum_a u_a a\ \in\ \conv(\calA)\,.
\]
By Birch's Theorem (Proposition~\ref{P:Birch}), restricting to the projective toric variety $Z_\calA$ gives a homeomorphism
$\pi_\calA\colon Z_\calA\xrightarrow{\sim}\conv(\calA)$, called the \demph{algebraic moment map}~\cite{vilnius}.
This isomorphism is also essentially proven by Krasauskas~\cite{Krasauskas}.

\begin{theorem}\label{T:ProjectiveITV}
	Suppose that $P \subset M$ is a polytope lying on an affine hyperplane with normal fan $\Sigma$. For any $\calA \subset P$ with $\conv(\calA) = P$, there is an injective map of irrational toric varieties $\Psi_\calA \colon X_\Sigma \to \simplex^\calA$ whose image is the projective irrational toric variety $Z_\calA$. The map composed with the algebraic moment map $\pi_\calA$ is a homeomorphism $X_\Sigma\xrightarrow{\sim}P$.
\end{theorem}

\begin{proof}
	Let $\calA \subset P$ be a polytope such that $\conv(\calA) = P$. Recall that by Definition \ref{D:NormalPolytope}, the normal polytope $\Sigma$ is the collection of cones of the form
	\[
	\sigma_\calF = \{v \in N \mid \langle f,v \rangle \leq \langle a , v\rangle \text{ for all } f \in \calF \text{ and } a \in \calA  \},
	\]
	where $\calF$ is a face of $\calA$. The lineality space of $\Sigma$ is spanned by $v \in N$ such that $\langle a,v \rangle = \langle b , v \rangle$ for any $a,b \in \calA$.
	
	For a subset $\calB \subset \calA$ and any $u \in M$, we write $\calB - u := \{b-u \mid b \in \calB \}$. Duals of cones $\sigma_\calF$ lie in the subspace $L$ of $M$ spanned by the differences $\{b-a \mid a,b \in \calA \},$ equivalently by $\calA-a$ for any $a\in \calA$. For a face $\calF$ of $\calA$, the dual cone of $\sigma_\calF$ is 
	\[
	\sigma_\calF^\vee = \cone(\calA -f) +\RR(\calF -f)
	\]  
	for any $f \in \calF$. Choosing another $f' \in \calF$ translates the points $\calA-f$ along the lineality space $\RR(\calF-f)$. Figure \ref{F:ProjectiveITV} shows an example of $\sigma_\calF^\vee$. The affine irrational toric variety $V_\calF$ corresponding to a face $\calF$ of $\calA$ is $\Hom_c(\sigma^\vee_\calF, \RR_\geq)$.
	\begin{figure}[!ht]
		\centering
		\hspace{1cm}
		\begin{tikzpicture}[line cap=round,line join=round,>=latex,scale=1]
		\fill[,fill=yellow!20!white] (0,4) -- (0,5)  --(7,5)-- (7,0)--(5,0) -- cycle;
		\filldraw[dashed,blue,fill=yellow!80!white] (1.25,3) -- (1,4)  --(2.25,4.5)-- (5.25,4.2)--(6,0.75)--(3.75,1) -- cycle;
		\filldraw[thick, blue] (3.75,1) circle (.06cm);
		\filldraw[thick, blue] (2.5,2) circle (.06cm);
		\filldraw[thick, blue] (1.25,3) circle (.06cm);
		\draw[-,thick,blue] (0,4) -- (5,0);
		
		\filldraw[thick, blue] (1,4) circle (.06cm);
		\filldraw[thick, blue] (2.25,4.5) circle (.06cm);
		\filldraw[thick, blue] (3.75,4.35) circle (.06cm);
		\filldraw[thick, blue] (5.25,4.2) circle (.06cm);
		\filldraw[thick, blue] (6,0.75) circle (.06cm);
		\filldraw[thick, blue] (5.625,2.475) circle (.06cm);
		
		\filldraw[thick, blue] (2,3.5) circle (.06cm);
		\filldraw[thick, blue] (3.5,3.7) circle (.06cm);
		\filldraw[thick, blue] (3.8,2) circle (.06cm);
		\filldraw[thick, blue] (4.8,2.5) circle (.06cm);		
		\draw[<-] (4.8,.35) -- (7.5,0.35) node [right] {$\RR(\calF-f)$};
		
		\node at (3.75,.5)  {$0$}; 
		\node at (0.75,.5)  {$\calF-f$};
		\node at (3.5,3)  {$\calA-f$};
		\node at (6.25,3.5)  {$\sigma_\calF^\vee$};
		
		\draw[->] (1.65,0.5) -- (3.5,.85);
		\draw[->] (1,.85) -- (2.25,1.85);
		\draw[->] (.65,1) -- (1.2,2.75);
		
		\end{tikzpicture}
		\caption{A dual cone $\sigma^\vee_\calF$.}
		\label{F:ProjectiveITV}
	\end{figure}
	
	Let $f \in \calF$ and consider the map 
	\[
	\psi_f \colon V_\calF \  \ni \ \psi \longmapsto  (\varphi(a-f) \mid a \in \calA) \ \in \ \RR_\geq^\calA.
	\]
	Since for $f, f' \in \calF$, $a \in \calA$, and $\varphi \in V_\calF$, we have $\varphi(a-f)= \varphi(f'-f)\varphi(a-f')$, it follows that $\psi_f(\varphi) = \varphi(f'-f)\psi_{f'}(\varphi)$. Hence, the two points $\psi_f(\varphi)$ and $\psi_{f'}(\varphi)$ lie along the same ray in $\RR_\geq^\calA$. So the map $V_\calF \to \simplex^\calA$ defined by
	\begin{equation}\label{M:ProjectiveITV}
	V_\calF \ \ni \ \varphi \longmapsto \left(\RR_\geq \cdot \psi_f(\varphi)\right) \cap \simplex^\calA \ \in \ \simplex^\calA \ ,
	\end{equation}
	is independent of the choice of $f \in \calF$. Write $\psi_\calF$ for the map \eqref{M:ProjectiveITV}, which is a continuous injection from $V_\calF$ into $\simplex^\calA$. 
	
	Suppose that $\calF, \calG$ are both faces of $\calA$ such that $\calF$ is a face of $\calG$. Then $V_\calG \subset  V_\calF$, and for $\varphi \in V_\calG$, we have $\psi_\calF(\varphi) = \psi_\calG(\varphi)$, as both maps are computed using $\psi_f$ for $f \in \calF \subset \calG$. Thus the maps $\psi_\calF$ for $\calF$ a face of $\calA$ are compatible with the gluing of the $V_\calF$ to form $X_\Sigma$, and so they induce a continuous map $\Psi \colon X_\Sigma \to \simplex^\calA$. The map $\Psi_\calA$ is an injection. This is because if $\calF,\calG$ are faces of $\calA$ that are not faces of each other, then the support of $\varphi \in V_\calF \setminus V_\calG$ contains $\calF$ and is disjoint from $\calG \setminus \calF$. Hence $\Psi_\calA$ is injective on the union $V_\calF \cup V_\calG$. 
	
	We claim that $\Psi_\calA(X_\Sigma) = Z_\calA$. Since both are complete, it suffices to show that both contain the same dense subset. Let $t=\gamma_v \in T_N$ with $v \in N$. Since for $u \in M$, $t^u = \exp ( -\langle u,v\rangle)$, we have $\varphi_\calA(t) = (\exp(-\langle a,v \rangle ) \mid a \in \calA)$. This lies on a ray in $\RR_\geq^\calA$ that meets $\simplex^\calA$ in the point 
	\begin{equation} \label{Eq:ProjectiveITV}
	\left(\RR_\geq \cdot (\exp(-\langle a,v \rangle) \mid a \in \calA)\right) \ \cap \ \simplex^\calA  \ \in \ Z_\calA .
	\end{equation}
	
	The corresponding point $\gamma_v\cdot \epsilon$ of $X_\sigma$ lies in $V_\calA = \Hom_c(L,\RR_>)$, where $L= \RR(\calA-b)$ for any $b \in \calA$. Its image $\Psi_\calA(\gamma_v \cdot \epsilon)$ is 
	\[
	\left(\RR_\geq \cdot (\gamma_v \cdot \epsilon(a-b) \mid a \in \calA)\right) \cap \simplex^\calA = \left( \RR_\geq \cdot \exp(-\langle a,v\rangle ) \mid a \in \calA \right) \cap \simplex^\calA \ ,
	\] 
	as $\epsilon (a-b) =1$ and $\gamma_v(a-b) = \exp ( - \langle a,v \rangle ) \exp (\langle b,v \rangle)$, so that the two rays are equal. Comparing this to \eqref{Eq:ProjectiveITV} completes the proof.
\end{proof}
\pagebreak{}

\chapter{HAUSDORFF LIMITS AND THE SECONDARY POLYTOPE} \label{CH:Hausdorff}

We study the space of Hausdorff limits of translates of irrational toric varieties. This work is motivated by \cite{PSV}, which was motivated by \cite{Graciun}. We start by reviewing some background on subdivisions, the secondary fan and the secondary polytope of a finite point configuration $\calA \subset M$. We then recall some notations and results studied in \cite{PSV} about Hausdorff limits of torus translates of $Z_\calA$. In Theorem \ref{Th:HausdorffLimits}, we construct a homeomorphism between the space of Hausdorff limits and $X_{\Sigma(\calA)}$.

\section{Subdivisions and Secondary Polytopes}
Let $\calA \subset M$ be a finite set of points which affinely spans $M$.  

\begin{definition}
	A \demph{subdivision} $\calS$ of $\calA$ is a collection of subsets $\calF$ of $\calA$ satisfying the following properties.
	\begin{enumerate}
		\item If $\calF, \calG \in \calS$, then $\mathcal{H} : =\calF \cap \calG$ is also in $\calS$, and $\conv(\mathcal{H})$ is a face of both $\conv(\calF)$ and  $\conv(\calG)$.
		\item ${\displaystyle \bigcup_{\calF \in \calS} } \conv(\calF) =\conv(\calA)$.
	\end{enumerate}
	The elements of $\calS$ are called \demph{faces} of $\calS$. The set of convex hulls $\{\conv(\calF) \mid \calF \in \calS \}$ form a complex, denoted by $\Pi_\calS$, which covers $\conv(\calA)$.
	
\end{definition}

\begin{definition}
	A \demph{triangulation} of $\calA$ is a subdivision  of $\calS$, in which every face $\conv(\calF)$ of $\Pi_\calS$ is a simplex with vertices $\calF$. 
\end{definition}

\begin{example}
	Consider the planar configuration $\calA$ in Figure \ref{F:PolyhedralSubdivisions}.  $\calS$ and $\calT$ are both subdivisions of $\calA$, and $\calT$ is a triangulation of $\calA$. On the other hand, $\calS'$ is not a subdivision of $\calA$ as $$\conv ( \{ a,d,e\}) \cap \conv(\{b,c,d\}) = \conv(\{d,e\}),$$ which is not a face of $\conv(\{b,c,d\}).$
	\begin{figure}[!ht]
		\centering
		
		\begin{tikzpicture}[line cap=round,line join=round,>=latex,scale=.5]
		\filldraw[very thick] (0,0) circle (.08cm) node[anchor =north ,color=black]{$a$};
		\filldraw[very thick] (0,5) circle (.08cm) node[anchor =south ,color=black]{$d$};
		\filldraw[very thick] (5,0) circle (.08cm) node[anchor =north ,color=black]{$b$};
		\filldraw[very thick] (5,5) circle (.08cm)node[anchor =south,color=black]{$c$};
		\filldraw[very thick] (3.5,1.5) circle (.08cm)node[anchor =south east ,color=black]{$e$};
		\node at (2.5,7)  {$\calA$};
		\end{tikzpicture} \hspace{1cm}
		\begin{tikzpicture}[line cap=round,line join=round,>=latex,scale=.5]
		\filldraw[very thick] (0,0) circle (.08cm) node[anchor =north ,color=black]{$a$};
		\filldraw[very thick] (0,5) circle (.08cm) node[anchor =south ,color=black]{$d$};
		\filldraw[very thick] (5,0) circle (.08cm) node[anchor =north ,color=black]{$b$};
		\filldraw[very thick] (5,5) circle (.08cm)node[anchor =south,color=black]{$c$};
		\filldraw[very thick] (3.5,1.5) circle (.08cm)node[anchor =south east ,color=black]{$e$};
		\draw[-] (0,0) -- (0,5); 
		\draw[-] (0,0) -- (5,0); 
		\draw[-] (0,0) -- (3.5,1.5);
		\draw[-] (3.5,1.5) -- (5,5);
		\draw[-] (3.5,1.5) -- (5,0);
		\draw[-] (5,5) -- (0,5);
		\draw[-] (5,5) -- (5,0);
		\node at (2.5,7)  {$\calS$};
		\end{tikzpicture}\hspace{1cm}
		\begin{tikzpicture}[line cap=round,line join=round,>=latex,scale=.5]
		\filldraw[very thick] (0,0) circle (.08cm) node[anchor =north ,color=black]{$a$};
		\filldraw[very thick] (0,5) circle (.08cm) node[anchor =south ,color=black]{$d$};
		\filldraw[very thick] (5,0) circle (.08cm) node[anchor =north ,color=black]{$b$};
		\filldraw[very thick] (5,5) circle (.08cm)node[anchor =south,color=black]{$c$};
		\filldraw[very thick] (3.5,1.5) circle (.08cm);
		\node at (3,1.57){$e$};
		\draw[-] (0,0) -- (0,5); 
		\draw[-] (0,0) -- (5,0); 
		\draw[-] (3.5,1.5) -- (5,5);
		\draw[-] (3.5,1.5) -- (0,0);
		\draw[-] (0,5) -- (5,0);
		\draw[-] (5,5) -- (0,5);
		\draw[-] (5,5) -- (5,0);
		\node at (2.5,7)  {$\calT$};
		\end{tikzpicture}
		\hspace{1cm}
		\begin{tikzpicture}[line cap=round,line join=round,>=latex,scale=.5]
		\filldraw[very thick] (0,0) circle (.08cm) node[anchor =north ,color=black]{$a$};
		\filldraw[very thick] (0,5) circle (.08cm) node[anchor =south ,color=black]{$d$};
		\filldraw[very thick] (5,0) circle (.08cm) node[anchor =north ,color=black]{$b$};
		\filldraw[very thick] (5,5) circle (.08cm)node[anchor =south,color=black]{$c$};
		\filldraw[very thick] (3.5,1.5) circle (.08cm)node[above ,color=black]{$e$};
		\draw[-] (0,0) -- (0,5); 
		\draw[-] (0,0) -- (5,0); 
		\draw[-] (3.5,1.5) -- (0,0);
		\draw[-] (0,5) -- (5,0);
		\draw[-] (5,5) -- (0,5);
		\draw[-] (5,5) -- (5,0);
		\node at (2.5,7)  {$\calS'$};
		\end{tikzpicture}
		
		\caption{Subdivisions and triangulations.}
		\label{F:PolyhedralSubdivisions}
	\end{figure}
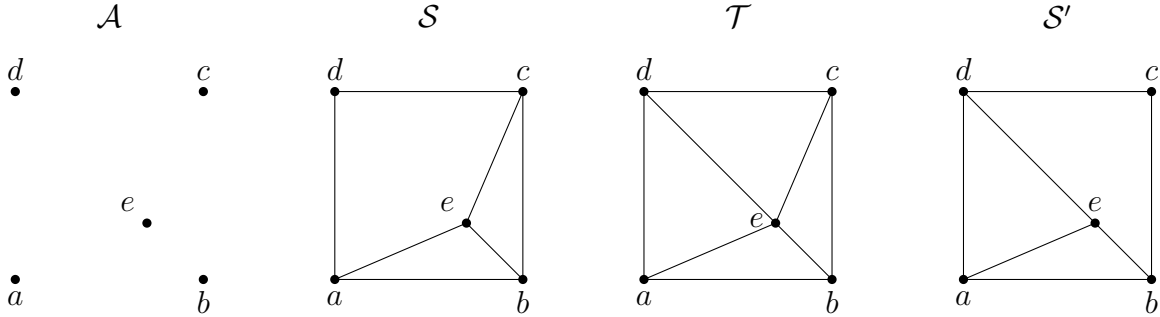
\end{example}

\begin{definition}\label{D:SubdivisionRefined}
	A subdivision $\calS$ of $\calA$ is \demph{refined} by $\calS'$, denoted by $\calS \prec \calS'$, if for every face $\calF'$ of $\calS'$, there is a face $\calF$ of $\calS$ with $\calF' \subseteq \calF$. 
\end{definition}
Definition \ref{D:SubdivisionRefined} makes the set of all subdivisions of $\calA$ into a poset. Triangulations are precisely minimal elements of this poset and its maximal element is the subdivision consisting $\calA$ itself.

Elements $\lambda \in \RR^\calA$ induce subdivisions of $\calA$ in the following way. Let $P_\lambda$ be the convex hull of the graph of $\lambda$, defined by
\[
P_\lambda := \conv  \left( \{ (a,\lambda_a) \mid a \in \calA \}\right).
\]
Its \demph{lower faces} are those having inward-pointing normal vector with last coordinate positive. For a lower face $F$ of $P_\lambda$, let 
\[
\calF(F) := \{ a \in \calA \mid (a,\lambda_a) \in F\}.
\]
Also, let 
\[
S_\lambda := \bigcup_{F} \calF(F),
\]
where $F$ ranges over the lower faces if $P_\lambda$. The union $S_\lambda$ is a subdivision of $\calA$ \cite[Lemma 2.3.11]{Triangulations}.

\begin{definition}
	A subdivision $\calS$ is \demph{regular} if $\calS = S_\lambda$ for some $\lambda \in \RR^\calA$. 
\end{definition}

\begin{example}\label{Ex:RegularSubdivision}
	Let $\calA = \RR^2$ be the vertices of two concentric triangles given in Figure \ref{F:RegularSubdivisionSet}. We will examine a regular and a non-regular subdivision of $\calA$.
	
	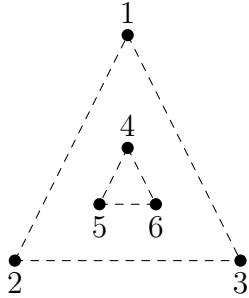
\begin{figure}[!ht]
		\centering
		\begin{tikzpicture}[line cap=round,line join=round,>=latex,scale=.75]
		\filldraw[very thick] (4,0) circle (.08cm) node[anchor =north ,color=black]{$3$};
		\filldraw[very thick] (0,0) circle (.08cm) node[below ,color=black]{$2$};
		\filldraw[very thick] (2,4) circle (.08cm) node[anchor =south ,color=black]{$1$};
		\filldraw[very thick] (2,2) circle (.08cm) node[above ,color=black]{$4$};
		\filldraw[very thick] (1.5,1) circle (.08cm) node[below ,color=black]{$5$};
		\filldraw[very thick] (2.5,1) circle (.08cm) node[below ,color=black]{$6$};
		\draw[-,dashed] (0,0) -- (4,0); 
		\draw[-,dashed] (0,0) -- (2,4); 
		\draw[-,dashed] (2,4) -- (4,0); 
		\draw[-,dashed] (2,2) -- (1.5,1); 
		\draw[-,dashed] (2,2) -- (2.5,1); 
		\draw[-,dashed] (1.5,1) -- (2.5,1); 
		\end{tikzpicture}
		
		\caption{Two concentric triangles.}
		\label{F:RegularSubdivisionSet}
	\end{figure}

	First consider $\calS_1$ in Figure \ref{F:RegularSubdivision}. Let $\lambda= (3,2,1,0,0,0)$. Then $\lambda$ induces $\calS_1$. 
	
	\begin{figure}[!ht]
		\centering
		
		\begin{tikzpicture}[line cap=round,line join=round,>=latex,scale=.75]
		\filldraw[very thick] (4,0) circle (.08cm) node[anchor =north ,color=black]{$3$};
		\filldraw[very thick] (0,0) circle (.08cm) node[below ,color=black]{$2$};
		\filldraw[very thick] (2,4) circle (.08cm) node[anchor =south ,color=black]{$1$};
		\filldraw[very thick] (2,2) circle (.08cm) node[below ,color=black]{$4$};
		\filldraw[very thick] (1.5,1) circle (.08cm) node[anchor=south east ,color=black]{$5$};
		\filldraw[very thick] (2.5,1) circle (.08cm) node[below ,color=black]{$6$};
		\draw[-] (0,0) -- (4,0); 
		\draw[-] (0,0) -- (2,4); 
		\draw[-] (2,4) -- (4,0); 
		\draw[-] (2,2) -- (1.5,1); 
		\draw[-] (2,2) -- (2.5,1); 
		\draw[-] (1.5,1) -- (2.5,1); 
		\draw[-] (2,4) -- (2.5,1); 
		\draw[-] (2,4) -- (1.5,1); 
		\draw[-] (0,0) -- (2.5,1); 
		\draw[-] (0,0) -- (1.5,1); 
		\draw[-] (4,0) -- (2.5,1); 
		\draw[-] (2,4) -- (2,2); 
		\node at (2,-2)  {$\calS_1$};
		\end{tikzpicture}\hspace{3cm}
		\begin{tikzpicture}[line cap=round,line join=round,>=latex,scale=0.75]
		\filldraw[very thick] (4,0) circle (.08cm) node[anchor =north ,color=black]{$3$};
		\filldraw[very thick] (0,0) circle (.08cm) node[below ,color=black]{$2$};
		\filldraw[very thick] (2,4) circle (.08cm) node[anchor =south ,color=black]{$1$};
		\filldraw[very thick] (2,2) circle (.08cm) node[right,color=black]{$4$};
		\filldraw[very thick] (1.5,1) circle (.08cm) node[anchor=south east ,color=black]{$5$};
		\filldraw[very thick] (2.5,1) circle (.08cm) node[below ,color=black]{$6$};
		\draw[-] (0,0) -- (4,0); 
		\draw[-] (0,0) -- (2,4); 
		\draw[-] (2,4) -- (4,0); 
		\draw[-] (2,2) -- (1.5,1); 
		\draw[-] (2,2) -- (2.5,1); 
		\draw[-] (1.5,1) -- (2.5,1); 
		\draw[-] (4,0) -- (2,2); 
		\draw[-] (2,4) -- (1.5,1); 
		\draw[-] (0,0) -- (2.5,1); 
		\draw[-] (0,0) -- (1.5,1); 
		\draw[-] (4,0) -- (2.5,1); 
		\draw[-] (2,4) -- (2,2); 
		\node at (2,-2)  {$\calS_2$};
		\end{tikzpicture}
		
		\caption{Subdivisions of $\calA$.}
		\label{F:RegularSubdivision}
	\end{figure}

	On the other hand, the subdivision $S_2$ in Figure \ref{F:RegularSubdivision} is not regular. To see this, assume $\lambda \in \RR^6$ induces $S_2$. Note that we can assume $\lambda_4=\lambda_5=\lambda_6=0$, as altering $\lambda$ by an affine function of $\calA$ produce the same subdivision of $\calA$. To get edge $\overline{15}$, we must have $\lambda_1 < \lambda_2$. The same happens for edges $\overline{26}$ and $\overline{34}$. Hence we get $\lambda_1 < \lambda_2 < \lambda_3 <\lambda_1$, which is a contradiction. Hence $\calS_2$ cannot be regular. 
\end{example}

For a regular triangulation $\calT$ of $\calA$, let $C(\calT) \subset \RR^\calA$ be the closure of the set all functions $\lambda \in \RR^\calA$ which induces $\calT$. 
The set $C(\calT)$ is a full-dimensional cone in $\RR^\calA$, and the cones $C(\calT)$ for all regular triangulations of $\calA$ together with all faces of these cones form a complete fan $\Sigma(\calA)$, called the \demph{secondary fan} of $\calA$ \cite[Proposition 1.5, Chapter 7]{GKZ}.

Let $\calT$ be a triangulation of $\calA$, where $\conv(\calA)$ is full dimensional. The \demph{characteristic function} of $\calT$ is the function $\varphi_\calT \colon \calA \to \RR$ defined by
\[
\varphi_\calT(w) = \sum \text{Vol}(\calF),
\]
where the summation is over all maximal simplices $\calF$ of $\calT$ for which $w$ is a vertex, and $\text{Vol}(\calF)$ denotes the usual Euclidean volume. 
The \demph{secondary polytope} $\calP(\calA)$ is the convex hull in $\RR^\calA$ of the vectors $\varphi_\calT$ for all the triangulations $\calT$ of $\calA$.

For any regular subdivision $\calS$ of $\calA$, let $F(\calS)$ be the convex hull of the characteristic functions $\varphi_\calT$ for all triangulations $\calT$ refining $\calS$. By \cite[Theorem 2.4, Chapter 7]{GKZ}, the faces of $\calP(\calA)$ are precisely the polytopes $F(\calS)$ for all regular subdivisions $\calS$ of $\calA$. In addition, $F(\calS) \subset F(\calS')$ if and only if $\calS$ refines $\calS'$. The normal cone to $F(\calS)$  coincides with $C(\calS)$.

\section{Hausdorff Limits of Torus Translates}

Given a compact metric space $(X,d)$, the \demph{Hausdorff distance} between closed sets $A$ and $B$ in $X$ is 
\[
d_H(A,B) := \max \left\{\sup_{a \in A} \inf_{b \in B} d(a, b) + \sup_{b \in B} \inf_{a \in A} d(a, b)\right\}.
\]
The set $\text{Closed}(X)$ of closed subsets of $X$ is a compact metric space with the Hausdorff distance \cite[p. 167]{Hausdorff}.

As in Section \ref{S:ProjectiveITV}, suppose that $\calA \subset M$ lies on an affine hyperplane. Then the irrational toric variety $X_\calA \subset \RR_\geq^\calA$ is a union of rays with associated irrational projective toric variety $Z_\calA = X_\calA \cap \simplex^\calA$. An element $\omega \in \RR^\calA_>$ of the positive torus acts on $\RR^\calA_\geq$ by translation,
\[
\omega \cdot x = w \cdot (x_a \mid a \in \calA ) = (\omega_a x_a \mid a \in \calA),
\]
for $x \in \RR_\geq^\calA$. This induces an action of $\RR_>^\calA$ on rays, and hence on $\simplex^\calA$.

For $w \in \RR^\calA$, the translate $\omega \cdot Z_\calA$ is a closed subset of $\simplex^\calA$ defined by binomials given as follows.
\begin{proposition}[{\cite[Proposition A.2]{GSZ}}] 
	A point $z \in \simplex^\calA$ lies in $\omega \cdot Z_\calA$ if and only if
	\[
	\prod_{a \in \calA} z_a^{u_a}\cdot \prod_{a \in \calA} \omega_a^{v_a} = \prod_{a \in \calA} z_a^{v_a}\cdot \prod_{a \in \calA} \omega_a^{u_a},
	\]
	for all $u,v \in \RR^\calA_\geq$ with $\calA u = \calA v$.
\end{proposition}

The association $\omega \mapsto \omega \cdot Z_\calA$ gives a continuous map $\RR^\calA_> \to \text{Closed}(\simplex^\calA)$. Let $\Delta_\calA$ be the closure of that image. This is a compact Hausdorff space equipped with an action of $\RR_>^\calA$, and it consists of all Hausdorff limits of translates of $Z_\calA$. In \cite{PSV}, the main result was a set-theoretic identification of the points in $\Delta_\calA$. We extend their work to construct a homeomorphism between $\Delta_\calA$ and the irrational toric variety associated to the secondary fan of $\calA$.

For a regular subdivision $\calS$ of $\calA$, the union of irrational toric varieties
\begin{equation}\label{Eq:ComplexesITV}
Z(\calS) := \bigcup_{\calF \in \calS} Z_\calF
\end{equation}
is called the \demph{complex of irrational toric varieties} corresponding to $\calS$. This is a complex in that if $\calF, \calF' \in \calS$ with $\emptyset \neq \calG = \calF \cap \calF'$, so that $\calG$ is also a face of $\calS$, then $Z_\calG = Z_\calF \cap Z_{\calF'}$.

An element $\omega \in \RR_>^\calA$ of the positive torus acts on the complex giving the translated complex,
\begin{equation}\label{Eq:TranslatedComplex}
Z(\calS, \omega) := \bigcup_{\calF \in \calS} w \cdot Z_\calF.
\end{equation}
\begin{lemma}[{\cite[Lemma 2.4]{PSV}}] \label{L:PSVLemma2.4}
	For $\omega,\omega' \in \RR^\calA_>$, the translated complexes $Z(\calS, \omega)$ and $Z(\calS,\omega')$ are equal if and only if $\Log(\omega) - \Log(\omega') \in \text{Aff}(\calS)$, where $\Log \colon \RR^\calA_\geq \to \RR^\calA$ is the coordinatewise logarithm map. 
\end{lemma}

The next theorem gives a set-theoretic identification between points of $\Delta_\calA$ and the translated complexes $Z(S,\omega)$ of irrational toric varieties. 

\begin{theorem}[{\cite[Theorem 3.3]{PSV}}] \label{T:PSVTheorem3.3}
	The points of $\Delta_\calA \subset \text{Closed}(\simplex^\calA)$ are exactly the translated complexes $Z(\calS, \omega)$ of irrational toric varieties for $\omega \in \RR^\calA_>$ and $\calS$ a regular subdivision of $\calA$. 
\end{theorem}

Recall that $ \epsilon \in X_{\Sigma(\calA)}$ is the distinguished point of its dense $\RR_>^\calA$-orbit. 

\begin{theorem}\label{Th:HausdorffLimits}
	For $\omega \in \RR_>^\calA$, the association $\psi: \omega \cdot Z_\calA \mapsto \omega \cdot \epsilon$ is a well defined continuous map from the set of translates of $Z_\calA$ to the dense orbit of $X_{\Sigma(\calA)}$ that extends to a $\RR^\calA_>$-equivariant homeomorphism $\Delta_\calA \xrightarrow{\sim} X_{\Sigma(\calA)}$. Composing it with an algebraic moment map $X_{\Sigma(\calA)} \to \calP(\calA)$ gives a homeomorphism between $\Delta_\calA$ and the secondary polytope $\calP(\calA)$. 
\end{theorem}

\begin{proof}
	Let $\sigma \in \Sigma(\calA)$ be a cone with corresponding regular subdivision $S_\sigma$. The orbit $W_\sigma$ in $X_{\Sigma(\calA)}$ has distinguished point $x_\sigma$. Under the map
	\begin{align*}
	\RR^\calA &\xrightarrow{\ \sim \ }\RR_>^\calA\\
	v \ & \longmapsto  \ \gamma_v,
	\end{align*}
	the orbit $W_\sigma = \RR^\calA_> \cdot x_\sigma$ is identified with $\RR_\calA / \langle \sigma \rangle$. So the stabilizer of $x_\sigma$ is the linear span $\langle \sigma \rangle$ of $\sigma$.
	
	The complex $Z(S_\sigma)= Z(S_\sigma,1)$ is the distinguished point of $\Delta_\calA$ corresponding to $\sigma$. By Lemma \ref{L:PSVLemma2.4}, the stabilizer of $Z(S_\sigma,1)$ in $\RR^\calA$ is also $\langle \sigma \rangle$. Hence, the association $\psi \colon Z(S_\sigma, \gamma_v) \mapsto \gamma_v \cdot x_\sigma$ is a well defined map. This induces a $\RR^\calA_>$-equivariant homeomorphism between the orbits in both spaces, because when both orbits are identified with $\RR^\calA / \langle \sigma \rangle$, the map $\psi$ becomes the identity map. When $S_\sigma$ has a single facet $\calA$, so that $\sigma$ is the minimal cone of $\Sigma(\calA)$, this is the map $\psi$ in the statement of the theorem.  
	
	Write $\psi$ for the bijection given by
	\begin{align*}
	\psi \colon \Delta_\calA &\longrightarrow X_{\Sigma(\calA)} \\
	Z(S_\sigma, \omega) & \longmapsto \omega \cdot x_\sigma,
	\end{align*}
	for $\omega \in \RR^\calA_>$ and $\sigma \in \Sigma(\calA)$. We show that if $\{\omega_n\} \subset \RR^\calA_>$, $\omega \in \RR_>^\calA$, and $\sigma \in \Sigma(\calA)$ are such that
	\begin{equation}\label{Eq:HausdorffLimits}
	\lim_{n \to \infty} \omega_n \cdot Z_\calA = Z(S_\sigma,\omega)
	\end{equation}
	in the space $\Delta_\calA$ of Hausdorff limits, then
	\begin{equation} \label{Eq:HausdorffLimits2}
	\lim_{n \to \infty} \omega_n \cdot \epsilon = \omega \cdot x_\sigma
	\end{equation}
	in the irrational toric variety $X_{\Sigma(\calA)}$, and vice versa. Since $\Delta_\calA$ and $X_{\Sigma(\calA)}$ are both $\RR_>^\calA$-equivariant closures of their dense orbits, this will complete the proof.
	
	Our argument that $\psi$ preserves limits of sequence follows from the proofs of Theorem \ref{T:PSVTheorem3.3} and Theorem \ref{T:CompactnessITV}. These proofs show that the sequences of translates $\omega_n \cdot Z_\calA$ and $\omega_n \cdot \epsilon$ have convergent sequences. In each, we replace $\omega_n \in \RR^\calA_>$ by $v_n \in \RR^\calA$, where $\omega_n = \gamma_{v_n}$, and then replace $\{ v_n \}$ by any subsequence $\{ v_n\} \cap \sigma$, where $\sigma$ is a cone in $\Sigma(\calA)$ that meets $\{ v_n \}$ in an infinite set. Finally, $\tau$ is the minimum face of boundedness of $\{ v_n \}$, so $\{v_n\}$ is bounded in $\RR^\calA /\langle \tau \rangle$, and then $v\in \RR^\calA$ satisfies 
	\[
	v+ \langle \tau \rangle \text{\ is an accumulation point of } \{ v_n + \langle \tau \rangle \} \text{ in } \RR^\calA / \langle \tau \rangle .
	\]
	Restricting to subsequences, we have $\lim_{n \to \infty} \omega_n \cdot \epsilon = v \cdot x_\tau$ and $\lim_{n \to \infty} \omega_n \cdot Z_\calA = Z(S_\tau,v)$.
	
	We assumed the original limits (\eqref{Eq:HausdorffLimits} and \eqref{Eq:HausdorffLimits2}) exist without restricting to subsequences. Thus, for every cone $\sigma$ with $\sigma \cap \{v_n\}$ infinite,  $\tau$ is the minimal face of boundness of $\{v_n\}$, and then $v$ is a point such that 
	\[
	\lim_{n \to \infty} v_n + \langle \tau \rangle = v + \langle \tau \rangle.
	\]
	This completes the proof, as both $\psi$ and $\psi^{-1}$ preserve limits of sequences. 
	
	The last statement is by Theorem \ref{T:ProjectiveITV}, as $\Sigma(\calA)$ is the normal fan to $P(\calA)$. 
\end{proof}
\pagebreak{}

\chapter{SUMMARY AND FUTURE WORK} \label{CH:Conclusion}

\section{Summary}
Inspired by applications in algebraic statistics and geometric modeling, we developed a theory of irrational toric varieties associated to fans in a real vector space $N$ that are not necessarily rational. This work required completely new methods without an  analog in the classical setting. This is due to the fact that irrational toric varieties  have no algebraic structure to exploit. In Theorem \ref{Th:ITV_structure} we showed that the irrational toric variety $X_\Sigma$ associated to a fan $\Sigma \subset N$ is a $T_N$-equivariant cell complex. Each cell is an orbit and corresponds to a unique cone in $\Sigma$. We also showed that the cell structure of $X_\Sigma$ and its poset of containment in closure is dual to that of the fan $\Sigma$. 
We also showed in Theorem \ref{Th:TVvsITV} that when the fan $\Sigma$ is rational, the irrational toric variety $X_\Sigma$ corresponding to $\Sigma$ is the nonnegative part of the complex toric variety $Y_\Sigma$.  

Our construction of irrational toric varieties has many pleasing parallel results with the classical theory. Theorem \ref{Th:mapsOfFans} shows that this construction is functorial, that is, maps of fans corresponds to  equivariant irrational toric morphisms, and we showed how to recover a fan $\Sigma$ from the irrational toric variety in Section \ref{S:RecoveringFans}.  Another similar result is Theorem \ref{T:CompactnessITV}, which shows that an irrational toric variety is compact if and only if the corresponding fan  is complete.
We also studied irrational toric varieties as monoids. By adjoining an absorbing element to $X_\Sigma$, it becomes a commutative topological monoid in which the inclusion of the irrational toric variety $V_\sigma$ is a monoid map for each cone $\sigma \in \Sigma$ (Theorem \ref{T:Monoids}). 
In Section \ref{S:ProjectiveITV}, we constructed projective (that may be embedded in a simplex) irrational toric varieties. In Theorem \ref{T:ProjectiveITV}, we showed that when the fan $\Sigma$ is the normal fan of a polytope, the irrational toric variety $X_\Sigma$ is projective and it is homeomorphic to that polytope. 

When $\calA$ is rational, the space of Hausdorff limits of translates of the irrational toric variety corresponding to $\calA$ is identified with the secondary polytope of the exponent vectors of its parametrizing monomials~\cite{GSZ}.
For irrational $\calA$, all Hausdorff limits were identified in~\cite{PSV} as toric degenerations, and thus were related to the secondary fan of a point configuration.
We extended that work in Theorem \ref{Th:HausdorffLimits}, where we constructed a homeomorphism between the space of Hausdorff limits and the irrational toric variety $X_{\Sigma(\calA)}$ associated to the secondary fan of $\calA$. 

\section{Future Work}

This new theory of irrational toric varieties may enable the generalization of other geometric-combinatorial aspects of classical toric varieties to irrational
polytopes and fans. For example, this will likely be able to show that fiber polytopes \cite{fiber} are naturally moduli spaces of Hausdorff limits of toric subvarieties of toric varieties. 

Irrational toric varieties admit deformations, while classical toric varieties do not. Understanding these deformations is another project. A related and more algebraic topic is to study deformations of the Wachspress varieties of Irving and Schenck \cite{Irving}, which are varieties from approximation theory that behave like polygons in the plane. These spaces of deformations will give a natural analytic or algebraic structure to spaces of polytopes, and in particular a structure on degenerations of polytopes. Understanding what these structures are and their relation to geometric combinatorics would be extremely fascinating.

Tropicalization of toric varieties were studied as skeletons of analytifications in \cite{thuillier}, and they were studied explicitly in \cite{kajiwara}. Then they were applied to the development of relations between tropical and nonarchimedean analytic geometry \cite{payne11,payne9a,rabinoff}. In particular, for a closed subscheme $X$ of a toric variety, the limit of the inverse system of tropicalizations of all toric embeddings of $X$ is homeomorphic to the analytification of $X$ \cite{payne14} . Another line of research is to define tropical irrational toric varieties by using the construction of irrational toric varieties and extend the mentioned results to tropical irrational toric varieties. 
\pagebreak{}

\let\oldbibitem\bibitem
\renewcommand{\bibitem}{\setlength{\itemsep}{0pt}\oldbibitem}
\bibliographystyle{plain}
\phantomsection
\addcontentsline{toc}{chapter}{REFERENCES}
\renewcommand{\bibname}{{\normalsize\rm REFERENCES}}
\bibliography{myReference.bib}
\end{document}